\documentclass[12pt]{article}

\usepackage{fullpage}
\usepackage{amsfonts}
\usepackage{amssymb}
\usepackage{amsmath}
\usepackage{stmaryrd} 
\usepackage{wasysym}  
\usepackage[colors]{optsys} 

\newcommand{\FunctionBigF}{F}
\newcommand{\SimplexWithoutEmptySet}{\Delta_{\IndexSet}^{\backslash \emptyset}}
\newcommand{\AllSubsetsExceptEmptyset}{\emptyset \subsetneq \IndexSubset \subset \IndexSet}
\newcommand{\FSM}{FSM}

\title{Constant Along Primal Rays Conjugacies\\
  and Generalized Convexity \\ 
  for Functions of the Support}

\author{Jean-Philippe Chancelier and 
  Michel De Lara, \\ 
  CERMICS, Ecole des Ponts, Marne-la-Vall\'ee, France}

\begin{document}

\maketitle

\begin{abstract}
  The support of a vector in~$\mathbb{R}^d$ is the set of indices
  with nonzero entries. Functions of the support have the property to
  be $0$-homogeneous and, because of that, 
  the Fenchel conjugacy fails to provide relevant analysis.
  In this paper, we define the coupling \Capra\ between~$\mathbb{R}^d$ and itself 
by dividing the classic Fenchel scalar product coupling by a
given (source) norm on~$\mathbb{R}^d$.
Our main result is that, when both the source norm and its dual norm are orthant-strictly monotonic,
any nondecreasing finite-valued function of the support mapping is \Capra-convex, that is, is 
equal to its Capra-biconjugate (generalized convexity).
We also establish that any such function is the composition of a proper convex lower semi continuous function
  on~$\RR^d$ with the normalization mapping on the unit sphere (hidden
  convexity), and that, when normalized, it admits a variational formulation, 
  which involves a family of generalized local-$\IndexSubset$-support
  dual~norms.
\end{abstract}

{{\bf Key words}: support of a vector, hidden convexity, 
  Fenchel-Moreau conjugacy, Capra conjugacy, generalized convexity, 
  orthant-strictly monotonic norms, coordinate-$\IndexSubset$ norms, 
top-$\IndexSubset$ norms, $\IndexSubset$-support norms.}


\section{Introduction}

The support of a vector in~$\RR^d$, where $d \geq 1$ is a fixed integer, 
is the set of indices in $\IndexSet=\Vset$ with nonzero entries.
We consider the support mapping that maps vectors in~$\RR^d$ to
subsets of indices in~\(  2^\IndexSet \),
and set functions that go from~\( 2^\IndexSet \) to~\( \barRR \).
By composing any set function with the support mapping, we obtain  
\emph{functions of the support mapping} (in short, \FSM).
An example of \FSM\ is the so-called \lzeropseudonorm\
(also called counting function, or cardinality function) which
counts the number of nonzero components of a vector.
The \lzeropseudonorm\ measures the sparsity of a vector,
and it is mentioned in an abundant literature in sparse optimization. 
However, because of its combinatorial nature, 
the problems of minimizing the \lzeropseudonorm\ under constraints
or of minimizing a criterion under $k$-sparsity constraint
(\lzeropseudonorm\ less than a given integer~$k$) 
are usually not tackled as such. Most of the literature in sparse
optimization studies surrogate problems where the \lzeropseudonorm\
either enters a penalization term or is replaced by a regularizing term.
We refer the reader to \cite{Nikolova:2016} that provides a brief tour of the
literature dealing with least squares minimization constrained by $k$-sparsity,
and to \cite{Hiriart-Urruty-Le:2013} for a survey of the rank
function of a matrix, that shares many properties with the
\lzeropseudonorm.

Conjugacies, and more generally dualities, are a powerful tool to tackle
suitable classes of optimization problems. For instance, the Fenchel conjugacy plays a central role in analyzing
solutions of convex problems (and beyond) \cite{Rockafellar:1974}.
Unfortunately, \FSM\ have the property to be $0$-homogeneous and, because of that, 
the Fenchel conjugacy fails to provide relevant analysis.
As an illustration, the Fenchel biconjugate 
of the characteristic function of the level sets of the \lzeropseudonorm\ 
and the Fenchel biconjugate of the \lzeropseudonorm\ both are zero.
However, the field of generalized convexity goes beyond the Fenchel conjugacy 
and convex functions, and provides conjugacies that are adapted to analyze
classes of functions such as 
increasing positive homogeneous, difference of convex, quasi-convex,
increasing and convex-along-rays.
For more details on the theory, and more examples, 
we refer the reader to the books \cite{Singer:1997,Rubinov:2000}
and to the nice introduction paper \cite{Martinez-Legaz:2005}.
To our knowledge, none of the conjugacies in the literature
is adapted to~\FSM.
In this paper, we study~\FSM\ as such (and not surrogate functions) and we 
display a class of conjugacies that 
are suitable for \FSM; thus equipped, we obtain results for nondecreasing finite-valued \FSM.

The paper is organized as follows. 
%
In Sect.~\ref{The_Capra_conjugacy_under_orthant-strict_monotonicity},
we define functions of the support mapping (\FSM), we provide background on
couplings and conjugacies, and we
introduce a source norm~$\TripleNorm{\cdot}$  on~$\RR^d$ and the induced
constant along primal rays coupling~$\CouplingCapra$ (\Capra).
Then, we state our main result: 
if both the source norm $\TripleNorm{\cdot}$ and the dual norm $\TripleNormDual{\cdot}$
are orthant-strictly monotonic, any nondecreasing finite-valued \FSM\ is equal to its 
\Capra-biconjugate, that is, is a \Capra-convex function in the sense 
of generalized convexity.
In Sect.~\ref{The_Capra_conjugacy_and_the_support_mapping}, 
we introduce families of local norms --- (dual) local-coordinate-$\IndexSubset$
norms, generalized top-$\IndexSubset$ and local-$\IndexSubset$-support
dual~norms --- 
based on the source norm~$\TripleNorm{\cdot}$.
Then, we provide formulas 
for \Capra-conjugates, \Capra-subdifferentials and \Capra-biconjugates
of~\FSM.
%
%
%
%
In Sect.~\ref{Hidden_convexity_and_variational_formulation_for_the_pseudo_norm},
we still suppose that both the source norm $\TripleNorm{\cdot}$ and the dual norm $\TripleNormDual{\cdot}$
are orthant-strictly monotonic, and we arrive at unexpected results. 
Indeed, 
we show that any nondecreasing finite-valued \FSM\
coincides, on the unit sphere~$\TripleNormSphere =
\bset{\primal \in \RR^d}{\TripleNorm{\primal} = 1} $, 
with a proper convex lower semi continuous function on~$\RR^d$, and
we deduce a variational formula for nondecreasing finite-valued \FSM\
taking the value~$0$ on the null vector.
%
Sect.~\ref{Conclusion} concludes.
In the Appendix, 
Sect.~\ref{Properties_Constant_along_primal_rays_coupling} provides background
on the constant along primal rays coupling (\Capra), and 
Sect.~\ref{Coordinate-k_and_dual_coordinate-k_norms} gathers material on
local-coordinate-$\IndexSubset$ and generalized local-top-$\IndexSubset$ norms,
and their dual norms.
%

\section{\Capra-convexity of nondecreasing functions of the support mapping (\FSM)}
\label{The_Capra_conjugacy_under_orthant-strict_monotonicity}

In~\S\ref{The_support_mapping_and_its_level_sets},
we define the support mapping, set functions and functions of the support mapping (\FSM).
In~\S\ref{Constant_along_primal_rays_coupling}, we gather background 
on Fenchel-Moreau conjugacies with respect to a coupling,
and we introduce the constant along primal rays coupling~$\CouplingCapra$ (\Capra)
induced by a source norm~$\TripleNorm{\cdot}$.
In~\S\ref{Capra-convexity_of_nondecreasing_FSM}, we state our main result on 
\Capra-convexity of nondecreasing functions of \FSM.

We work on the Euclidian space~$\RR^d$
(with $d \in \NN^*$), equipped with the scalar product 
\( \proscal{\cdot}{\cdot} \) (but not necessarily with the Euclidian norm).
We use the notation $\barRR = [-\infty,+\infty] $.

\subsection{The support mapping}
\label{The_support_mapping_and_its_level_sets}

Let $d \ge 1$ be a fixed integer. We denote by~$\IndexSet$ the following set of indices:
\begin{equation}
  \IndexSet=\Vset \eqfinp  
\end{equation}
We consider the so-called \emph{$\RR^d$-support mapping},
denoted by $\SupportMapping$  and defined by 
\begin{align}
  \SupportMapping : \RR^d
  &\to 2^{\IndexSet} 
    \nonumber 
  \\
  \primal
  &\mapsto \bset{ j \in \IndexSet }{\primal_j \not= 0 } 
    \eqfinp
    \label{eq:support_mapping}
\end{align}
We consider \emph{set functions}
\begin{equation}
  \FunctionBigF: 2^\IndexSet \to \barRR
  \eqfinp   
  \label{eq:FunctionBigF}
\end{equation}
A set function~$\FunctionBigF$ is said to be normalized if 
\( \FunctionBigF\np{\emptyset}=0\).

\begin{definition}
  A \emph{function of the support mapping} (in short, \FSM)
is a function of the form $\FunctionBigF\circ \SupportMapping$,
where~$\FunctionBigF$ is a set function as in~\eqref{eq:FunctionBigF}.
A \FSM\ is said to be \emph{normalized} if it takes the value~$0$ on the null
vector.
\label{de:FSM}
\end{definition}
An example of function of the support mapping is given by 
\( \lzero : \RR^d \to \IndexSet \subset \barRR \),
defined by \( \lzero\np{\primal} = \cardinal{ \Support{\primal} }
= \textrm{number of nonzero components of } \primal \),
\( \forall \primal \in \RR^d \), where $\cardinal{\IndexSubset}$ denotes the cardinal of 
a subset \( \IndexSubset \subset \IndexSet \). The function~\( \lzero \)
falls in our scope with \( \FunctionBigF=\cardinal{\cdot} \), as
$\lzero= \cardinal{\cdot} \circ \SupportMapping$.

It is relevant to observe that \FSM\ are 0-homogeneous, since
the mapping $\SupportMapping$ itself is 0-homogeneous in the sense that:
\begin{equation}
  \SupportMapping\np{\rho\primal} = \SupportMapping\np{\primal} 
  \eqsepv \forall \rho \in \RR\backslash\{0\}
  \eqsepv \forall \primal \in \RR^d
  \eqfinp
  \label{eq:support_mapping_is_0-homogeneous}
\end{equation}
\begin{subequations}
We introduce the \emph{level sets} of the mapping $\SupportMapping$
\begin{equation}
  \SuppLevelSet{\SupportMapping}{\IndexSubset} 
  = 
  \nset{ \primal \in \RR^d }{ \Support{\primal} \subset \IndexSubset}
  \eqsepv \forall \IndexSubset \subset \IndexSet 
  \eqfinv
  \label{eq:support_mapping_level_set}
\end{equation}
and the \emph{level curves} of the mapping $\SupportMapping$
\begin{equation}
  \SuppLevelCurve{\SupportMapping}{\IndexSubset} 
  = 
  \nset{ \primal \in \RR^d }{ \Support{\primal} = \IndexSubset}
  \eqsepv \forall \IndexSubset \subset \IndexSet 
  \eqfinp
  \label{eq:support_mapping_level_curve}
\end{equation}  
\end{subequations}
%

For any subset \( \IndexSubset \subset \IndexSet \), 
we denote the subspace of~$\RR^d$ made of vectors
whose components vanish outside of~$\IndexSubset$ by\footnote{%
  Here, following notation from Game Theory, 
  we have denoted by $-\IndexSubset$ the complementary subset 
  of~$\IndexSubset$ in \( \IndexSet \): \( \IndexSubset \cup (-\IndexSubset) = \IndexSet \)
  and \( \IndexSubset \cap (-\IndexSubset) = \emptyset \).}
\begin{equation}
  \FlatRR_{\IndexSubset} = \RR^\IndexSubset \times \{0\}^{-\IndexSubset} =
  \bset{ \primal \in \RR^d }{ \primal_j=0 \eqsepv \forall j \not\in \IndexSubset } 
  \subset \RR^d 
  \eqfinv
  \label{eq:FlatRR}
\end{equation}
where \( \FlatRR_{\emptyset}= \na{0} \).
We denote by \( \pi_\IndexSubset : \RR^d \to \FlatRR_{\IndexSubset} \) 
the \emph{orthogonal projection mapping}
and, for any vector \( \primal \in \RR^d \), by 
\( \primal_\IndexSubset = \pi_\IndexSubset\np{\primal} \in \FlatRR_{\IndexSubset} \) 
the vector which coincides with~\( \primal \),
except for the components outside of~$\IndexSubset$ that are zero.
It is easily seen that the orthogonal projection mapping~$\pi_\IndexSubset$
is self-dual, giving
\begin{equation}
  \proscal{\primal_\IndexSubset}{\dual_\IndexSubset} 
  = \proscal{\primal_\IndexSubset}{\dual} 
  = \bscal{\pi_\IndexSubset\np{\primal}}{\dual} 
  = \bscal{\primal}{\pi_\IndexSubset\np{\dual}}
  = \proscal{\primal}{\dual_\IndexSubset} 
  \eqsepv 
  \forall \primal \in \RR^d 
  \eqsepv 
  \forall \dual \in \RR^d 
  \eqfinp
  \label{eq:orthogonal_projection_self-dual}
\end{equation}
The level sets of the $\SupportMapping$ mapping 
in~\eqref{eq:support_mapping_level_set}
are easily related to the subspaces~\( \FlatRR_{\IndexSubset} \) of~\( \RR^d \),
as defined in~\eqref{eq:FlatRR}, by
\begin{equation}
  \SuppLevelSet{\SupportMapping}{\IndexSubset} 
  =
  \bset{ \primal \in \RR^d }{ \Support{\primal} \subset \IndexSubset}
  {= \FlatRR_{\IndexSubset} }
  \eqsepv \forall \IndexSubset \subset \IndexSet
  \eqfinp
  \label{eq:support_mapping_level_set_FlatRR}
\end{equation}
As we manipulate functions with values 
in~$\barRR = [-\infty,+\infty] $,
we adopt the Moreau \emph{lower ($\LowPlus$) and upper ($\UppPlus$) additions} \cite{Moreau:1970},
which extend the usual addition~($+$) with 
\( \np{+\infty} \LowPlus \np{-\infty}=\np{-\infty} \LowPlus \np{+\infty}=-\infty \) and
\( \np{+\infty} \UppPlus \np{-\infty}=\np{-\infty} \UppPlus \np{+\infty}=+\infty \).
Let \( \UNCERTAIN \) be a set.
For any function \( \fonctionuncertain : \UNCERTAIN \to \barRR \),
its \emph{epigraph} is \( \epigraph\fonctionuncertain= 
\defset{ \np{\uncertain,t}\in\UNCERTAIN\times\RR}%
{\fonctionuncertain\np{\uncertain} \leq t} \),
its \emph{effective domain} is 
\( \dom\fonctionuncertain= 
\defset{\uncertain\in\UNCERTAIN}{ \fonctionuncertain\np{\uncertain} <+\infty}
\).
A function \( \fonctionuncertain : \UNCERTAIN \to \barRR \)
is said to be \emph{proper} if it never takes the value~$-\infty$
and if \( \dom\fonctionuncertain \not = \emptyset \).
When \( \UNCERTAIN \) is equipped with a topology,
the function \( \fonctionuncertain : \UNCERTAIN \to \barRR \)
is said to be \emph{lower semi continuous (lsc)}
if its epigraph is closed.
For any subset \( \Uncertain \subset \UNCERTAIN \),
$\delta_{\Uncertain} : \UNCERTAIN \to \barRR $ denotes the \emph{characteristic function} of the
set~$\Uncertain$:
\begin{equation}
  \delta_{\Uncertain}\np{\uncertain} = 0 \mtext{ if } \uncertain \in \Uncertain
  \eqsepv
  \delta_{\Uncertain}\np{\uncertain} = +\infty \mtext{ if } \uncertain \not\in \Uncertain 
  \eqfinp 
  \label{eq:characteristic_function}
\end{equation}

\subsection{Constant along primal rays coupling (\Capra)}
\label{Constant_along_primal_rays_coupling}

In~\S\ref{Background_on_Fenchel-Moreau_conjugacies}, we gather background 
on Fenchel-Moreau conjugacies with respect to a coupling~$\coupling$ and we give definition and
characterization of $\coupling$-convex functions.
Then, in~\S\ref{subsubsectionConstant_along_primal_rays_coupling},
we show how to define a constant along primal rays
  coupling~$\CouplingCapra$ (\Capra) by means of a norm.

\subsubsection{Background on Fenchel-Moreau conjugacies}
\label{Background_on_Fenchel-Moreau_conjugacies}

We review concepts and notations related to the Fenchel conjugacy
(we refer the reader to \cite{Rockafellar:1974}), and 
then we present how they are extended to general conjugacies
\cite{Singer:1997,Rubinov:2000,Martinez-Legaz:2005}.

\subsubsubsection{The Fenchel conjugacy on~\( \RR^d \)}

The classic \emph{Fenchel conjugacy}~$\star$ is defined, 
for any functions \( \fonctionprimal : \PRIMAL  \to \barRR \)
and \( \fonctiondual : \DUAL \to \barRR \), by\footnote{%
  In convex analysis, one does not use the notation~\( \LFMr{} \)
in~\eqref{eq:Fenchel_conjugate_reverse} and \eqref{eq:Fenchel_biconjugate},
  but simply~\( \LFM{} \). We use~\( \LFMr{} \) and ~\( \LFMbi{} \)
  to be consistent with the notation~\eqref{eq:Fenchel-Moreau_reverse_conjugate} 
and~\eqref{eq:Fenchel-Moreau_biconjugate} for general conjugacies.}
\begin{subequations}
  \begin{align}
    \LFM{\fonctionprimal}\np{\dual} 
    &= 
      \sup_{\primal \in \RR^d} \Bp{ \proscal{\primal}{\dual} 
      \LowPlus \bp{ -\fonctionprimal\np{\primal} } } 
      \eqsepv \forall \dual \in \RR^d
      \eqfinv
      \label{eq:Fenchel_conjugate}
    \\
    \LFMr{\fonctiondual}\np{\primal} 
    &= 
      \sup_{ \dual \in \DUAL } \Bp{ \proscal{\primal}{\dual} 
      \LowPlus \bp{ -\fonctiondual\np{\dual} } } 
      \eqsepv \forall \primal \in \RR^d
      \eqfinv
      \label{eq:Fenchel_conjugate_reverse}
    \\
    \LFMbi{\fonctionprimal}\np{\primal} 
    &= 
      \sup_{\dual \in \RR^d} \Bp{ \proscal{\primal}{\dual} 
      \LowPlus \bp{ -\LFM{\fonctionprimal}\np{\dual} } } 
      \eqsepv \forall \primal \in \RR^d
      \eqfinp
      \label{eq:Fenchel_biconjugate}
  \end{align}
\end{subequations}
Recall that a function is said to be \emph{convex} if its
epigraph is a convex subset of $\RR^d \times \RR$,
and is said to be \emph{closed} if it is either lsc
and nowhere having the value $-\infty$,
or is the constant function~$-\infty$
\cite[p.~15]{Rockafellar:1974}.
It is proved that the Fenchel conjugacy induces a one-to-one correspondence
between the closed convex functions on~$\RR^d$ and themselves
\cite[Theorem~5]{Rockafellar:1974}.
Closed convex functions are the two constant functions~$-\infty$ and~$+\infty$
united with all proper convex lsc functions.\footnote{%
In particular, any closed convex function that takes at least one finite value
is necessarily proper convex~lsc. \label{ft:closed_convex_function}}

\subsubsubsection{The general case \cite{Singer:1997,Rubinov:2000,Martinez-Legaz:2005}}

We consider two sets $\PRIMAL$ (``primal''), $\DUAL$ (``dual''),
not necessarily vector spaces, together 
with a \emph{coupling} function
\begin{equation}
  \coupling : \PRIMAL \times \DUAL \to \barRR 
  \eqfinp 
\end{equation}
With any coupling, one associates \emph{conjugacies} 
from \( \barRR^\PRIMAL \) to \( \barRR^\DUAL \) 
and from \( \barRR^\DUAL \) to \( \barRR^\PRIMAL \) as follows.

\begin{subequations}
  \begin{definition}
    The \emph{$\coupling$-Fenchel-Moreau conjugate} of a 
    function \( \fonctionprimal : \PRIMAL  \to \barRR \), 
    with respect to the coupling~$\coupling$, is
    the function \( \SFM{\fonctionprimal}{\coupling} : \DUAL  \to \barRR \) 
    defined by
    \begin{equation}
      \SFM{\fonctionprimal}{\coupling}\np{\dual} = 
      \sup_{\primal \in \PRIMAL} \Bp{ \coupling\np{\primal,\dual} 
        \LowPlus \bp{ -\fonctionprimal\np{\primal} } } 
      \eqsepv \forall \dual \in \DUAL
      \eqfinp
      \label{eq:Fenchel-Moreau_conjugate}
    \end{equation}
    With the coupling $\coupling$, we associate 
    the \emph{reverse coupling~$\coupling'$} defined by 
    \begin{equation}
      \coupling': \DUAL \times \PRIMAL \to \barRR 
      \eqsepv
      \coupling'\np{\dual,\primal}= \coupling\np{\primal,\dual} 
      \eqsepv
      \forall \np{\dual,\primal} \in \DUAL \times \PRIMAL
      \eqfinp
      \label{eq:reverse_coupling}
    \end{equation}
    The \emph{$\coupling'$-Fenchel-Moreau conjugate} of a 
    function \( \fonctiondual : \DUAL \to \barRR \), 
    with respect to the coupling~$\coupling'$, is
    the function \( \SFM{\fonctiondual}{\coupling'} : \PRIMAL \to \barRR \) 
    defined by
    \begin{equation}
      \SFM{\fonctiondual}{\coupling'}\np{\primal} = 
      \sup_{ \dual \in \DUAL } \Bp{ \coupling\np{\primal,\dual} 
        \LowPlus \bp{ -\fonctiondual\np{\dual} } } 
      \eqsepv \forall \primal \in \PRIMAL 
      \eqfinp
      \label{eq:Fenchel-Moreau_reverse_conjugate}
    \end{equation}
    The \emph{$\coupling$-Fenchel-Moreau biconjugate} of a 
    function \( \fonctionprimal : \PRIMAL  \to \barRR \), 
    with respect to the coupling~$\coupling$, is
    the function \( \SFMbi{\fonctionprimal}{\coupling} : \PRIMAL \to \barRR \) 
    defined by
    \begin{equation}
      \SFMbi{\fonctionprimal}{\coupling}\np{\primal} = 
      \bp{\SFM{\fonctionprimal}{\coupling}}^{\coupling'} \np{\primal} = 
      \sup_{ \dual \in \DUAL } \Bp{ \coupling\np{\primal,\dual} 
        \LowPlus \bp{ -\SFM{\fonctionprimal}{\coupling}\np{\dual} } } 
      \eqsepv \forall \primal \in \PRIMAL 
      \eqfinp
      \label{eq:Fenchel-Moreau_biconjugate}
    \end{equation}
  \end{definition}
\end{subequations}
The biconjugate of a 
function \( \fonctionprimal : \PRIMAL  \to \barRR \) satisfies
\begin{equation}
  \SFMbi{\fonctionprimal}{\coupling}\np{\primal}
  \leq \fonctionprimal\np{\primal}
  \eqsepv \forall \primal \in \PRIMAL 
  \eqfinp
  \label{eq:galois-cor}
\end{equation}
With the notion of $\coupling$-biconjugate, 
the classic notion of convex function is generalized as follows.
\begin{definition}
  A function \( \fonctionprimal : \PRIMAL \to \barRR \) 
  is said to be \emph{$\coupling$-convex} it is equal to its
  $\coupling$-biconjugate:
  \begin{equation}
    \fonctionprimal \textrm{ is } \coupling\textrm{-convex }
    \iff 
    \SFMbi{\fonctionprimal}{\coupling}=\fonctionprimal 
    \eqfinp
  \end{equation}
  \label{de:coupling-convex_function}
\end{definition}
In generalized convexity, it is established
that $\coupling$-convex functions 
are all functions of the form 
$\SFM{ \fonctiondual }{\coupling'}$, 
for all \( \fonctiondual : \DUAL \to \barRR \),
or, equivalently,
all functions of the form $\SFMbi{\fonctionprimal}{\coupling}$, 
for all \( \fonctionprimal : \PRIMAL \to \barRR \).
As an illustration, the $\star$-convex functions are the closed convex functions
since, as recalled above, the Fenchel conjugacy induces a one-to-one correspondence
between the closed convex functions on~$\RR^d$ and themselves.

\subsubsection{Constant along primal rays coupling (\Capra)}
\label{subsubsectionConstant_along_primal_rays_coupling}

  Let $\TripleNorm{\cdot}$ be a norm on~$\RR^d$, called the \emph{source norm}.
We denote the unit sphere~$\TripleNormSphere$ and the unit ball~$\TripleNormBall$ 
of the norm~$\TripleNorm{\cdot}$ by 
\begin{equation}
  \TripleNormSphere= 
      \defset{\primal \in \RR^d}{\TripleNorm{\primal} = 1} 
\eqsepv
    \TripleNormBall = 
      \defset{\primal \in \RR^d}{\TripleNorm{\primal} \leq 1} 
      \eqfinp
 \label{eq:triplenorm_unit_sphere}
\end{equation}

Following \cite{Chancelier-DeLara:2020_ECAPRA_JCA,Chancelier-DeLara:2020_CAPRA_OPTIMIZATION}, we introduce 
the \emph{constant along primal rays coupling~$\CouplingCapra$ (\Capra)}.

\begin{definition}(\cite[Definition~\ref{CAPRA-de:coupling_CAPRA}]{Chancelier-DeLara:2020_CAPRA_OPTIMIZATION})
  We define the \emph{coupling}~$\CouplingCapra$, or \Capra,
  between $\RR^d$ and $\RR^d$ by
  \begin{equation}
    \forall \dual \in \RR^d \eqsepv 
    \begin{cases}
      \CouplingCapra\np{\primal, \dual} 
      &= \displaystyle
      \frac{ \proscal{\primal}{\dual} }{ \TripleNorm{\primal} }
      \eqsepv \forall \primal \in \RR^d\backslash\{0\} \eqfinv
      \\[4mm]
      \CouplingCapra\np{0, \dual} &= 0.
    \end{cases}
    \label{eq:coupling_CAPRA}
  \end{equation}
    \label{de:coupling_CAPRA}
\end{definition}
We stress the point that, in~\eqref{eq:coupling_CAPRA},
the Euclidian scalar product \( \proscal{\primal}{\dual} \)
and the norm term \( \TripleNorm{\primal} \) need not be related, 
that is, the norm~$\TripleNorm{\cdot}$ is not necessarily Euclidian.

The coupling \Capra\ has the property of being 
\emph{constant along primal rays}, hence the acronym~\Capra\
(Constant Along Primal RAys).
We introduce 
the primal \emph{normalization mapping}~$\normalized$,
from $\RR^d$ towards the unit sphere~\( \TripleNormSphere \)
united with $\{0\}$, 
as follows:
\begin{equation}
  \normalized : \RR^d \to \TripleNormSphere \cup \{0\} 
  \eqsepv
  \normalized\np{\primal}=
  \begin{cases}
    \frac{\primal}{\TripleNorm{\primal}}
    & \mtext{ if } \primal \neq 0 \eqfinv 
    \\[3mm]
    0 
    & \mtext{ if } \primal = 0 \eqfinp
  \end{cases}  
  \label{eq:normalization_mapping}
\end{equation}

\subsection{\Capra-convexity of nondecreasing \FSM}
\label{Capra-convexity_of_nondecreasing_FSM}


We recall definitions of orthant-monotonic 
and orthant-strictly monotonic norms.
For any \( \primal \in \RR^d \), we denote by \( \module{\primal} \)
the vector of~$ \RR^{d}$ with components $|\primal_i|$, $i=1,\ldots,d$:
\begin{equation}
  \primal=\np{\primal_1,\ldots,\primal_d} \implies 
  \module{\primal} =\np{\module{\primal_1},\ldots,\module{\primal_d}}
  \eqfinp 
  \label{eq:module}
\end{equation}

\begin{definition}    
  A norm \( \TripleNorm{\cdot}\) on the space~\( \RR^d \) is called
  \begin{itemize}
  \item 
    \emph{orthant-monotonic} \cite{Gries:1967}
    if, for all 
    $\primal$, $\primal'$ in~$ \RR^{d}$, we have 
    \bp{  \(
      |\primal| \le |\primal'| \text{ and } 
      \primal~\circ~\primal' \ge 0 \implies 
      \TripleNorm{\primal} \le \TripleNorm{\primal'}
      \) },
    where $|\primal| \leq |\primal'|$ means 
    $|\primal_i| \leq |\primal'_i|$ for all $i=1,\ldots,d$, 
    and     where $\primal~\circ~\primal' =
    \np{ \primal_1 \primal'_1,\ldots, \primal_d \primal'_d}$
    is the Hadamard (entrywise) product, 
  \item 
    \emph{orthant-strictly monotonic} \cite[Definition~\ref{OSM-de:orthant-monotonic}]{Chancelier-DeLara:2020_OSM}
    if, for all 
    $\primal$, $\primal'$ in~$ \RR^{d}$, we have 
    \bp{  \(
      |\primal| < |\primal'| \text{ and } 
      \primal~\circ~\primal' \ge 0 \implies 
      \TripleNorm{\primal} < \TripleNorm{\primal'}
      \) },
    where \( |\primal| < |\primal'| \) means that 
    $|\primal_i| \le |\primal'_i|$ for all $i=1,\ldots,d$, 
    and there exists $j \in \na{1,\ldots,d}$ such that
    $|\primal_j| < |\primal'_j|$.
  \end{itemize}
  \label{de:orthant-monotonic}
\end{definition}

We now establish that,
when both the source norm and its dual norm are orthant-strictly monotonic,
any nondecreasing finite-valued \FSM\ is equal to its 
\Capra-biconjugate, that is, is a \Capra-convex function.

\begin{theorem}
  \label{th:support_mapping_conjugate}
   Let $\TripleNorm{\cdot}$ be the source norm with 
associated coupling $\CouplingCapra$, as in Definition~\ref{de:coupling_CAPRA}. 
Suppose that both the norm $\TripleNorm{\cdot}$ and the dual norm $\TripleNormDual{\cdot}$
  are orthant-strictly monotonic.
  Then, for any nondecreasing finite-valued set function 
\( \FunctionBigF : 2^{\IndexSet} \to \RR \), we have
  \begin{equation}
    \SFMbi{ \np{ \FunctionBigF \circ \SupportMapping } }{\CouplingCapra} 
    = \FunctionBigF \circ \SupportMapping 
    \eqfinv
    \label{eq:biconjugate_l0norm_TopDualNorm}        
  \end{equation}
  that is, the function \( \FunctionBigF \circ \SupportMapping \) 
  is a \Capra-convex function (see Definition~\ref{de:coupling-convex_function}). 
\end{theorem}

The proof of Theorem~\ref{th:support_mapping_conjugate} is relegated
in~\S\ref{Proof_of_Theorem},
after we analyze functions of the support mapping by means of the
\Capra-conjugacy in Sect.~\ref{The_Capra_conjugacy_and_the_support_mapping}.

\section{The \Capra-conjugacy and functions of the support mapping (\FSM)}
\label{The_Capra_conjugacy_and_the_support_mapping}

In~\S\ref{Definition_of_local_coordinate-k_and_dual_coordinate-k_norms}, 
we introduce families of local norms --- local-coordinate-$\IndexSubset$ norms and their dual
norms, as well as generalized local-top-$\IndexSubset$ and
local-$\IndexSubset$-support dual~norms.
Then, we provide formulas 
for \Capra-conjugates of functions of the support mapping
in~\S\ref{CAPRAC_conjugates_related_to_the_pseudo_norm},
for \Capra-subdifferentials
in~\S\ref{Capra-subdifferentials_related_to_the_function_of_the_support_mapping},
and for \Capra-biconjugates
in~\S\ref{CAPRAC_biconjugates_related_to_the_pseudo_norm}.
Finally, we prove Theorem~\ref{th:support_mapping_conjugate}
in~\S\ref{Proof_of_Theorem}.

\subsection{Local-$\IndexSubset$ norms}
\label{Definition_of_local_coordinate-k_and_dual_coordinate-k_norms}

To analyze the support mapping by means of the \Capra\ conjugacy, 
we introduce families of local norms --- on the subspaces
$\FlatRR_\IndexSubset \subset \RR^d$ in~\eqref{eq:FlatRR}, 
for any subset \( \IndexSubset \subset \IndexSet \) --- as follows.

\subsubsubsection{Restriction norms}

We introduce the so-called restriction norms.

\begin{definition}
  For any norm~$\TripleNorm{\cdot}$ on~$\RR^d$
  and any subset \( K \subset\ic{1,d} \),
  we define 
  \begin{itemize}
  \item 
    the \emph{$K$-restriction norm} \( \TripleNorm{\cdot}_{K} \)
    on the subspace~\( \FlatRR_{K} \) of~\( \RR^d \),
    as defined in~\eqref{eq:FlatRR}, by
    \begin{equation}
      \TripleNorm{\primal}_{K} = \TripleNorm{\primal} 
      \eqsepv
      \forall \primal \in \FlatRR_{K} 
      \eqfinp 
      \label{eq:K_norm}
    \end{equation}
  \item 
    the \emph{$\SetStar{K}$-norm} \( \TripleNorm{\cdot}_{K,\star} \),
    on the subspace~\( \FlatRR_{K} \) of~\( \RR^d \), which is 
    the norm \( \bp{\TripleNorm{\cdot}_{K}}_{\star} \),
    given by the dual norm (on the subspace~\( \FlatRR_{K} \))
    of the restriction norm~\( \TripleNorm{\cdot}_{K} \) 
    to the subspace~\( \FlatRR_{K} \) (first restriction, then dual),
 \item 
    the $\StarSet{K}$-norm
    \( \TripleNorm{\cdot}_{\star,K} \) is 
    the norm \( \bp{\TripleNorm{\cdot}_{\star}}_{K} \),
    given by the restriction to the subspace~\( \FlatRR_{K} \) of
    the dual norm~$\TripleNormDual{\cdot}$ (first dual, then restriction).
  \end{itemize}
  \label{de:K_norm}
\end{definition}

We have that \cite[Equation~\eqref{OSM-eq:K_star=sigma}]{Chancelier-DeLara:2020_OSM}
\begin{equation}
  \TripleNorm{\dual}_{K,\star}
  =
  \sigma_{ \FlatRR_{K} \cap \TripleNormBall }\np{\dual}  
  = \sigma_{ \FlatRR_{K} \cap \TripleNormSphere }\np{\dual}  
  \eqsepv \forall \dual \in \FlatRR_{K} 
  \eqfinp
  \label{eq:K_star=sigma}
\end{equation}

\subsubsubsection{Local-coordinate-$\IndexSubset$ norms and their dual norms}

\begin{definition}
  For any subset \( \IndexSubset \subset \IndexSet \), 
we call \emph{local-coordinate-$\IndexSubset$ norm} 
  (associated with the source norm~$\TripleNorm{\cdot}$)
  the norm \( \CoordinateNorm{\TripleNorm{\cdot}}{\IndexSubset} \)
  (on the subspace $\FlatRR_\IndexSubset$ in~\eqref{eq:FlatRR}) 
given by 
\begin{equation}
  \CoordinateNorm{\TripleNorm{\cdot}}{\IndexSubset} 
= \bp{ \CoordinateNormDual{\TripleNorm{\cdot}}{\IndexSubset} }_\star 
\eqfinv
    \label{eq:coordinate_norm_definition}
\end{equation}
that is, whose dual norm (on the subspace $\FlatRR_\IndexSubset$) is the 
  \emph{dual local-coordinate-$\IndexSubset$ norm}, denoted by
  \( \CoordinateNormDual{\TripleNorm{\cdot}}{\IndexSubset} \), 
  with expression
  \begin{equation}
    \CoordinateNormDual{\TripleNorm{\dual}}{\IndexSubset}
    =
    \sup_{\IndexSubsetLess \subset \IndexSubset} \TripleNorm{\dual_{\IndexSubsetLess}}_{\IndexSubsetLess,\star} 
    \eqsepv \forall \dual \in \FlatRR_\IndexSubset
    \eqfinv
    \label{eq:dual_coordinate_norm_definition}
  \end{equation}
  where the $\SetStar{\IndexSubset}$-norm \( \TripleNorm{\cdot}_{\IndexSubset,\star} \) is given in
  Definition~\ref{de:K_norm}.

  \label{de:coordinate_norm}
\end{definition}
We adopt the convention \( \CoordinateNormDual{\TripleNorm{\cdot}}{\emptyset} = 0 \).
%
%
We denote the unit sphere and the unit ball
of the local-coordinate-$\IndexSubset$ 
norm~\( \CoordinateNorm{\TripleNorm{\cdot}}{\IndexSubset} \) 
in~\eqref{eq:coordinate_norm_definition} by
\begin{subequations}
  \begin{align}
    \CoordinateNorm{\TripleNormSphere}{\IndexSubset}
    &=
      \defset{\primal \in \FlatRR_\IndexSubset}{%
      \CoordinateNorm{\TripleNorm{\primal}}{\IndexSubset} = 1 } 
      \subset \FlatRR_\IndexSubset
      \eqsepv \forall \IndexSubset \subset\IndexSet
      \eqfinv
      \label{eq:coordinate_norm_unit_sphere}
    \\
    \CoordinateNorm{\TripleNormBall}{\IndexSubset}
    &=
      \defset{\primal \in \FlatRR_\IndexSubset}{%
      \CoordinateNorm{\TripleNorm{\primal}}{\IndexSubset} \leq 1 } 
      \subset \FlatRR_\IndexSubset
      \eqsepv \forall \IndexSubset \subset\IndexSet
      \eqfinp
      \label{eq:coordinate_norm_unit_ball}
  \end{align}
  \label{eq:coordinate_norm_unit_sphere_ball}
\end{subequations}
We denote the unit sphere and the unit ball 
of the dual local-coordinate-$\IndexSubset$ norm
\( \CoordinateNormDual{\TripleNorm{\cdot}}{\IndexSubset} \) 
in~\eqref{eq:dual_coordinate_norm_definition} 
by
\begin{subequations}
  \begin{align}
    \CoordinateNormDual{\TripleNormSphere}{\IndexSubset} 
    &  = 
      \bset{\dual \in \FlatRR_\IndexSubset}{\CoordinateNormDual{\TripleNorm{\dual}}{\IndexSubset} = 1} 
      \subset \FlatRR_\IndexSubset
      \eqsepv \forall \IndexSubset \subset\IndexSet
      \eqfinv
      \label{eq:dual_coordinate_norm_unit_sphere}
    \\
    \CoordinateNormDual{\TripleNormBall}{\IndexSubset} 
    &  = 
      \defset{\dual \in \FlatRR_\IndexSubset}{\CoordinateNormDual{\TripleNorm{\dual}}{\IndexSubset} \leq 1} 
      \subset \FlatRR_\IndexSubset
      \eqsepv \forall \IndexSubset \subset\IndexSet
      \eqfinp
      \label{eq:dual_coordinate_norm_unit_ball}
  \end{align}
\end{subequations}
%

\subsubsubsection{Generalized local-top-$\IndexSubset$ and
local-$\IndexSubset$-support dual~norms}


\begin{definition}
  For any subset \( \IndexSubset \subset \IndexSet \), 
  we call \emph{generalized local-top-$\IndexSubset$ dual~norm}
  (associated with the source norm~$\TripleNorm{\cdot}$)
  the local norm (on the subspace $\FlatRR_\IndexSubset$ in~\eqref{eq:FlatRR}) defined by 
  \begin{equation}
    \TopDualNorm{\TripleNorm{\dual}}{\IndexSubset}
    =
    \sup_{\IndexSubsetLess \subset \IndexSubset} \TripleNormDual{\dual_{\IndexSubsetLess}} 
    =
    \sup_{\IndexSubsetLess \subset \IndexSubset} \TripleNorm{\dual_{\IndexSubsetLess}}_{\star,\IndexSubsetLess}
    \eqsepv \forall \dual \in \FlatRR_\IndexSubset 
    \eqfinv
    \label{eq:top_dual_norm}
  \end{equation}
  where we adopt the convention \( \TopDualNorm{\TripleNorm{\cdot}}{\emptyset} = 0 \).
  We call \emph{generalized local-$\IndexSubset$-support dual norm}
  the local dual norm (on the subspace $\FlatRR_\IndexSubset$) 
of the generalized local-top-$\IndexSubset$ dual~norm, denoted by\footnote{%
    We use the symbol~$\star$ in the superscript to indicate that the generalized
    local-$\IndexSubset$-support dual~norm \( \SupportDualNorm{\TripleNorm{\cdot}}{\IndexSubset} \)
    is a dual norm.}
  \( \SupportDualNorm{\TripleNorm{\cdot}}{\IndexSubset} \): 
  \begin{equation}
    \SupportDualNorm{\TripleNorm{\cdot}}{\IndexSubset} 
    = \bp{ \TopDualNorm{\TripleNorm{\cdot}}{\IndexSubset} }_{\star}
    \eqfinp
    \label{eq:support_dual_norm}
  \end{equation}
  \label{de:top_dual_norm}
\end{definition}

\subsection{\Capra-conjugates of \FSM}
\label{CAPRAC_conjugates_related_to_the_pseudo_norm}

With the Fenchel conjugacy, we calculate that 
\( \LFM{ \delta_{ \SuppLevelSet{\SupportMapping}{\IndexSubset}  } }= 
\delta_{  \{0\} } \) --- 
where \( \IndexSubset \neq \emptyset \) and
\( \delta_{ \SuppLevelSet{\SupportMapping}{\IndexSubset}  } \) is the characteristic function
of the level sets~\eqref{eq:support_mapping_level_set}
as defined in~\eqref{eq:characteristic_function} --- 
and, more generally, that 
\( \LFM{ \np{\FunctionBigF\circ\SupportMapping } } =
\sup \ba{ -\FunctionBigF\np{\emptyset}, 
  \delta_{\{0\}} \LowPlus \bp{ -\inf_{\IndexSubset \neq
      \emptyset}\FunctionBigF\np{\IndexSubset} } } \),
for any set function~$\FunctionBigF$ as in~\eqref{eq:FunctionBigF}.
Hence, the Fenchel conjugacy is not suitable
to handle \FSM.
By contrast, we will now show that we obtain more interesting formulas with the \Capra-conjugacy.

\begin{proposition}
   Let $\TripleNorm{\cdot}$ be the source norm with 
associated coupling $\CouplingCapra$, as in Definition~\ref{de:coupling_CAPRA}.
Let \( \FunctionBigF : 2^\IndexSet \to \barRR \) be a set function.

We have
  \begin{equation}
    \SFM{\Bp{\FunctionBigF\circ\mathrm{supp}}}{\CouplingCapra}\np{\dual} 
    =  \sup_{\IndexSubset \subset \IndexSet }
    \Bp{ \CoordinateNormDual{\TripleNorm{\dual_\IndexSubset}}{\IndexSubset} -
      \FunctionBigF\np{\IndexSubset} }
    \eqsepv \forall \dual \in \RR^d 
    \eqfinv
    \label{eq:CAPRA_conjugate_function_of_the_Support}
  \end{equation}
  where the family 
  \( \sequence{\CoordinateNormDual{\TripleNorm{\cdot}}{\IndexSubset}}{\IndexSubset\subset \IndexSet} \) 
  of dual local-coordinate-$\IndexSubset$ norms is as in
  Definition~\ref{de:coordinate_norm}, and with the convention
  \(\CoordinateNormDual{\TripleNorm{\dual_\emptyset}}{\emptyset} =0 \).

  If, in addition, the norm $\TripleNorm{\cdot}$ is orthant-monotonic, then we
  have that 
  \begin{equation}
    \SFM{ \np{ \FunctionBigF \circ \SupportMapping } }{\CouplingCapra} 
    =
    \sup_{\IndexSubset\subset \IndexSet} \Bc{ \TopDualNorm{\TripleNorm{\cdot}}{\IndexSubset}-\FunctionBigF\np{\IndexSubset} }  
    \eqfinv
    \label{eq:conjugate_l0norm_TopDualNorm}        
  \end{equation}
  where the family 
 \( \sequence{\TopDualNorm{\TripleNorm{\cdot}}{\IndexSubset}}{\IndexSubset \subset \IndexSet} \)
  of generalized local-top-$\IndexSubset$ dual~norms is 
  as in Definition~\ref{de:top_dual_norm}, and
  with the convention that \( \TopDualNorm{\TripleNorm{\cdot}}{\emptyset} = 0
  \).

  \label{pr:support_mapping_conjugate}
\end{proposition}

\begin{proof}
  To prove~\eqref{eq:CAPRA_conjugate_function_of_the_Support}, we apply the postponed
  Lemma~\ref{le:capra_support_and_delta},
  as well as results about the dual local-coordinate-$\IndexSubset$ norms~\(
  \CoordinateNormDual{\TripleNorm{\cdot}}{\IndexSubset} \)
  gathered in Sect.~\ref{Coordinate-k_and_dual_coordinate-k_norms}. 
  For any $\dual \in \RR^d$, we have
  \begin{align*}
    \SFM{ \bp{\FunctionBigF\circ\SupportMapping} }{\CouplingCapra}\np{\dual}
    &=
      \sup_{\IndexSubset \subset \IndexSet } \Bp{ 
      \sigma_{ \TripleNormSphere \cap (\SuppLevelCurve{\SupportMapping}{\IndexSubset})} \np{\dual} 
      \LowPlus \bp{-\FunctionBigF\np{\IndexSubset}} }
      \tag{Lemma~\ref{le:capra_support_and_delta} with $\Gamma = \RR^d$}
    \\
    &=
      \sup_{\IndexSubset \subset \IndexSet }
      \Bp{ \CoordinateNormDual{\TripleNorm{\dual_\IndexSubset}}{\IndexSubset} - \FunctionBigF\np{\IndexSubset}}
      \tag{as \( \sigma_{\SuppLevelCurve{\SupportMapping}{\IndexSubset} \cap \TripleNormSphere}\np{\dual}  =
      \sigma_{\SuppLevelCurve{\SupportMapping}{\IndexSubset} \cap \TripleNormSphere}\np{\dual_\IndexSubset} =
      \CoordinateNormDual{\TripleNorm{\dual_\IndexSubset}}{\IndexSubset} \)
      by~\eqref{eq:dual_coordinate_norm} } 
     \eqfinp
  \end{align*}
If the norm~$\TripleNorm{\cdot}$ is orthant-monotonic,
we have that 
 \( \CoordinateNormDual{\TripleNorm{\cdot}}{\IndexSubset}
  =
  \TopDualNorm{\TripleNorm{\cdot}}{\IndexSubset} \)
by Proposition~\ref{pr:dual_coordinate-k_norm_=_generalized_top-k_norm}.

  This ends the proof.
\end{proof}

\begin{lemma}
  \label{le:capra_support_and_delta}
  Let $\TripleNorm{\cdot}$ be the source norm with 
associated coupling $\CouplingCapra$, as in Definition~\ref{de:coupling_CAPRA}.

  For any set function $\FunctionBigF: 2^\IndexSet \to \barRR$ and any $\Gamma \subset \RR^d$
  such that
  $0\in \Gamma$ and $\normalized\np{\Gamma \cap \SuppLevelCurve{\SupportMapping}{\IndexSubset}}
  =\Gamma \cap \normalized\np{\SuppLevelCurve{\SupportMapping}{\IndexSubset}}$
  for all $\IndexSubset\subset \IndexSet$,
  where the normalization mapping~$\normalized$ has been defined in~\eqref{eq:normalization_mapping},
  we have
  \begin{equation}
    \SFM{ \bp{\FunctionBigF\circ\SupportMapping \UppPlus \delta_{\Gamma}} }{\CouplingCapra}\np{\dual}
    =
    \sup_{\IndexSubset \subset \IndexSet }
    \Bp{ 
      \sigma_{ \TripleNormSphere\cap\Gamma \cap (\SuppLevelCurve{\SupportMapping}{\IndexSubset})} \np{\dual} 
      \LowPlus \bp{-\FunctionBigF\np{\IndexSubset}} }
    \eqsepv \forall \dual \in \RR^d 
    \eqfinp
    \label{eq:conjugate_support_plus_delta}
  \end{equation}
\end{lemma}

\begin{proof}
  Using the level curves of the $\SupportMapping$ mapping
  in~\eqref{eq:support_mapping_level_curve},
  we obtain a representation of the function $\FunctionBigF\circ\SupportMapping$ as
  \begin{equation}
    \FunctionBigF\circ \SupportMapping =
    \inf_{\IndexSubset \subset \IndexSet } \bp{ \delta_{\SuppLevelCurve{\SupportMapping}{\IndexSubset}}  \UppPlus 
      \FunctionBigF\np{\IndexSubset}}
    \eqfinp 
    \label{eq:level_set_support}
  \end{equation}
  For all $\dual \in \RR^d$, we have
  \begin{align*}
    \SFM{ \Bp{\FunctionBigF
    \circ\SupportMapping \, \UppPlus \delta_{\Gamma}} }{\CouplingCapra}\np{\dual}
    &= 
      \SFM{ \Bp{ \inf_{ \IndexSubset \subset \IndexSet } 
      \bp{ \delta_{\np{\SuppLevelCurve{\SupportMapping}{\IndexSubset}}}
      \UppPlus \FunctionBigF\np{\IndexSubset} }
      \UppPlus \delta_{\Gamma}
      } }%
      {\CouplingCapra}\np{\dual} 
      \tag{using Equation~\eqref{eq:level_set_support}}
    \\
    &= 
      \SFM{ \Bp{ \inf_{ \IndexSubset \subset \IndexSet } 
      \bp{ \delta_{\Gamma} \UppPlus \delta_{\np{\SuppLevelCurve{\SupportMapping}{\IndexSubset}}}
      \UppPlus \FunctionBigF\np{\IndexSubset} }
      } }%
      {\CouplingCapra}\np{\dual} 
      \intertext{as $\inf_{K}\varphi(K) \UppPlus r = \inf_{K}\bp{\varphi(K)
      \UppPlus r}$, for any set function~$\varphi : 2^\IndexSet \to \barRR$ and any $r \in \barRR$ \cite{Moreau:1970}}
    &= 
      \SFM{ \Bp{ \inf_{ \IndexSubset \subset \IndexSet } 
      \bp{ \delta_{\Gamma\cap\np{\SuppLevelCurve{\SupportMapping}{\IndexSubset}}}
      \UppPlus \FunctionBigF\np{\IndexSubset} }
      } }%
      {\CouplingCapra}\np{\dual} 
      \tag{as \( \delta_{\Gamma} \UppPlus
      \delta_{\np{\SuppLevelCurve{\SupportMapping}{\IndexSubset}}}
      = \delta_{\Gamma\cap\np{\SuppLevelCurve{\SupportMapping}{\IndexSubset}}} \)}
    \\
    &= 
      \sup_{\IndexSubset \subset \IndexSet }
      \SFM{ \Bp{ \delta_{\Gamma\cap\np{\SuppLevelCurve{\SupportMapping}{\IndexSubset}}}
      \UppPlus \bp{-\FunctionBigF\np{\IndexSubset}}
      }}{\CouplingCapra}\np{\dual} 
 \intertext{as conjugacies, being dualities, turn infima into suprema}
    &= 
      \sup_{\IndexSubset \subset \IndexSet }
      \Bp{ 
      \SFM{ \delta_{\Gamma\cap\np{\SuppLevelCurve{\SupportMapping}{\IndexSubset}}} }{\CouplingCapra}\np{\dual} 
      \LowPlus \bp{-\FunctionBigF\np{\IndexSubset}} }  
      \tag{by property of conjugacies}
    \\
    &= 
      \sup_{\IndexSubset \subset \IndexSet }
      \Bp{ 
      \sigma_{ \normalized\np{\Gamma\cap\SuppLevelCurve{\SupportMapping}{\IndexSubset}}}\np{\dual} 
      \LowPlus \bp{-\FunctionBigF\np{\IndexSubset}} }  
      \tag{as \( \SFM{ \delta_{S} }{\CouplingCapra}
      = \sigma_{ \normalized\np{S}}  \) for any subset $S\subset \RR^d$,
      by~\eqref{eq:one-sided_linear_Fenchel-Moreau_characteristic}}
    \\
    &=
      \sup_{\IndexSubset \subset \IndexSet }
      \Bp{ 
      \sigma_{ \Gamma \cap \normalized\np{\SuppLevelCurve{\SupportMapping}{\IndexSubset}}}\np{\dual} 
      \LowPlus \bp{-\FunctionBigF\np{\IndexSubset}} }  
      \tag{because $\normalized\np{\Gamma \cap \SuppLevelCurve{\SupportMapping}{\IndexSubset}}
  =\Gamma \cap \normalized\np{\SuppLevelCurve{\SupportMapping}{\IndexSubset}}$ by assumption}
    \\
    &=
      \sup_{\IndexSubset \subset \IndexSet } \Bp{ 
      \sigma_{ \Gamma \cap \bp{\{0\} \cup \np{ \TripleNormSphere
      \cap\SuppLevelCurve{\SupportMapping}{\IndexSubset} } } }\np{\dual}  
      \LowPlus \bp{-\FunctionBigF\np{\IndexSubset}} }
      \tag{as \(
      \normalized\np{\SuppLevelCurve{\SupportMapping}{\IndexSubset}}=
      \{0\} \cup \bp{ \TripleNormSphere \cap \SuppLevelCurve{\SupportMapping}{\IndexSubset}} \)
      by~\eqref{eq:normalization_mapping}}
    \\
    &=
      \sup_{\IndexSubset \subset \IndexSet } \Bp{ 
      \sup\bp{0, \sigma_{ \Gamma \cap \TripleNormSphere \cap\SuppLevelCurve{\SupportMapping}{\IndexSubset} }}\np{\dual}  
      \LowPlus \bp{-\FunctionBigF\np{\IndexSubset}} }
      \intertext{as 
      \( \Gamma \cap \bp{\{0\} \cup \np{\TripleNormSphere
      \cap\SuppLevelCurve{\SupportMapping}{\IndexSubset} }}
      = \{0\} \cup \bp{ \Gamma \cap \TripleNormSphere \cap
      \SuppLevelCurve{\SupportMapping}{\IndexSubset}} \) since \( 0 \in \Gamma
      \) by assumption}
    &= 
      \sup_{\IndexSubset \subset \IndexSet } \Bp{ 
      \sigma_{ \TripleNormSphere\cap\Gamma \cap \SuppLevelCurve{\SupportMapping}{\IndexSubset}} \np{\dual} 
      \LowPlus \bp{-\FunctionBigF\np{\IndexSubset}} }
      \tag{as \( \sigma_{ \TripleNormSphere \cap\SuppLevelCurve{\SupportMapping}{\IndexSubset}} \geq 0 \) }
      \eqfinp
  \end{align*}
  This ends the proof.
\end{proof}

\subsection{\Capra-subdifferentials of \FSM}
\label{Capra-subdifferentials_related_to_the_function_of_the_support_mapping}

We now show that \FSM\ 
display \Capra-subdifferentials, as in~\eqref{eq:Capra-subdifferential_b}, that are related to 
the family of dual local-coordinate-$\IndexSubset$ norms, in Definition~\ref{de:coordinate_norm},
as follows. 
For this purpose, 
we recall that the \emph{normal cone}~$\NORMAL_{\Convex}(\primal)$ 
to the (nonempty) closed convex subset~${\Convex} \subset \RR^d $
at~$\primal \in \Convex$ is the closed convex cone defined by 
\cite[p.136]{Hiriart-Urruty-Lemarechal-I:1993}
\begin{equation}
  \NORMAL_{\Convex}(\primal) =
  \bset{ \dual \in \RR^d }%
  {
    \proscal{\primal'-\primal}{\dual} \leq 0 \eqsepv 
    \forall \primal' \in \Convex
  }
  \eqfinp
  \label{eq:normal_cone}
\end{equation}

\begin{proposition}
  \label{pr:support_mapping_subdifferential_varphi}
  Let $\TripleNorm{\cdot}$ be the source norm with 
associated coupling $\CouplingCapra$, as in Definition~\ref{de:coupling_CAPRA},
and with associated families
  \( \sequence{\CoordinateNorm{\TripleNorm{\cdot}}{\IndexSubset}}{\IndexSubset\subset \IndexSet} \)
  of local-coordinate-$\IndexSubset$ norms 
and 
  \( \sequence{\CoordinateNormDual{\TripleNorm{\cdot}}{\IndexSubset}}{\IndexSubset\subset \IndexSet} \)
  of dual local-coordinate-$\IndexSubset$ norms, as in Definition~\ref{de:coordinate_norm}. 
  
  We consider a set function \( \FunctionBigF : 2^{\IndexSet} \to \barRR \)
  and a vector $\primal \in \RR^d$.
  \begin{subequations}
    \begin{itemize}
    \item 
      The \Capra-subdifferential, as in~\eqref{eq:Capra-subdifferential_zero}, of the function
      \( \FunctionBigF \circ \SupportMapping \) at~$\primal=0$ is given by 
      \begin{equation}
        \partial_{\CouplingCapra}\np{ \FunctionBigF \circ \SupportMapping}\np{0}
        = \bigcap_{ \IndexSubset \subset \IndexSet } 
        \bc{ \FunctionBigF\np{\IndexSubset} \UppPlus \bp{-\FunctionBigF\np{\emptyset} } } 
        \CoordinateNormDual{\TripleNormBall}{\IndexSubset} 
        \eqfinv
        \label{eq:support_mapping_subdifferential_at0}
      \end{equation}
      where, by convention, 
      \( \lambda \CoordinateNormDual{\TripleNormBall}{\IndexSubset} =\emptyset \),
      for any \( \lambda \in [-\infty,0[ \), and 
      \( +\infty\CoordinateNormDual{\TripleNormBall}{\IndexSubset} =\RR^d \).
    \item 
      The \Capra-subdifferential, as in~\eqref{eq:Capra-subdifferential_neq_zero}, of the function
      \( \FunctionBigF \circ \SupportMapping \) at~$\primal\not=0$ is given by
      the following cases
      \begin{itemize}
      \item 
        if \( L=\SupportMapping\np{\primal} \not=\emptyset \) 
        and either \( \FunctionBigF\np{L}=-\infty \)
        or \( \FunctionBigF \equiv +\infty \), 
        then \( \partial_{\CouplingCapra}\np{ \FunctionBigF \circ \SupportMapping}\np{\primal} =\RR^d\),
      \item 
        if \( L=\SupportMapping\np{\primal} \not=\emptyset  \) 
        and \( \FunctionBigF\np{L}=+\infty \) and there exists 
        \( \IndexSubset \subset \IndexSet \) such that 
        \( \FunctionBigF\np{\IndexSubset} \not= +\infty \), then 
        \( \partial_{\CouplingCapra}\np{ \FunctionBigF \circ \SupportMapping}\np{\primal} =\emptyset \),
      \item 
        if \( L=\SupportMapping\np{\primal} \not=\emptyset \) and \( -\infty < \FunctionBigF\np{L} < +\infty \), then 
        \begin{equation}
          \dual \in \partial_{\CouplingCapra}\np{ \FunctionBigF \circ \SupportMapping}\np{\primal} 
          \iff 
          \begin{cases}
            \dual \in 
            \NORMAL_{ \CoordinateNorm{\TripleNormBall}{L} }
            \np{\frac{ \primal }{ \CoordinateNorm{ \TripleNorm{\primal} }{L} } }
               \mtext{ and }
            \\
            L \in \argmax_{\IndexSubset\subset \IndexSet} 
            \bc{
              \CoordinateNormDual{\TripleNorm{\dual}}{\IndexSubset}-\FunctionBigF\np{\IndexSubset}
            }
          \eqfinp 
          \end{cases}
          \label{eq:support_mapping_subdifferential}
        \end{equation}
        %
      \end{itemize}
    \end{itemize}
  \end{subequations}
\end{proposition}

\begin{proof}
  We have
  \begin{align*}
    \dual \in 
    \partial_{\CouplingCapra}\np{ \FunctionBigF \circ \SupportMapping}\np{\primal} 
    &\iff \SFM{ \np{ \FunctionBigF \circ \SupportMapping} }{\CouplingCapra}\np{\dual} 
      = \CouplingCapra\np{\primal, \dual} 
      \LowPlus \bp{ -\np{ \FunctionBigF \circ \SupportMapping}\np{\primal} }
      \tag{by definition~\eqref{eq:Capra-subdifferential_b} of the 
      \Capra-subdifferential }
    \\
    &\iff 
      \sup_{\IndexSubset\subset \IndexSet} \Bc{ \CoordinateNormDual{\TripleNorm{\dual}}{\IndexSubset} -\FunctionBigF\np{\IndexSubset} }  
      = \CouplingCapra\np{\primal, \dual} 
      \LowPlus \bp{ -\np{ \FunctionBigF \circ \SupportMapping}\np{\primal} }
      \tag{as 
      \( \SFM{ \np{ \FunctionBigF \circ \SupportMapping} }{\CouplingCapra}\np{\dual} 
      = \sup_{\IndexSubset\subset \IndexSet} \Bc{ \CoordinateNormDual{\TripleNorm{\dual}}{\IndexSubset} -\FunctionBigF\np{\IndexSubset} } \) 
      by~\eqref{eq:CAPRA_conjugate_function_of_the_Support}
      }
    \\
    &\iff 
      \Bp{ \primal=0 \mtext{ and } 
      \sup_{\IndexSubset\subset \IndexSet} \Bc{
      \CoordinateNormDual{\TripleNorm{\dual}}{\IndexSubset}
      -\FunctionBigF\np{\IndexSubset} } =-\FunctionBigF\np{\emptyset} }
     \tag{by definition~\eqref{eq:coupling_CAPRA} of \(
      \CouplingCapra\np{\primal, \dual} \) when \( \primal=0 \) }
    \\
    & \mtext{ or } 
      \Bp{ \primal \neq 0 \mtext{ and } 
      \sup_{\IndexSubset\subset \IndexSet} \Bc{ \CoordinateNormDual{\TripleNorm{\dual}}{\IndexSubset} -\FunctionBigF\np{\IndexSubset} }
      = \frac{ \proscal{\primal}{\dual} }{ \TripleNorm{\primal} } 
      - \FunctionBigF\bp{\SupportMapping\np{\primal}} 
      } 
      \tag{by definition~\eqref{eq:coupling_CAPRA} of \(
      \CouplingCapra\np{\primal, \dual} \) when \( \primal \not=0 \) }
\eqfinp 
  \end{align*}
    Therefore, on the one hand, we obtain that
    \begin{align*}
      \dual \in \partial_{\CouplingCapra}\np{ \FunctionBigF \circ \SupportMapping}\np{0} 
      &\iff 
        \CoordinateNormDual{\TripleNorm{\dual}}{\IndexSubset}
        -\FunctionBigF\np{\IndexSubset} \leq -\FunctionBigF\np{\emptyset} \eqsepv
        \forall \IndexSubset\subset \IndexSet
        \tag{as \( \CoordinateNormDual{\TripleNorm{\dual}}{\emptyset}=0 \) 
        by convention} 
      \\
      &\iff 
        \CoordinateNormDual{\TripleNorm{\dual}}{\IndexSubset}
        \leq \FunctionBigF\np{\IndexSubset} \UppPlus \bp{-\FunctionBigF\np{\emptyset} } \eqsepv
        \forall \IndexSubset \subset \IndexSet
        \tag{because \( u \LowPlus (-v) \leq w \iff u \leq v \UppPlus w \) for
        all $u$, $v$, $w$ in $\barRR$ \cite{Moreau:1970}}
      \\
      &\iff 
        \dual \in \bigcap_{ \IndexSubset \subset \IndexSet} 
        \bc{ \FunctionBigF\np{\IndexSubset} \UppPlus \bp{-\FunctionBigF\np{\emptyset} } } 
        \CoordinateNormDual{\TripleNormBall}{\IndexSubset} 
        \eqfinv
    \end{align*}
    where, by convention 
    \( \lambda \CoordinateNormDual{\TripleNormBall}{\IndexSubset} =\emptyset \),
    for any \( \lambda \in [-\infty,0[ \), and 
    \( +\infty\CoordinateNormDual{\TripleNormBall}{\IndexSubset} =\RR^d \).

    On the other hand, when \( \primal \neq 0 \), we get
    \begin{equation}
      \dual \in \partial_{\CouplingCapra}\np{ \FunctionBigF \circ \SupportMapping}\np{\primal} 
      \iff 
      \sup_{\IndexSubset\subset \IndexSet} \bc{ \CoordinateNormDual{\TripleNorm{\dual}}{\IndexSubset}-\FunctionBigF\np{\IndexSubset} }=
      \frac{ \proscal{\primal}{\dual} }{ \TripleNorm{\primal} } 
      - \FunctionBigF\bp{\SupportMapping\np{\primal}} 
      \eqfinp
      \label{eq:support_mapping_subdifferential_proof_neq_0}
    \end{equation}
  We now establish necessary and sufficient conditions for \( \dual \) to belong to
  \( \partial_{\CouplingCapra}\np{ \FunctionBigF \circ \SupportMapping}\np{\primal} \) when \( \primal \neq 0 \).
  For this purpose, we consider \( \primal \in \RR^d\backslash\{0\} \), and we denote
  \( L = \Support{\primal} \).
  We have
  \begin{align*}
    \dual
    &\in 
      \partial_{\CouplingCapra}\np{ \FunctionBigF \circ \SupportMapping}\np{\primal} 
    \\
    &\iff
      \sup_{\IndexSubset\subset \IndexSet} \bc{ \CoordinateNormDual{\TripleNorm{\dual}}{\IndexSubset}-\FunctionBigF\np{\IndexSubset} }=
      \frac{ \proscal{\primal}{\dual} }{ \TripleNorm{\primal} } -\FunctionBigF\np{L}
      \tag{by~\eqref{eq:support_mapping_subdifferential_proof_neq_0} with \( \SupportMapping\np{\primal}=L \)}
    \\
    &\iff 
      \CoordinateNormDual{\TripleNorm{\dual}}{L} -\FunctionBigF\np{L} \leq
      \sup_{\IndexSubset\subset \IndexSet} \bc{ \CoordinateNormDual{\TripleNorm{\dual}}{\IndexSubset}-\FunctionBigF\np{\IndexSubset} }=
      \frac{ \proscal{\primal}{\dual} }{ \TripleNorm{\primal} } -\FunctionBigF\np{L} 
    \\
    &\iff 
      \TripleNorm{\dual_L}_{L,\star} -\FunctionBigF\np{L} \leq 
      \CoordinateNormDual{\TripleNorm{\dual}}{L} -\FunctionBigF\np{L} \leq
      \sup_{\IndexSubset\subset \IndexSet} \bc{ \CoordinateNormDual{\TripleNorm{\dual}}{\IndexSubset}-\FunctionBigF\np{\IndexSubset} }=
      \frac{ \proscal{\primal}{\dual} }{ \TripleNorm{\primal} } -\FunctionBigF\np{L} 
      \intertext{ as \( \TripleNorm{\dual_L}_{L,\star} \leq
      \CoordinateNormDual{\TripleNorm{\dual}}{L} \) 
      by the expression~\eqref{eq:dual_coordinate_norm_definition} of the dual local-coordinate-$L$ norm
      \( \CoordinateNormDual{\TripleNorm{\dual}}{L} \)
      }
    &\iff 
      \TripleNorm{\dual_L}_{L,\star} -\FunctionBigF\np{L} \leq 
      \CoordinateNormDual{\TripleNorm{\dual}}{l} -\FunctionBigF\np{L} \leq
      \sup_{\IndexSubset\subset \IndexSet} \bc{ \CoordinateNormDual{\TripleNorm{\dual}}{\IndexSubset}-\FunctionBigF\np{\IndexSubset} }=
      \frac{ \proscal{\primal}{\dual} }{ \TripleNorm{\primal} } -\FunctionBigF\np{L} \leq
      \TripleNorm{\dual_L}_{L,\star} -\FunctionBigF\np{L} 
      \intertext{as we have \( \frac{ \proscal{\primal}{\dual} }{ \TripleNorm{\primal} } 
      = \frac{ \proscal{\primal_L}{\dual_L} }{ \TripleNorm{\primal_L} } 
      = \frac{ \proscal{\primal_L}{\dual_L} }{ \TripleNorm{\primal_L}_L } 
      \leq \TripleNorm{\dual_L}_{L,\star} \) since \( \primal=\primal_L \) and 
      by property~\eqref{eq:norm_dual_norm_inequality} of the dual norm} 
    &\iff 
      \TripleNorm{\dual_L}_{L,\star} -\FunctionBigF\np{L} = 
      \CoordinateNormDual{\TripleNorm{\dual}}{L} -\FunctionBigF\np{L} =
      \sup_{\IndexSubset\subset \IndexSet} \bc{ \CoordinateNormDual{\TripleNorm{\dual}}{\IndexSubset}-\FunctionBigF\np{\IndexSubset} }=
      \frac{ \proscal{\primal}{\dual} }{ \TripleNorm{\primal} } -\FunctionBigF\np{L} 
      \tag{as all terms in the inequalities are necessarily equal }
    \\
    &\iff
      \begin{cases}
        \text{either } \FunctionBigF\np{L}=-\infty 
        \\
        \text{or } \bp{ \FunctionBigF\np{L}=+\infty \text{ and }
          \FunctionBigF\np{\IndexSubset}=+\infty \eqsepv \forall \IndexSubset\subset \IndexSet }
        \\[2mm]
        \text{or }
        \Big( -\infty < \FunctionBigF\np{L} < +\infty \mtext{ and } 
        \\ \qquad 
        \TripleNorm{\dual_L}_{L,\star} = 
        \CoordinateNormDual{\TripleNorm{\dual}}{L} =
        \frac{ \proscal{\primal}{\dual} }{ \TripleNorm{\primal} }
        \mtext{ and } \CoordinateNormDual{\TripleNorm{\dual}}{L}-\FunctionBigF\np{L} = 
        \sup_{\IndexSubset\subset \IndexSet} \bc{ \CoordinateNormDual{\TripleNorm{\dual}}{\IndexSubset}-\FunctionBigF\np{\IndexSubset} }
        \Big) \eqfinp
      \end{cases}
  \end{align*}
  Let us make a brief insert and notice that 
  \begin{align*}
    \primal=\primal_L &\eqsepv \SupportMapping\np{\primal}=L\not=\emptyset \eqsepv
                        \proscal{\primal}{\dual} =
                        \TripleNorm{\primal} \times
                        \CoordinateNormDual{\TripleNorm{\dual}}{L} 
    \\
    \implies & \quad
                  \SupportMapping\np{\primal}=L \not=\emptyset \eqsepv
                  \proscal{\primal_L}{\dual_L} =
                  \TripleNorm{\primal_L}_L \times
                  \CoordinateNormDual{\TripleNorm{\dual}}{L} 
                  \tag{by~\eqref{eq:orthogonal_projection_self-dual} }
    \\
    \implies & \quad
                  \SupportMapping\np{\primal}=L \not=\emptyset \eqsepv
                  \TripleNorm{\primal_L}_L \times
                  \CoordinateNormDual{\TripleNorm{\dual}}{L} 
                  \leq
                  \TripleNorm{\primal_L}_L \times
                  \TripleNorm{\dual_L}_{L,\star} 
                  \tag{by~\eqref{eq:norm_dual_norm_inequality}}
    \\
    \implies & \quad
                  \CoordinateNormDual{\TripleNorm{\dual}}{L} 
                  \leq
                  \TripleNorm{\dual_L}_{L,\star} 
                  \tag{since $\TripleNorm{\primal_L}_L=\TripleNorm{\primal_L} \neq 0$ because 
\( \SupportMapping\np{\primal}=L \not=\emptyset \)}
    \\
    \implies & \quad
                  \CoordinateNormDual{\TripleNorm{\dual}}{L} 
                  = \TripleNorm{\dual_L}_{L,\star}
  \end{align*}              
  as \( \TripleNorm{\dual_L}_{L,\star} \leq
  \CoordinateNormDual{\TripleNorm{\dual}}{L} \) 
      by the expression~\eqref{eq:dual_coordinate_norm_definition} of the dual local-coordinate-$L$ norm
      \( \CoordinateNormDual{\TripleNorm{\dual}}{L} \). 
\smallskip

  Now, let us go back to the equivalences regarding
  \(  \dual \in \partial_{\CouplingCapra}
  \np{ \FunctionBigF \circ \SupportMapping}\np{\primal} \).
  Focusing on the case where \( \SupportMapping\np{\primal}=L \not=\emptyset \)
and \( -\infty < \FunctionBigF\np{L} < +\infty \), 
  we have 
  \begin{align*}
    \dual \in \partial_{\CouplingCapra}
    &\np{ \FunctionBigF \circ \SupportMapping}\np{\primal}
      \Leftrightarrow 
      \TripleNorm{\dual_L}_{L,\star} = 
      \CoordinateNormDual{\TripleNorm{\dual}}{L} =
      \frac{ \proscal{\primal}{\dual} }{ \TripleNorm{\primal} }
      \mtext{ and } 
      L \in \argmax_{\IndexSubset\subset \IndexSet} 
\bc{ \CoordinateNormDual{\TripleNorm{\dual}}{\IndexSubset}-\FunctionBigF\np{\IndexSubset} }
    \\ 
    &\Leftrightarrow 
      \TripleNorm{\dual_L}_{L,\star} = 
      \CoordinateNormDual{\TripleNorm{\dual}}{L} 
      \mtext{ and } 
      \proscal{\primal}{\dual} =
      \TripleNorm{\primal} \times
      \CoordinateNormDual{\TripleNorm{\dual}}{L} 
      \mtext{ and } 
      L \in \argmax_{\IndexSubset\subset \IndexSet} \bc{ \CoordinateNormDual{\TripleNorm{\dual}}{\IndexSubset}-\FunctionBigF\np{\IndexSubset} }
    \\ 
    &\Leftrightarrow 
      \proscal{\primal}{\dual} =
      \TripleNorm{\primal} \times
      \CoordinateNormDual{\TripleNorm{\dual}}{L} 
      \mtext{ and } 
      L \in \argmax_{\IndexSubset\subset \IndexSet} \bc{ \CoordinateNormDual{\TripleNorm{\dual}}{\IndexSubset}-\FunctionBigF\np{\IndexSubset} }
      \intertext{as just established in the insert right above}
    &\Leftrightarrow 
      \proscal{\primal}{\dual} =
      \CoordinateNorm{\TripleNorm{\primal}}{L} \times
      \CoordinateNormDual{\TripleNorm{\dual}}{L} 
      \mtext{ and } 
      L \in \argmax_{\IndexSubset\subset \IndexSet} \bc{ \CoordinateNormDual{\TripleNorm{\dual}}{\IndexSubset}-\FunctionBigF\np{\IndexSubset} }
      \tag{as \( \SupportMapping\np{\primal}=L \implies \primal \in \FlatRR_L
     \implies \TripleNorm{\primal}=
      \CoordinateNorm{\TripleNorm{\primal}}{L} \)
      by~\eqref{eq:coordinate_norm_graded_b} }
    \\
    &\Leftrightarrow 
      \dual \in \NORMAL_{ \CoordinateNorm{\TripleNormBall}{L} }
      \np{ \frac{ \primal }{ \CoordinateNorm{\TripleNorm{\primal}}{L} } }
      \mtext{ and } 
      L \in \argmax_{\IndexSubset\subset \IndexSet} \bc{ \CoordinateNormDual{\TripleNorm{\dual}}{\IndexSubset}-\FunctionBigF\np{\IndexSubset} }
      %
  \end{align*}
    by the equivalence
  \( \proscal{\primal}{\dual} =
      \CoordinateNorm{\TripleNorm{\primal}}{l} \times
      \CoordinateNormDual{\TripleNorm{\dual}}{l} \iff
      \dual \in \NORMAL_{ \CoordinateNorm{\TripleNormBall}{l} }
      \np{ \frac{ \primal }{ \CoordinateNorm{\TripleNorm{\primal}}{l} } } \).
      \medskip
      
  This ends the proof. 
\end{proof}

We now show that, when both the source norm and its dual norm are orthant-strictly
monotonic, the \Capra-subdifferential of a nondecreasing finite-valued \FSM\ is nonempty.

\begin{proposition}
  \label{pr:nonempty_subdifferential}
  Let $\TripleNorm{\cdot}$ be the source norm with 
associated coupling $\CouplingCapra$, as in Definition~\ref{de:coupling_CAPRA},
and with associated families
  \( \sequence{\CoordinateNorm{\TripleNorm{\cdot}}{\IndexSubset}}{\IndexSubset\subset \IndexSet} \)
  of local-coordinate-$\IndexSubset$ norms 
and 
  \( \sequence{\CoordinateNormDual{\TripleNorm{\cdot}}{\IndexSubset}}{\IndexSubset\subset \IndexSet} \)
  of dual local-coordinate-$\IndexSubset$ norms, as in Definition~\ref{de:coordinate_norm}. 

Suppose that both the norm $\TripleNorm{\cdot}$ and the dual norm $\TripleNormDual{\cdot}$
  are orthant-strictly monotonic.
  Let \( \FunctionBigF : 2^{\IndexSet} \to \RR \) be a nondecreasing
  finite-valued set function.
  Then, we have 
  \begin{equation*}
    \partial_{\CouplingCapra}\np{ \FunctionBigF \circ \SupportMapping }\np{\primal} \neq \emptyset
    \eqsepv \forall \primal \in \RR^d 
    \eqfinp 
  \end{equation*}
  More precisely, 
  \( \partial_{\CouplingCapra}\np{ \FunctionBigF \circ \SupportMapping }\np{0}
  = \bigcap_{ \IndexSubset \subset \IndexSet } \bc{ \FunctionBigF\np{\IndexSubset} \UppPlus \bp{-\FunctionBigF\np{\emptyset}}}
  \CoordinateNormDual{\TripleNormBall}{\IndexSubset} \neq \emptyset \)
  and, when $\primal\not=0$, for any \( \dual \in \RR^d \) such that
  \( \Support{\dual} = \Support{\primal} \),
  and that \( \proscal{\primal}{\dual} =
  \TripleNorm{\primal} \times \TripleNormDual{\dual} \),
  we have that \( \lambda \dual \in \partial_{\CouplingCapra}\np{ \FunctionBigF \circ \SupportMapping }\np{\primal} \)
  for $\lambda >0$ large enough.
\end{proposition}

\begin{proof}
  %
  When \( \primal=0 \), we have,
  by~\eqref{eq:support_mapping_subdifferential_at0}, that 
  \( \partial_{\CouplingCapra}\np{ \FunctionBigF \circ \SupportMapping }\np{0}
  = \bigcap_{ \IndexSubset \subset \IndexSet } \bc{ \FunctionBigF\np{\IndexSubset} \UppPlus \bp{-\FunctionBigF\np{\emptyset} }}
  \CoordinateNormDual{\TripleNormBall}{\IndexSubset} \)
  because \( \FunctionBigF\np{\IndexSubset} \UppPlus \bp{-\FunctionBigF\np{\emptyset}} 
  = \FunctionBigF\np{\IndexSubset} -\FunctionBigF\np{\emptyset} \) since the function~$\FunctionBigF$
  takes finite values. 
  The set \( \bigcap_{ \IndexSubset \subset \IndexSet } \bc{ \FunctionBigF\np{\IndexSubset} \UppPlus \bp{-\FunctionBigF\np{\emptyset} }}
  \CoordinateNormDual{\TripleNormBall}{\IndexSubset} \) is nonempty
  (it contains~$0$), because
  \( \FunctionBigF\np{\IndexSubset} -\FunctionBigF\np{\emptyset} \geq 0 \)
  for \( \IndexSubset \subset \IndexSet \)
  since \( \FunctionBigF : 2^{\IndexSet} \to \RR \) is nondecreasing.
  \medskip

  From now on, we consider $\primal \in \RR^d\backslash\{0\} $ such that 
  $\Support{\primal}=L \subset \IndexSet$, where $L\not=\emptyset$ since $x\not=0$.
  We will use the
  characterization~\eqref{eq:support_mapping_subdifferential}
  of the subdifferential~\( \partial_{\CouplingCapra}\np{ \FunctionBigF \circ \SupportMapping }\np{\primal} \).

  Since the norm $\TripleNorm{\cdot}$ is orthant-strictly monotonic,
  by Proposition~\ref{pr:orthant-strictly_monotonic}
  (equivalence between Item~\ref{it:OSM} and Item~\ref{it:SDC}),
  there exists a vector
  \( \dual \in \RR^d \) such that 
  \begin{subequations}
    \begin{align}
      L= &\Support{\primal} = \Support{\dual} \not= \emptyset
           \eqfinv 
           \label{eq:conjugate_point_proof_a}
      \\
         &\text{ and }
           \proscal{\primal}{\dual}
           = 
           \TripleNorm{\primal} \times \TripleNormDual{\dual} 
           \eqfinp
           \label{eq:conjugate_point_proof_b}
    \end{align}
    \label{eq:conjugate_point_proof}
  \end{subequations}
  Since both the norm $\TripleNorm{\cdot}$ and the dual norm $\TripleNormDual{\cdot}$
  are orthant-strictly monotonic,
  using Proposition~\ref{pr:properties-of-dual-ccord-IndexSubset-under-orth-mon} we have that\footnote{%
    $\IndexSubset \subsetneq L$ stands for $\IndexSubset\subset L$ and $\IndexSubset\not=L$.}
  \begin{equation}
\IndexSubset\subsetneq L \implies
    \CoordinateNormDual{\TripleNorm{\dual}}{\IndexSubset}  < 
    \CoordinateNormDual{\TripleNorm{\dual}}{L} =
    \CoordinateNormDual{\TripleNorm{\dual}}{\IndexSet} = \TripleNormDual{\dual}
    \eqfinp
    \label{eq:level_curve_l0_characterization}
  \end{equation}
  \medskip

  \noindent $\bullet$ First, we are going to establish that we have
  \( \dual \in 
  \NORMAL_{ \CoordinateNorm{\TripleNormBall}{L} }
  \np{\frac{ \primal }{ \CoordinateNorm{ \TripleNorm{\primal} }{L} } } \), that
  is, the first of the two conditions in the characterization~\eqref{eq:support_mapping_subdifferential}
  of the subdifferential~\( \partial_{\CouplingCapra}\np{ \FunctionBigF \circ \SupportMapping }\np{\primal} \). 

  On the one hand, by~\eqref{eq:level_curve_l0_characterization},
  we have that \( \TripleNormDual{\dual} =
  \CoordinateNormDual{\TripleNorm{\dual}}{L} \). 
  On the other hand, we have 
  \( \TripleNorm{\primal} = \CoordinateNorm{\TripleNorm{\primal}}{L} \)
  by~\eqref{eq:coordinate_norm_graded_b} since $\Support{\primal}=L$.
  Hence, from~\eqref{eq:conjugate_point_proof_b},
  we get that \( \proscal{\primal}{\dual} =
  \CoordinateNorm{\TripleNorm{\primal}}{L} 
  \times \CoordinateNormDual{\TripleNorm{\dual}}{L} \),
  from which we obtain 
  \( \dual \in \NORMAL_{ \CoordinateNorm{\TripleNormBall}{L} }
  \np{\frac{ \primal }{ \CoordinateNorm{ \TripleNorm{\primal} }{L} } } \)
    by the equivalence
    \( \proscal{\primal}{\dual} =
  \CoordinateNorm{\TripleNorm{\primal}}{L} 
  \times \CoordinateNormDual{\TripleNorm{\dual}}{L}
  \iff
      \dual \in \NORMAL_{ \CoordinateNorm{\TripleNormBall}{L} }
      \np{\frac{ \primal }{ \CoordinateNorm{ \TripleNorm{\primal} }{L} } } \)
  as \( \primal \neq 0 \).
  %
  To close this part, notice that, for all $\lambda > 0 $,
  we have that 
  \(  \lambda\dual \in 
  \NORMAL_{ \CoordinateNorm{\TripleNormBall}{L} }
  \np{\frac{ \primal }{ \CoordinateNorm{ \TripleNorm{\primal} }{L} } } \), 
  because this last set is a (normal) cone. 
  \medskip

  \noindent $\bullet$ Second, we turn to proving the second of the two conditions
  in the characterization~\eqref{eq:support_mapping_subdifferential}
  of the subdifferential~\( \partial_{\CouplingCapra}\np{ \FunctionBigF \circ \SupportMapping }\np{\primal} \).
  More precisely, we are going to show that, for $\lambda$ large enough,
  \( \CoordinateNormDual{\TripleNorm{\lambda\dual}}{L}-\FunctionBigF\np{L} = 
  \sup_{\IndexSubset\subset \IndexSet} \bc{ \CoordinateNormDual{\TripleNorm{\lambda\dual}}{\IndexSubset}-\FunctionBigF\np{\IndexSubset} }
  \).
  For this purpose, we consider the mapping $\psi: ]0,+\infty[ \to \RR$ defined by
  \begin{equation*}
    \psi(\lambda) =
    \CoordinateNormDual{\TripleNorm{\lambda\dual}}{L}-\FunctionBigF\np{L} -
    \sup_{\IndexSubset\subset \IndexSet} \bc{ \CoordinateNormDual{\TripleNorm{\lambda\dual}}{\IndexSubset}-\FunctionBigF\np{\IndexSubset} }
    \eqsepv \forall \lambda > 0 
    \eqfinv
  \end{equation*}
  and we show that \( \psi(\lambda) =0 \) for $\lambda$ large enough.
  We have
  \begin{align*}
    \psi(\lambda) 
    &=  \inf_{\IndexSubset\subset \IndexSet} \phi_\IndexSubset(\lambda)
      \text{ with }
      \phi_\IndexSubset(\lambda) = \Bp{ \lambda \bp{ \CoordinateNormDual{\TripleNorm{\dual}}{L} -
      \CoordinateNormDual{\TripleNorm{\dual}}{\IndexSubset} } +
      \FunctionBigF\np{\IndexSubset}-\FunctionBigF\np{L} } 
      \eqfinp
  \end{align*}
  Using
  Proposition~\ref{pr:properties-of-dual-ccord-IndexSubset-under-orth-mon},
  we get the following.
  \begin{itemize}
   \item 
    If $\IndexSubset\cap L = L$ (which is equivalent to $L \subset \IndexSubset$), we get that
    $\bp{ \CoordinateNormDual{\TripleNorm{\dual}}{L} -
      \CoordinateNormDual{\TripleNorm{\dual}}{\IndexSubset} } =0$,
    which implies that 
    $\phi_\IndexSubset(\lambda) =
    \FunctionBigF\np{\IndexSubset}-\FunctionBigF\np{L} \geq 0$ 
    since $\FunctionBigF$ is assumed nondecreasing.
  \item 
    If $\IndexSubset\cap L \subsetneq L$ (which  is equivalent to $L \not\subset \IndexSubset$), 
we get that $\bp{ \CoordinateNormDual{\TripleNorm{\dual}}{L} -
      \CoordinateNormDual{\TripleNorm{\dual}}{\IndexSubset} } >0$, 
    which implies that $\phi_\IndexSubset(\lambda)$ goes to infinity with $\lambda$.
 \end{itemize}

  Wrapping up the above results, we have shown that, 
  for any vector \( \dual \in \RR^d \) such that 
  \( \Support{\dual} = \Support{\primal} \),
  and that \( \proscal{\primal}{\dual} =
  \TripleNorm{\primal} \times \TripleNormDual{\dual} \),
  then, for $\lambda >0$ large enough,
  \( \lambda\dual \) satisfies the two conditions 
  in the characterization~\eqref{eq:support_mapping_subdifferential}
  of the subdifferential~\( 
  \partial_{\CouplingCapra}\np{ \FunctionBigF \circ \SupportMapping }\np{\primal} \).
  \medskip

  This ends the proof. 
\end{proof}

\subsection{\Capra-biconjugates of \FSM} 
\label{CAPRAC_biconjugates_related_to_the_pseudo_norm}

With the Fenchel conjugacy, we calculate that 
\( \LFMbi{ \delta_{ \SuppLevelSet{\SupportMapping}{\IndexSubset}  } }
= 0 \) --- 
where \( \IndexSubset \neq \emptyset \), and
\( \delta_{ \SuppLevelSet{\SupportMapping}{\IndexSubset}  } \) is the characteristic function
of the level sets~\eqref{eq:support_mapping_level_set}
as defined in~\eqref{eq:characteristic_function} --- 
and, more generally, that 
\( \LFMbi{ \np{\FunctionBigF\circ\SupportMapping } }= 
\inf \ba{ \delta_{\{0\}} \LowPlus \FunctionBigF\np{\emptyset}, 
  \inf_{\IndexSubset \neq
    \emptyset}\FunctionBigF\np{\IndexSubset} } \),
with any set function~$\FunctionBigF$ as in~\eqref{eq:FunctionBigF}.
Hence, the Fenchel conjugacy is not suitable
to handle  \FSM.

By contrast, we will now show that \FSM\ are related to 
the families of local-coordinate-$\IndexSubset$ norms and dual 
local-coordinate-$\IndexSubset$ norms, in Definition~\ref{de:coordinate_norm},
by the following \Capra-biconjugacy formulas.

In the sequel we will use the notation~$\RR_{\IndexSet}$ to refer to 
the subset $\prod_{\IndexSubset\subset \IndexSet} \FlatRR_{\IndexSubset}$ of
\( \np{\RR^d}^{2^\IndexSet} \),
and the notation~\( \Delta_{\IndexSet} \) to refer to the simplex of~$\RR_{\IndexSet}$
(which can be identified with the simplex of $\RR^v$, where $v=\cardinal{2^\IndexSet}$).
Hence, an element $z\in \RR_{\IndexSet}$ is a family
$\sequence{z_{(\IndexSubset)}}{\IndexSubset \subset \IndexSet}$
where $z_{(\IndexSubset)}\in \FlatRR_{\IndexSubset}$ for
$\IndexSubset\subset \IndexSet$
(with the convention $z_{(\emptyset)}=0$), 
and an element $\lambda \in \Delta_{\IndexSet}$ is a family
$\sequence{\lambda_\IndexSubset}{\IndexSubset \subset \IndexSet}$ of nonnegative
real numbers such that $\sum_{\IndexSubset\subset \IndexSet} \lambda_\IndexSubset=1$.
We will also use $\RR_{\IndexSet}^{\backslash\emptyset}$ to denote the projection
of~$\RR_{\IndexSet}$ when we get rid of the element $\lambda_{\emptyset}$, and we define {$\SimplexWithoutEmptySet$}
in the same way, that is, 
\begin{equation}
  \SimplexWithoutEmptySet = 
  \Bset{ \sequence{\lambda_\IndexSubset { \in \RR_{\IndexSet}^{\backslash\emptyset} } }{%
      \IndexSubset \subset \IndexSet, \IndexSubset\not=\emptyset}}
  {\lambda_\IndexSubset\ge 0 \text{ and } \sum_{\IndexSubset\subset
      \IndexSet,\IndexSubset\not=\emptyset} \lambda_\IndexSubset \le 1}
  \eqfinp
\end{equation}
We denote by~$\CoordinateNorm{\TripleNormBall}{\cdot}$ (resp. $\CoordinateNorm{\TripleNormSphere}{\cdot}$)
the family $\na{\CoordinateNorm{\TripleNormBall}{\IndexSubset}}_{{\IndexSubset
    \subset \IndexSet }}$
of unit local balls in~\eqref{eq:coordinate_norm_unit_ball}
(resp. $\na{\CoordinateNorm{\TripleNormSphere}{\IndexSubset}}_{{\IndexSubset
    \subset \IndexSet}}$ of unit local spheres in~\eqref{eq:coordinate_norm_unit_sphere}) with the convention
\( \CoordinateNorm{\TripleNormBall}{\emptyset} =\CoordinateNorm{\TripleNormSphere}{\emptyset}= \{0\}\).
Then, $z\in \CoordinateNorm{\TripleNormBall}{\cdot}$ will denote a family
$\na{z_{(\IndexSubset)}}_{{\IndexSubset \subset \IndexSet,
    \IndexSubset\not=\emptyset}}$, 
such that $z_{(\IndexSubset)} \in \CoordinateNorm{\TripleNormBall}{\IndexSubset}$ for all
$\IndexSubset\subset \IndexSet$, with $\IndexSubset\not=\emptyset$.

\begin{proposition}
  \label{pr:Support_mapping_biconjugate}
  Let $\TripleNorm{\cdot}$ be the source norm with 
associated coupling $\CouplingCapra$, as in Definition~\ref{de:coupling_CAPRA},
and with associated families
  \( \sequence{\CoordinateNorm{\TripleNorm{\cdot}}{\IndexSubset}}{\IndexSubset\subset \IndexSet} \)
  of local-coordinate-$\IndexSubset$ norms 
and 
  \( \sequence{\CoordinateNormDual{\TripleNorm{\cdot}}{\IndexSubset}}{\IndexSubset\subset \IndexSet} \)
  of dual local-coordinate-$\IndexSubset$ norms, as in Definition~\ref{de:coordinate_norm}. 

  Let $\sequence{\Gamma_\IndexSubset}{\IndexSubset \subset \IndexSet}$ be a family of subsets of $\RR^d$
  such that the closed convex hull
  $\closedconvexhull{\np{\Gamma_\IndexSubset}}= \CoordinateNorm{\TripleNormBall}{\IndexSubset}$
  for all $\IndexSubset\subset \IndexSet$ (with the convention that \( \Gamma_{\emptyset}=\{0\} \)).

  \begin{subequations}
    \begin{itemize}
    \item 
      For any set function \( \FunctionBigF: 2^\IndexSet \to \barRR \), we have
      \begin{align}
        \SFMbi{ \bp{\FunctionBigF\circ\SupportMapping} }{\CouplingCapra}\np{\primal}
        &=
          \LFMr{ \bp{ \SFM{ \bp{\FunctionBigF\circ\SupportMapping}}{\CouplingCapra} } }
          \np{ \frac{\primal}{\TripleNorm{\primal}} }
          \eqsepv \forall \primal \in \RR^d\backslash\{0\}
          \eqfinv 
         \label{eq:Biconjugate_normalization_mapping}
           \intertext{where the function 
          \( \LFMr{ \bp{ \SFM{ \bp{\FunctionBigF\circ\SupportMapping}}{\CouplingCapra} } } \) 
          is closed convex and has the following expression as a Fenchel conjugate}
          \LFMr{ \bp{ \SFM{ \bp{\FunctionBigF\circ\SupportMapping}}{\CouplingCapra} } }
        &=
          \LFMr{ \bgp{ 
          \sup_{\IndexSubset \subset \IndexSet} \Bp{
          \CoordinateNormDual{\TripleNorm{\pi_{\IndexSubset}\np{\cdot}}}{\IndexSubset} - \FunctionBigF\np{\IndexSubset} } } }
          \eqfinv 
          \label{eq:Fenchel_of_CAPRA_conjugate_function_of_the_Support_as_conjugate_of_supremum}
          \intertext{and also has the following two equivalent expressions as a Fenchel biconjugate}
        &=
          \LFMbi{ \Bp{ \inf_{\IndexSubset \subset \IndexSet}
          \bp{ 
          \delta_{  \Gamma_{\IndexSubset}} \UppPlus \FunctionBigF\np{\IndexSubset}}}}
          \eqfinv 
          \label{eq:Biconjugate_of_min_Gamma_ind}
        \\
        &=
          \LFMbi{
          \Bp{ \primal \mapsto \inf \bset{\FunctionBigF\np{\IndexSubset}}{\primal \in \Gamma_\IndexSubset}}}
          \eqfinp
          \label{eq:Biconjugate_of_min_Gamma_ind_bis}
      \end{align}
    \item 
      For any finite-valued set function \( \FunctionBigF: 2^\IndexSet \to \RR \), 
      the function 
      \( \LFMr{ \bp{ \SFM{ \bp{\FunctionBigF\circ\SupportMapping}}{\CouplingCapra} } } \) 
      is proper convex lsc and has the following variational expression
      \begin{align}
        \LFMr{ \bp{ \SFM{ \bp{\FunctionBigF\circ\SupportMapping}}{\CouplingCapra} } }\np{\primal}
        &= 
          \min_{
          \substack{%
          \lambda \in \Delta_{\IndexSet} \\
        \primal \in \sum_{\IndexSubset\subset \IndexSet} \lambda_{\IndexSubset} \Gamma_\IndexSubset }}
        \sum_{\IndexSubset \subset \IndexSet} \lambda_{\IndexSubset} \FunctionBigF\np{\IndexSubset}
        \eqsepv \forall  \primal \in \RR^d 
        \eqfinv 
        \label{eq:biconjugate_with_Gamma}
      \end{align}
    \item
      For any finite-valued set function \( \FunctionBigF: 2^\IndexSet \to \RR
      \) such that $\FunctionBigF(\emptyset)=0$, the function 
      \( \LFMr{ \bp{ \SFM{ \bp{\FunctionBigF\circ\SupportMapping}}{\CouplingCapra} } } \) 
      is  proper convex lsc and has the following variational expression
      \begin{align}
        \LFMr{ \bp{ \SFM{ \bp{\FunctionBigF\circ\SupportMapping}}{\CouplingCapra} } }\np{\primal}
        &= 
          \min_{ \substack{%
          z \in \RR_{\IndexSet}^{\backslash\emptyset}
        \\
        \sum_{\AllSubsetsExceptEmptyset} \CoordinateNorm{\TripleNorm{z_{(\IndexSubset)}}}{\IndexSubset} \leq 1
        \\
        \sum_{\AllSubsetsExceptEmptyset} z_{(\IndexSubset)} = \primal } }
        \sum_{\AllSubsetsExceptEmptyset} \CoordinateNorm{\TripleNorm{z_{(\IndexSubset)}}}{\IndexSubset} \FunctionBigF\np{\IndexSubset}
        \eqsepv \forall \primal \in \RR^d
        \eqfinv
        \label{eq:Fenchel_of_CAPRA_conjugate_function_of_the_Support_variational_expression}
      \end{align}
      \label{eq:Fenchel_of_CAPRA_conjugate_function_of_the_Support}
      and the $\CouplingCapra$-biconjugate function \( \SFMbi{  \bp{\FunctionBigF\circ\SupportMapping} }{\CouplingCapra} \)
      has the following variational expression
    \end{itemize}
     \begin{equation}
        \SFMbi{  \bp{\FunctionBigF\circ\SupportMapping} }{\CouplingCapra}\np{\primal}
        =
        \frac{ 1 }{ \TripleNorm{\primal} } 
        \min_{
          \substack{
            z \in \RR_{\IndexSet}^{\backslash\emptyset}
            \\
            \sum_{\AllSubsetsExceptEmptyset} \CoordinateNorm{\TripleNorm{z_{(\IndexSubset)}}}{\IndexSubset} \leq \TripleNorm{\primal}
            \\
            \sum_{\AllSubsetsExceptEmptyset} z_{(\IndexSubset)} = \primal }
        }
        \sum_{\AllSubsetsExceptEmptyset} \CoordinateNorm{\TripleNorm{z_{(\IndexSubset)}}}{\IndexSubset} \FunctionBigF\np{\IndexSubset}
        \eqsepv \forall \primal \in \RR^d\backslash \na{0}
        \eqfinp
        \label{eq:CAPRA_biconjugate_function_of_the_Support_variational_expression}
      \end{equation}
   \end{subequations}
\end{proposition}

\begin{proof} 

   \noindent $\bullet$ 
  We consider the set function \( \FunctionBigF: 2^\IndexSet \to \barRR \).
By~\eqref{eq:CAPRA_Fenchel-Moreau_biconjugate} and the definition~\eqref{eq:normalization_mapping},
of the normalization mapping~$\normalized$, we immediately
get~\eqref{eq:Biconjugate_normalization_mapping}.
\medskip

 \noindent $\bullet$ 
  We consider the set function \( \FunctionBigF: 2^\IndexSet \to \barRR \). 
  We have
  \begin{align*}
    \LFMr{ \bp{ \SFM{ \bp{\FunctionBigF\circ\SupportMapping}}{\CouplingCapra} } }
    &= 
      \LFMr{ \Bp{ \sup_{\IndexSubset \subset \IndexSet} 
      \bp{ \CoordinateNormDual{\TripleNorm{\pi_\IndexSubset\np{\cdot}}}{\IndexSubset} -\FunctionBigF\np{\IndexSubset} } } }
      \tag{by~\eqref{eq:CAPRA_conjugate_function_of_the_Support} }
      \intertext{so that we have proved~\eqref{eq:Fenchel_of_CAPRA_conjugate_function_of_the_Support_as_conjugate_of_supremum} }
    &=
      \LFMr{ \Bp{ \sup_{\IndexSubset \subset \IndexSet} \bp{ \sigma_{ \CoordinateNorm{\TripleNormBall}{\IndexSubset} } -\FunctionBigF\np{\IndexSubset} } } }
      \tag{by~\eqref{eq:norm_dual_norm} as \( \CoordinateNorm{\TripleNormBall}{\IndexSubset}
      \) is the unit ball of the norm~\( \CoordinateNorm{\TripleNorm{\cdot}}{\IndexSubset} \)
      by~\eqref{eq:coordinate_norm_unit_ball}, and with the convention 
      \( \CoordinateNorm{\TripleNormBall}{\emptyset} =\{0\} \)}
    \\
    &=
      \LFMr{ \Bp{ \sup_{\IndexSubset \subset \IndexSet} \bp{ 
      \LFM{ \delta_{ \Gamma_\IndexSubset }} - \FunctionBigF\np{\IndexSubset} } } }
      \intertext{because 
      \( \LFM{ \delta_{  \Gamma_\IndexSubset }} =
      \sigma_{ \Gamma_\IndexSubset} = \sigma_{ \closedconvexhull\np{\Gamma_\IndexSubset}}=
      \sigma_{ \CoordinateNorm{\TripleNormBall}{\IndexSubset} } \), where the
      last equality comes from the assumption \( \closedconvexhull\np{\Gamma_\IndexSubset}
      =  \CoordinateNorm{\TripleNormBall}{\IndexSubset} \)}
    &=
      \LFMr{ \Bp{ \sup_{\IndexSubset \subset \IndexSet} \bp{ 
      \LFM{ \bp{ \delta_{  \Gamma_\IndexSubset } \UppPlus \FunctionBigF\np{\IndexSubset} } } } } }
      \tag{by property of conjugacies}
    \\
    &=
      \LFMr{ \Bp{ \LFM{ \Bp{ \inf_{\IndexSubset \subset \IndexSet} \bp{ 
      \delta_{\Gamma_\IndexSubset } \UppPlus \FunctionBigF\np{\IndexSubset} } } } } }
      \intertext{as conjugacies, being dualities, turn infima into suprema}
    &=
      \LFMbi{ \Bp{ \inf_{\IndexSubset \subset \IndexSet} \bp{ 
      \delta_{\Gamma_\IndexSubset } \UppPlus \FunctionBigF\np{\IndexSubset} } } }
      \tag{by definition~\eqref{eq:Fenchel_biconjugate} of the Fenchel biconjugate}
      \eqfinp
  \end{align*}
  Thus, we have obtained~\eqref{eq:Biconjugate_of_min_Gamma_ind}
  and \eqref{eq:Biconjugate_of_min_Gamma_ind_bis}.
  
  \noindent $\bullet$ 
  We consider the set function \( \FunctionBigF: 2^\IndexSet \to \RR \). 
  For the remaining expressions for
  \( \LFMr{ \bp{ \SFM{ \bp{\FunctionBigF\circ\SupportMapping}}{\CouplingCapra} } } \),
  we use a general formula 
  \cite[Corollary~2.8.11]{Zalinescu:2002}
  for the Fenchel conjugate of the supremum of proper convex functions
  \(\fonctionprimal_{\IndexSubset}: \RR^d \to \barRR \), $\IndexSubset \subset \IndexSet$:
  \begin{equation}
    \label{eq:Fenchel_conjugate_of_the_supremum_of_proper_convex_functions}
    \bigcap_{\IndexSubset \subset \IndexSet} \dom\fonctionprimal_{\IndexSubset} \neq \emptyset
    \implies
    \LFM{ \bp{ \sup_{\IndexSubset \subset \IndexSet} \fonctionprimal_{\IndexSubset} } }
    =
    \min_{\np{\lambda_\IndexSubset}_{\IndexSubset\subset \IndexSet}\in \Delta_{\IndexSet}} 
    \LFM{ \Bp{ \sum_{\IndexSubset\subset \IndexSet} \lambda_{\IndexSubset} \fonctionprimal_{\IndexSubset} } }
    \eqfinp
  \end{equation}
  Thus, we obtain 
  \begin{align*}
    \LFMr{ \Bp{ \SFM{ \bp{\FunctionBigF\circ\SupportMapping}}{\CouplingCapra} } }
    &= 
      \LFMr{ \Bp{ 
      \sup_{\IndexSubset \subset \IndexSet} \bp{ \CoordinateNormDual{\TripleNorm{\cdot}}{\IndexSubset} -\FunctionBigF\np{\IndexSubset} } } }
      \tag{by~\eqref{eq:CAPRA_conjugate_function_of_the_Support} }
    \\
    &=
      \LFMr{ \Bp{ \sup_{\IndexSubset \subset \IndexSet} 
      \bp{ \sigma_{ \CoordinateNorm{\TripleNormBall}{\IndexSubset} } -\FunctionBigF\np{\IndexSubset} } } }
      \tag{by~\eqref{eq:norm_dual_norm} as \( \CoordinateNorm{\TripleNormBall}{\IndexSubset}
      \) is the unit ball of the norm~\( \CoordinateNorm{\TripleNorm{\cdot}}{\IndexSubset} \)
      by~\eqref{eq:coordinate_norm_unit_ball}, and with the convention 
      \( \CoordinateNorm{\TripleNormBall}{\emptyset} =\{0\} \)}
    \\
    &=
      \LFMr{ \Bp{ \sup_{\IndexSubset \subset \IndexSet} 
      \bp{ \sigma_{ \Gamma_\IndexSubset} -\FunctionBigF\np{\IndexSubset} } } }
      \tag{as \( \closedconvexhull\np{\Gamma_\IndexSubset}=
      \CoordinateNorm{\TripleNormBall}{\IndexSubset} \) by assumption}
    \\
    &=
      \min_{\np{\lambda_\IndexSubset}_{\IndexSubset\subset \IndexSet}\in \Delta_{\IndexSet}} 
      \LFMr{ \Bp{ \sum_{ \IndexSubset\subset \IndexSet} \lambda_{\IndexSubset} 
      \bp{ \sigma_{ \Gamma_\IndexSubset } -\FunctionBigF\np{\IndexSubset} } } } 
      \tag{
      by~\eqref{eq:Fenchel_conjugate_of_the_supremum_of_proper_convex_functions} }
      \intertext{as the functions
      \(\fonctionprimal_{\IndexSubset} = \sigma_{ \Gamma_\IndexSubset } -\FunctionBigF\np{\IndexSubset} \)
      are proper convex (they even take finite values) for $\IndexSubset \subset \IndexSet$}
    &=
      \min_{\np{\lambda_\IndexSubset}_{\IndexSubset\subset \IndexSet}\in \Delta_{\IndexSet}} 
      \LFMr{ \Bp{ { \sigma_{ \sum_{\IndexSubset\subset \IndexSet} \lambda_{\IndexSubset} \Gamma_\IndexSubset }
      - \sum_{\IndexSubset\subset\IndexSet} \lambda_{\IndexSubset} \FunctionBigF\np{\IndexSubset} 
      } } } 
      \intertext{as, for all $\IndexSubset\subset \IndexSet$, 
      \( \lambda_{\IndexSubset} \sigma_{ \Gamma_\IndexSubset } = 
      \sigma_{ \lambda_{\IndexSubset} \Gamma_\IndexSubset } \) since \( \lambda_{\IndexSubset} \geq 0 \), 
      and then using the well-known property that the support function of 
      a Minkowski sum of subsets is the sum of the support functions of 
      the individual subsets \cite[p.~226]{Hiriart-Urruty-Lemarechal-I:1993}}
    &=
      \min_{\np{\lambda_\IndexSubset}_{\IndexSubset\subset \IndexSet}\in \Delta_{\IndexSet}} 
      \Bp{ \LFMr{ \bp{ \sigma_{ \sum_{\IndexSubset\subset \IndexSet} \lambda_{\IndexSubset}
      \Gamma_\IndexSubset } } }
      + \sum_{\IndexSubset\subset \IndexSet} \lambda_{\IndexSubset} \FunctionBigF\np{\IndexSubset} 
      } 
      \tag{by property of conjugacies} 
    \\
    &=
      \min_{\np{\lambda_\IndexSubset}_{\IndexSubset\subset \IndexSet}\in \Delta_{\IndexSet}}
      \Bp{ \delta_{ \sum_{{\IndexSubset\subset \IndexSet}} \lambda_{\IndexSubset} \Gamma_\IndexSubset } 
      +
      \sum_{\IndexSubset\subset \IndexSet} \lambda_{\IndexSubset}
      \FunctionBigF\np{\IndexSubset} }
      \eqfinp
      \tag{because \( \sum_{\IndexSubset\subset \IndexSet} \lambda_{\IndexSubset}
      \Gamma_\IndexSubset \)
      is a closed convex set }
      \intertext{Therefore, for all \( \primal \in \RR^d \), we have }
      \LFMr{ \bp{ \SFM{ \bp{\FunctionBigF\circ\SupportMapping}}{\CouplingCapra} } }\np{\primal}
    &= 
      \min_{
      \substack{
      \np{\lambda_\IndexSubset}_{\IndexSubset\subset \IndexSet}\in \Delta_{\IndexSet} \\
    \primal \in \sum_{\IndexSubset\subset \IndexSet}\lambda_{\IndexSubset} \Gamma_\IndexSubset}}
    \sum_{\IndexSubset\subset \IndexSet} \lambda_{\IndexSubset} \FunctionBigF\np{\IndexSubset} 
    \eqfinv
  \end{align*}
  which is~\eqref{eq:biconjugate_with_Gamma}.
  \medskip
  
  \noindent $\bullet$ 
  We consider the set function \( \FunctionBigF: 2^\IndexSet \to \RR \), 
such that $\FunctionBigF(\emptyset)=0$. 
We are going to
  prove~\eqref{eq:Fenchel_of_CAPRA_conjugate_function_of_the_Support_variational_expression}.
  For this purpose, we start from the already proven
  Equation~\eqref{eq:biconjugate_with_Gamma} where we choose 
  $\Gamma_\IndexSubset = \CoordinateNorm{\TripleNormSphere}{\IndexSubset}$ for all $\IndexSubset \subset \IndexSet$
  (except $K=\emptyset$ for which $\Gamma_\IndexSubset=\na{0}$),
  which is licit since $\closedconvexhull\np{\CoordinateNorm{\TripleNormSphere}{\IndexSubset}}= \CoordinateNorm{\TripleNormBall}{\IndexSubset}$. We obtain
  \begin{align*}
    \LFMr{ \bp{ \SFM{ \bp{\FunctionBigF\circ\SupportMapping}}{\CouplingCapra} } }\np{\primal} 
    &=
      \min_{
      \substack{
      \lambda\in \Delta_{\IndexSet} \\
    \primal \in \sum_{\IndexSubset\subset \IndexSet}\lambda_{\IndexSubset} \Gamma_\IndexSubset}}
    \sum_{\IndexSubset\subset \IndexSet} \lambda_{\IndexSubset} \FunctionBigF\np{\IndexSubset}
    \tag{by \eqref{eq:biconjugate_with_Gamma}}
    \\
    &=
      \min_{
      \substack{
      \lambda\in \SimplexWithoutEmptySet \\
    \primal \in \sum_{\AllSubsetsExceptEmptyset}\lambda_{\IndexSubset} \Gamma_\IndexSubset}}
    \sum_{\AllSubsetsExceptEmptyset} \lambda_{\IndexSubset} \FunctionBigF\np{\IndexSubset}
    \tag{since $\FunctionBigF(\emptyset)=0$ and $\Gamma_{\emptyset}=\{0\}$}
    \\
    &=
      \min_{
      \substack{
      \lambda\in \SimplexWithoutEmptySet
    \\
    \primal \in \sum_{\AllSubsetsExceptEmptyset}\lambda_{\IndexSubset} \CoordinateNorm{\TripleNormSphere}{\IndexSubset}}}
    \sum_{\AllSubsetsExceptEmptyset} \lambda_{\IndexSubset} \FunctionBigF\np{\IndexSubset}
    \tag{using $\Gamma_\IndexSubset = \CoordinateNorm{\TripleNormSphere}{\IndexSubset}$}
    \\
    &= 
      \min_{ \substack{%
      z' \in \CoordinateNorm{\TripleNormSphere}{\cdot},
      \lambda \in \SimplexWithoutEmptySet
    \\
    \sum_{\AllSubsetsExceptEmptyset} \lambda_{\IndexSubset} z'_{({\IndexSubset})} = \primal
    } }  
    \sum_{\AllSubsetsExceptEmptyset} \lambda_{\IndexSubset} \FunctionBigF\np{\IndexSubset}
    \tag{getting rid of $\lambda_{\emptyset}$ since $\FunctionBigF(\emptyset)=0$}
    \\
    &=
      \min_{ \substack{%
      z \in \RR_{\IndexSet}^{\backslash\emptyset}
    \\
    \sum_{\AllSubsetsExceptEmptyset} \CoordinateNorm{\TripleNorm{z_{({\IndexSubset})}}}{\IndexSubset} \leq 1 
    \\
    \sum_{\AllSubsetsExceptEmptyset} z_{({\IndexSubset})} = \primal
    } }
    \sum_{\AllSubsetsExceptEmptyset}\CoordinateNorm{\TripleNorm{z_{({\IndexSubset})}}}{\IndexSubset} \FunctionBigF\np{\IndexSubset} 
    \eqfinv
  \end{align*}
  by putting \( z_{({\IndexSubset})} = \lambda_{\IndexSubset} z'_{({\IndexSubset})}\),
  for all $\emptyset \subsetneq\IndexSubset\subset \IndexSet$.
  Thus, we have obtained~\eqref{eq:Fenchel_of_CAPRA_conjugate_function_of_the_Support_variational_expression}.

  Finally, from \( \SFMbi{ \bp{\FunctionBigF\circ\SupportMapping}}{\CouplingCapra}
  =\LFMr{ \bp{ \SFM{ \bp{\FunctionBigF\circ\SupportMapping}}{\CouplingCapra} } }
  \circ \normalized \), by~\eqref{eq:CAPRA_Fenchel-Moreau_biconjugate},
  we get that 
  \begin{equation*}
    \SFMbi{ \bp{\FunctionBigF\circ\SupportMapping} }{\CouplingCapra}\np{\primal} 
    = \frac{ 1 }{ \TripleNorm{\primal} } 
    \min_{ \substack{%
        z \in \RR_{\IndexSet}^{\backslash\emptyset}
        \\
        \sum_{\AllSubsetsExceptEmptyset}
        \CoordinateNorm{\TripleNorm{z_{(\IndexSubset)}}}{\IndexSubset} \leq \TripleNorm{\primal}
        \\
        \sum_{\AllSubsetsExceptEmptyset} z_{(\IndexSubset)} = \primal  } }
    \sum_{\AllSubsetsExceptEmptyset}
    \CoordinateNorm{\TripleNorm{z_{(\IndexSubset)}}}{\IndexSubset} \FunctionBigF\np{\IndexSubset} 
    \eqsepv \forall \primal \in \RR^d\backslash\{0\}
    \eqfinp 
  \end{equation*}
  Therefore, we have proved~\eqref{eq:CAPRA_biconjugate_function_of_the_Support_variational_expression}.
  \medskip

  This ends the proof.
\end{proof}

\subsection{Proof of Theorem~\ref{th:support_mapping_conjugate} }
\label{Proof_of_Theorem}


\begin{proof}
  %
  By~\eqref{eq:CAPRA_conjugate_function_of_the_Support}, we have
  that \( \SFM{ \np{ \FunctionBigF \circ \SupportMapping } }{\CouplingCapra} =
  \sup_{\IndexSubset\subset \IndexSet} 
  \Bp{
    \CoordinateNormDual{\TripleNorm{\cdot}}{\IndexSubset}-\FunctionBigF\np{\IndexSubset}
  } \), where \( \CoordinateNormDual{\TripleNorm{\cdot}}{\IndexSubset}
  =
  \TopDualNorm{\TripleNorm{\cdot}}{\IndexSubset} \)
by Equation~\eqref{eq:dual_coordinate-k_norm_=_generalized_top-k_norm} 
in Proposition~\ref{pr:dual_coordinate-k_norm_=_generalized_top-k_norm}
since the norm $\TripleNorm{\cdot}$ is orthant-monotonic.
Hence, we obtain~\eqref{eq:conjugate_l0norm_TopDualNorm}. 

  As, by assumption, both the norm $\TripleNorm{\cdot}$ and the dual norm $\TripleNormDual{\cdot}$
  are orthant-strictly monotonic, 
  Proposition~\ref{pr:nonempty_subdifferential} applies.
  Therefore, for any vector~\( \primal \in \RR^d \)
  and any \( \dual \in \partial_{\CouplingCapra}\np{ \FunctionBigF \circ \SupportMapping }\np{\primal}  \neq \emptyset \),
  we obtain
  \begin{align*}
    \SFMbi{\np{ \FunctionBigF \circ \SupportMapping }}{\CouplingCapra}\np{\primal} 
    & \geq 
      \CouplingCapra\np{\primal,\dual} \LowPlus
      \bp{ - \SFM{\np{ \FunctionBigF \circ \SupportMapping }}{\CouplingCapra}\np{\dual} }
      \tag{by definition~\eqref{eq:Fenchel-Moreau_biconjugate} of the biconjugate}
    \\
    &=
      \CouplingCapra\np{\primal,\dual} 
      - \SFM{\np{ \FunctionBigF \circ \SupportMapping }}{\CouplingCapra}\np{\dual}
      \tag{because \( -\infty < \CouplingCapra\np{\primal,\dual} < +\infty \) by~\eqref{eq:coupling_CAPRA} }
    \\
    &=
      \CouplingCapra\np{\primal,\dual} 
      - \bp{ \CouplingCapra\np{\primal,\dual} - \np{ \FunctionBigF \circ \SupportMapping }\np{\primal} }
      \tag{by definition~\eqref{eq:Capra-subdifferential_b} 
      of the \Capra-subdifferential \( \partial_{\CouplingCapra}\np{ \FunctionBigF \circ \SupportMapping }\np{\primal} \)}
    \\
    &=
      \np{ \FunctionBigF \circ \SupportMapping }\np{\primal} 
      \eqfinp
  \end{align*}
  On the other hand, we have that 
  \( \SFMbi{\np{ \FunctionBigF \circ \SupportMapping }}{\CouplingCapra}\np{\primal} \leq \np{ \FunctionBigF \circ \SupportMapping }\np{\primal} \) 
  by~\eqref{eq:galois-cor}.
  We conclude that \( \SFMbi{\np{ \FunctionBigF \circ \SupportMapping }}{\CouplingCapra}\np{\primal} =
  \np{ \FunctionBigF \circ \SupportMapping }\np{\primal} \), which is~\eqref{eq:biconjugate_l0norm_TopDualNorm}.  
  \medskip

  This ends the proof.
\end{proof}

\section{Hidden convexity and variational formulations for nondecreasing \FSM}
\label{Hidden_convexity_and_variational_formulation_for_the_pseudo_norm}

In this
Sect.~\ref{Hidden_convexity_and_variational_formulation_for_the_pseudo_norm},
we suppose that both the source norm~$\TripleNorm{\cdot}$ and the dual norm~$\TripleNormDual{\cdot}$
are orthant-strictly monotonic.
From our main result obtained
in Sect.~\ref{The_Capra_conjugacy_under_orthant-strict_monotonicity}
--- namely, Theorem~\ref{th:support_mapping_conjugate} which provides
conditions under which a function of the support mapping (\FSM) is a \Capra-convex function ---
we will derive the following results.
In~\S\ref{Hidden_convexity_in_the_support_mapping}, 
we show that any nondecreasing finite-valued \FSM\
coincides, on the unit sphere~$\TripleNormSphere =
\bset{\primal \in \RR^d}{\TripleNorm{\primal} = 1} $, 
with a proper convex lsc function on~$\RR^d$.
In~\S\ref{Variational_formulation_for_the_pseudo_norm}, 
we deduce that a nondecreasing finite-valued \FSM\
taking the value~$0$ on the null vector
admits a variational formula 
which involves the whole family of generalized local-$\IndexSubset$-support
dual~norms in Definition~\ref{de:top_dual_norm}.
Then, in~\S\ref{Upper_and_lower_bounds_for_the_support_mapping_as_norm_ratios},
we show that the variational formulation obtained 
yields a family of lower and upper bounds for suitable nondecreasing
\FSM, as a ratio between two norms.




\subsection{Hidden convexity in nondecreasing \FSM }
\label{Hidden_convexity_in_the_support_mapping}

We now present a (rather unexpected) consequence of the just established
Theorem~\ref{th:support_mapping_conjugate}.

\begin{proposition}
  \label{pr:calL0}

 Let $\TripleNorm{\cdot}$ be the source norm with 
associated coupling $\CouplingCapra$, as in Definition~\ref{de:coupling_CAPRA},
and with associated families 
 \( \sequence{\TopDualNorm{\TripleNorm{\cdot}}{\IndexSubset}}{\IndexSubset \subset \IndexSet} \)
  of generalized local-top-$\IndexSubset$ dual~norms
and
 \( \sequence{\SupportDualNorm{\TripleNorm{\cdot}}{\IndexSubset}}{\IndexSubset \subset \IndexSet} \) 
of generalized local-$\IndexSubset$-support dual~norms,
  as in Definition~\ref{de:top_dual_norm}. 

Suppose that both the norm $\TripleNorm{\cdot}$ and the dual norm $\TripleNormDual{\cdot}$
  are orthant-strictly monotonic.
  \begin{subequations}
    For any nondecreasing finite-valued set function \( \FunctionBigF :
    2^\IndexSet \to \RR \), the following statements hold true.
    \begin{description}
    \item[$(i)$\label{pr:calL0-first}]
      The following function~${\cal L}_{0}^{\FunctionBigF} : \RR^d \to \barRR $, defined by
      \begin{equation}
        {\cal L}_{0}^{\FunctionBigF} = \LFMr{ \bp{ \SFM{\np{ \FunctionBigF \circ \SupportMapping }}{\CouplingCapra} } }
        \eqfinv
        \label{eq:definition_calL0}
      \end{equation}
      is proper convex lsc.
    \item[$(ii)$]\label{pr:calL0-second}
      The function~$\FunctionBigF \circ \SupportMapping$ coincides, on the unit sphere~$\TripleNormSphere =
      \bset{\primal \in \RR^d}{\TripleNorm{\primal} = 1} $, 
with the function~${\cal L}_{0}^{\FunctionBigF}$ in~\eqref{eq:definition_calL0}:
      \begin{equation}
        \np{ \FunctionBigF \circ \SupportMapping }\np{\primal} = {\cal L}_{0}^{\FunctionBigF}\np{\primal} 
        \eqsepv 
        \forall \primal \in \TripleNormSphere 
        \eqfinp
        \label{eq:support_mapping_convex_sphere}
      \end{equation}
    \item[$(iii)$]\label{pr:calL0-third}
      The function~$\FunctionBigF \circ \SupportMapping$ can be expressed as the 
      composition of the proper convex lsc function~${\cal L}_{0}^{\FunctionBigF}$
      in~\eqref{eq:definition_calL0}
      with the 
      normalization mapping~$\normalized$ in~\eqref{eq:normalization_mapping}:
      \begin{equation}
        \np{ \FunctionBigF \circ \SupportMapping }\np{\primal} = 
        {\cal L}_{0}^{\FunctionBigF} \np{ \frac{ \primal }{ \TripleNorm{\primal} } } 
        \eqsepv 
        \forall \primal \in \RR^d\backslash\{0\}
        \eqfinp 
        \label{eq:support_mapping_convex_normalization_mapping}
      \end{equation}
    \item[$(iv)$]\label{pr:calL0-four}
      The proper convex lsc function~${\cal L}_{0}^{\FunctionBigF}$
      in~\eqref{eq:definition_calL0} is given by
      \begin{equation}
        {\cal L}_{0}^{\FunctionBigF}= 
        \LFMr{ \Bp{ \sup_{\IndexSubset\subset \IndexSet} 
            \Bc{ \TopDualNorm{\TripleNorm{\cdot}}{\IndexSubset}-\FunctionBigF\np{\IndexSubset} } } }
        \eqfinp
        \label{eq:calL_0_definition}
      \end{equation}
    \item[$(v)$]\label{pr:calL0-five}
      The function~${\cal L}_{0}^{\FunctionBigF}$ in~\eqref{eq:definition_calL0}
is the largest proper convex lsc function below the function 
      \begin{equation}
        \inf_{\IndexSubset\subset \IndexSet} \Bc{ \delta_{
            \SupportDualNorm{\TripleNormBall}{\IndexSubset} } \UppPlus \FunctionBigF\np{\IndexSubset} }
        \eqfinv
        \label{eq:support_mapping_largest_convex_balls}
      \end{equation}
      {      that is, below the function
        \( \primal \in \RR^d \mapsto \inf \bset{ \FunctionBigF\np{\IndexSubset} }%
        { \IndexSubset \subset \IndexSet, \, \primal \in
          \SupportDualNorm{\TripleNormBall}{\IndexSubset} } \),
        with the convention that \( \SupportDualNorm{\TripleNormBall}{\emptyset}=\{0\} \)
        and that \( \inf \emptyset = +\infty \).}
    \item[$(vi)$]\label{pr:calL0-six} 
      The function~${\cal L}_{0}^{\FunctionBigF}$ in~\eqref{eq:definition_calL0}
is the largest proper convex lsc function below the function 
      \begin{equation}
        \inf_{\IndexSubset \subset \IndexSet} \Bc{ \delta_{
            \SupportDualNorm{\TripleNormSphere}{\IndexSubset} } \UppPlus
          \FunctionBigF\np{\IndexSubset} }
        \eqfinv
        \label{eq:support_mapping_largest_convex_spheres}
      \end{equation}
      that is, below the function
      \( \primal \in \RR^d \mapsto \inf \bset{ \FunctionBigF\np{\IndexSubset} }%
      { \IndexSubset \subset \IndexSet, \, \primal \in
        \SupportDualNorm{\TripleNormSphere}{\IndexSubset} } \),
      with the convention that \( \SupportDualNorm{\TripleNormSphere}{\emptyset}=\{0\} \)
      and that \( \inf \emptyset = +\infty \).
    \item[$(vii)$]\label{pr:calL0-seven}  
The function~${\cal L}_{0}^{\FunctionBigF}$ in~\eqref{eq:definition_calL0}
also has the following variational expressions
      \begin{align}
        {\cal L}_{0}^{\FunctionBigF}\np{\primal} 
        &= 
          \min_{
          \substack{%
          \lambda \in \Delta_{\IndexSet} \\
        \primal \in \sum_{\IndexSubset\subset \IndexSet} \lambda_{\IndexSubset}
        \SupportDualNorm{\TripleNormBall}{\IndexSubset} } }
        \sum_{\IndexSubset \subset \IndexSet} \lambda_{\IndexSubset} \FunctionBigF\np{\IndexSubset}
        \eqsepv \forall  \primal \in \RR^d 
        \label{eq:biconjugate_with_balls}
        \\
        &= 
          \min_{
          \substack{%
          \lambda \in \Delta_{\IndexSet} \\
        \primal \in \sum_{\IndexSubset\subset \IndexSet} \lambda_{\IndexSubset}
        \SupportDualNorm{\TripleNormSphere}{\IndexSubset} } }
        \sum_{\IndexSubset \subset \IndexSet} \lambda_{\IndexSubset} \FunctionBigF\np{\IndexSubset}
        \eqsepv \forall  \primal \in \RR^d 
        \label{eq:biconjugate_with_spheres}
        \intertext{and, if $\FunctionBigF\np{\emptyset}=0$, }
        {\cal L}_{0}^{\FunctionBigF}\np{\primal} 
        &= 
          \min_{ \substack{%
          z \in \RR_{\IndexSet}^{\backslash\emptyset}
        \\
        \sum_{\emptyset \subsetneq\IndexSubset\subset \IndexSet} 
        \SupportDualNorm{\TripleNorm{z_{(\IndexSubset)}}}{\IndexSubset} \leq 1
        \\
        \sum_{\emptyset \subsetneq\IndexSubset\subset \IndexSet} z_{(\IndexSubset)} = \primal } }
        \sum_{\emptyset \subsetneq \IndexSubset\subset \IndexSet}
        \SupportDualNorm{\TripleNorm{z_{(\IndexSubset)}}}{\IndexSubset}
        \FunctionBigF\np{\IndexSubset}
        \eqsepv \forall \primal \in \RR^d
        \eqfinp
        \label{eq:support_mapping_convex_minimum}
      \end{align}
    \end{description}
  \end{subequations}
\end{proposition}

\begin{proof}
  Since the norm $\TripleNorm{\cdot}$ is orthant-strictly monotonic,
  it is orthant-monotonic, so that we have 
  \( \CoordinateNorm{\TripleNorm{\cdot}}{\IndexSubset} 
  =
  \SupportDualNorm{\TripleNorm{\cdot}}{\IndexSubset} \) 
  and
  \(      \CoordinateNormDual{\TripleNorm{\cdot}}{\IndexSubset}
  =
  \TopDualNorm{\TripleNorm{\cdot}}{\IndexSubset} \),
  for any \( \IndexSubset\subset \IndexSet \) 
  by Equation~\eqref{eq:dual_coordinate-k_norm_=_generalized_top-k_norm} 
in Proposition~\ref{pr:dual_coordinate-k_norm_=_generalized_top-k_norm}
  (with the proper conventions that they are all zero in the case \( \IndexSubset=\emptyset \)).
\smallskip

  \noindent $\bullet\, (i)$
  The function~${\cal L}_{0}^{\FunctionBigF}$ in~\eqref{eq:definition_calL0}
is a Fenchel conjugate, hence is a closed convex function
since the Fenchel conjugacy induces a one-to-one correspondence
  between the closed convex functions on~$\RR^d$ and themselves 
  \cite[Theorem~5]{Rockafellar:1974}.
  By~\eqref{eq:support_mapping_convex_sphere}, proven below, the
  function~${\cal L}_{0}^{\FunctionBigF}$ takes finite values on the unit sphere.
  Thanks to Footnote~\ref{ft:closed_convex_function}, we conclude that 
  the function~${\cal L}_{0}^{\FunctionBigF}$ is proper convex lsc.
\smallskip

  \noindent $\bullet\, (ii)$
  The assumptions make it possible to conclude that 
  \( \SFMbi{ \np{ \FunctionBigF \circ \SupportMapping } }{\CouplingCapra} = \FunctionBigF \circ \SupportMapping \), 
  thanks to Theorem~\ref{th:support_mapping_conjugate}.
  We deduce from~\eqref{eq:Biconjugate_normalization_mapping} 
that 
\[
\np{\FunctionBigF \circ \SupportMapping }\np{\primal}
=
\SFMbi{ \bp{\FunctionBigF\circ\SupportMapping} }{\CouplingCapra}\np{\primal}
= 
\LFMr{ \bp{ \SFM{ \bp{\FunctionBigF\circ\SupportMapping}}{\CouplingCapra} } }
          \np{ \frac{\primal}{\TripleNorm{\primal}} }
          \eqsepv \forall \primal \in \RR^d\backslash\{0\}
\eqfinp
\]
Thus, the function~$\FunctionBigF \circ \SupportMapping$ coincides, 
on the unit sphere~$\TripleNormSphere$, 
  with the closed convex function ${\cal L}_{0}^{\FunctionBigF} : \RR^d \to \barRR $
  given by~\eqref{eq:CAPRA_Fenchel-Moreau_biconjugate}, 
  namely 
  \( {\cal L}_{0}^{\FunctionBigF} = \LFMr{ \bp{ \SFM{\np{ \FunctionBigF \circ \SupportMapping
        }}{\CouplingCapra} } } \).
  Thus, we have proved~\eqref{eq:support_mapping_convex_sphere}.
\smallskip

  \noindent $\bullet\, (iii)$
  The equality~\eqref{eq:support_mapping_convex_normalization_mapping}
  is an easy consequence of the
  property~\eqref{eq:support_mapping_is_0-homogeneous}, 
  implying that the function~$\FunctionBigF \circ \SupportMapping$ is invariant along
  any open ray of~$\RR^d$.
\smallskip

  \noindent $\bullet\, (iv)$ 
  As \( {\cal L}_{0}^{\FunctionBigF} = \LFMr{ \bp{ \SFM{\np{ \FunctionBigF \circ \SupportMapping }}{\CouplingCapra} } } \)
  by definition~\eqref{eq:definition_calL0},
  and as \( \LFMr{ \bp{ \SFM{\np{ \FunctionBigF \circ \SupportMapping }}{\CouplingCapra} } }
  = \LFMr{ \Bp{ \sup_{\IndexSubset\subset \IndexSet} \Bc{ \TopDualNorm{\TripleNorm{\cdot}}{\IndexSubset}-\FunctionBigF\np{\IndexSubset}
      } } } \) 
  by~\eqref{eq:conjugate_l0norm_TopDualNorm},
  we get \eqref{eq:calL_0_definition}.
  \smallskip

  \noindent $\bullet\, (v)$
  We use~\eqref{eq:Biconjugate_of_min_Gamma_ind} with 
  \( \Gamma_{\IndexSubset} =
  \SupportDualNorm{\TripleNormBall}{\IndexSubset} \)
  to obtain~\eqref{eq:support_mapping_largest_convex_balls}. Indeed, since
  $\SupportDualNorm{\TripleNormBall}{\IndexSubset} =
  \CoordinateNorm{\TripleNormBall}{\IndexSubset}$ 
  (because \( \CoordinateNorm{\TripleNorm{\cdot}}{\IndexSubset} 
  =
  \SupportDualNorm{\TripleNorm{\cdot}}{\IndexSubset} \)),
  we have that
  $\closedconvexhull{\np{\Gamma_\IndexSubset}}=\SupportDualNorm{\TripleNormBall}{\IndexSubset}=
  \CoordinateNorm{\TripleNormBall}{\IndexSubset}$ and thus
  $\SupportDualNorm{\TripleNormBall}{\IndexSubset}$ can be used as $\Gamma_{\IndexSubset}$.
\smallskip

  \noindent $\bullet\, (vi)$ 
  In the same way, we use~\eqref{eq:Biconjugate_of_min_Gamma_ind} with 
  \( \Gamma_{\IndexSubset} = \SupportDualNorm{\TripleNormSphere}{\IndexSubset} \)
  to obtain~\eqref{eq:support_mapping_largest_convex_spheres}.

  \noindent $\bullet\, (vii)$ We use~\eqref{eq:biconjugate_with_Gamma} with
  \( \Gamma_{\IndexSubset} = \SupportDualNorm{\TripleNormBall}{\IndexSubset} \)
  to obtain~\eqref{eq:biconjugate_with_balls}. In the same way, we also
  use~\eqref{eq:biconjugate_with_Gamma} with
  \( \Gamma_{\IndexSubset} = \SupportDualNorm{\TripleNormSphere}{\IndexSubset}
  \) to obtain~\eqref{eq:biconjugate_with_spheres}.  We
  use~\eqref{eq:Fenchel_of_CAPRA_conjugate_function_of_the_Support_variational_expression}
with  \( \CoordinateNorm{\TripleNorm{\cdot}}{\IndexSubset} 
  =
  \SupportDualNorm{\TripleNorm{\cdot}}{\IndexSubset} \)
  to obtain~\eqref{eq:support_mapping_convex_minimum}.  
\smallskip
  
  This ends the proof.
\end{proof}

\subsection{Variational formulation for nondecreasing \FSM}
\label{Variational_formulation_for_the_pseudo_norm}

As a straightforward application of Proposition~\ref{pr:calL0},
we obtain our second main result, namely a variational formulation for suitable nondecreasing \FSM.

\begin{theorem}
 Let $\TripleNorm{\cdot}$ be the source norm with associated family 
 \( \sequence{\SupportDualNorm{\TripleNorm{\cdot}}{\IndexSubset}}{\IndexSubset \subset \IndexSet} \) 
of generalized local-$\IndexSubset$-support dual~norms,
  as in Definition~\ref{de:top_dual_norm}. 

Suppose that both the norm $\TripleNorm{\cdot}$ and the dual norm $\TripleNormDual{\cdot}$
  are orthant-strictly monotonic. Then,
    for any nondecreasing finite-valued set function \( \FunctionBigF :
    2^\IndexSet \to \RR \) such that $\FunctionBigF\np{\emptyset}=0$, we have the equality
  \begin{equation}
    \np{ \FunctionBigF \circ \SupportMapping }\np{\primal} 
    = 
    \frac{ 1 }{ \TripleNorm{\primal} } 
    \min_{ \substack{%
        z \in \RR_{\IndexSet}^{\backslash\emptyset}
        \\
        \sum_{\emptyset \subsetneq\IndexSubset\subset \IndexSet} 
        \SupportDualNorm{\TripleNorm{z_{(\IndexSubset)}}}{\IndexSubset} \leq \TripleNorm{\primal} 
        \\
        \sum_{\emptyset \subsetneq\IndexSubset\subset \IndexSet} z_{(\IndexSubset)} = \primal } }
    \sum_{\emptyset \subsetneq\IndexSubset\subset \IndexSet}
    \SupportDualNorm{\TripleNorm{z_{(\IndexSubset)}}}{\IndexSubset}
    \FunctionBigF\np{\IndexSubset}
    \eqsepv \forall \primal \in \RR^d\backslash\{0\}
    \eqfinp
    \label{eq:Variational_formulation_for_the_pseudo_norm}  
  \end{equation}
  When \( \SupportMapping\np{\primal} = L  \neq \emptyset \), 
  the minimum in~\eqref{eq:Variational_formulation_for_the_pseudo_norm} 
  is achieved at \( z \in \RR_{\IndexSet}^{\backslash\emptyset}\) such that 
  \( z_{(\IndexSubset)}=0 \) for \( \IndexSubset \neq L \)
  and \( z_{(L)}=\primal \). 
  \label{th:Variational_formulation_for_the_pseudo_norm}
\end{theorem}

\begin{proof}
  Equation~\eqref{eq:Variational_formulation_for_the_pseudo_norm}  
  derives from~\eqref{eq:support_mapping_convex_normalization_mapping}
  and~\eqref{eq:support_mapping_convex_minimum}.

  When \( \SupportMapping\np{\primal} = L \neq \emptyset \), we have 
  \( \TripleNorm{\primal} =  \CoordinateNorm{\TripleNorm{\primal}}{L} \)
  by~\eqref{eq:coordinate_norm_graded_b}.
  Moreover,  \( \CoordinateNorm{\TripleNorm{\cdot}}{\IndexSubset} 
  =
  \SupportDualNorm{\TripleNorm{\cdot}}{\IndexSubset} \) 
  by Equation~\eqref{eq:dual_coordinate-k_norm_=_generalized_top-k_norm} 
in Proposition~\ref{pr:dual_coordinate-k_norm_=_generalized_top-k_norm},
  since the norm $\TripleNorm{\cdot}$ is orthant-strictly monotonic,
  hence is orthant-monotonic.
  As a consequence, we have that 
  \( \TripleNorm{\primal} =
  \SupportDualNorm{\TripleNorm{\primal}}{L} \).
  Therefore, \( z \in \RR_{\IndexSet}^{\backslash\emptyset} \) such that 
  \( z_{(L)}=x \) and \( z_{(K)}=0 \) for \( K \not=L \)
  is admissible for the minimization
  problem~\eqref{eq:Variational_formulation_for_the_pseudo_norm}.
  We deduce that 
  \( \FunctionBigF\np{L} =    \np{ \FunctionBigF \circ \SupportMapping }\np{\primal} 
  \leq 
  \frac{ 1 }{ \TripleNorm{\primal} } \FunctionBigF\np{L} 
  \SupportDualNorm{\TripleNorm{\primal}}{L}
  =  \FunctionBigF\np{L}\).
  \medskip

  This ends the proof. 
\end{proof}

As an illustration, Theorem~\ref{th:Variational_formulation_for_the_pseudo_norm} applies 
when the norm~$\TripleNorm{\cdot}$ is any of the 
$\ell_p$-norms~$\Norm{\cdot}_{p}$ on the space~\( \RR^d \), 
for $p\in ]1,\infty[$, giving: 
\begin{equation}
  \np{ \FunctionBigF \circ \SupportMapping }\np{\primal} 
  = 
  \frac{ 1 }{ \Norm{\primal}_{p} } 
  \min_{ \substack{%
      z \in \RR_{\IndexSet}^{\backslash\emptyset}
      \\
      \sum_{\emptyset \subsetneq\IndexSubset\subset \IndexSet} 
\Norm{z_{(\IndexSubset)}}_{p,\star,(\IndexSubset)}^{\star \mathrm{sn}} \leq \Norm{\primal}_{p} 
      \\
      \sum_{\emptyset \subsetneq\IndexSubset\subset \IndexSet} z_{(\IndexSubset)} = \primal } }
  \sum_{\emptyset \subsetneq\IndexSubset\subset \IndexSet}
  \Norm{z_{(\IndexSubset)}}_{p,\star,(\IndexSubset)}^{\star \mathrm{sn}} 
  \FunctionBigF\np{\IndexSubset}
  \eqsepv \forall \primal \in \RR^d\backslash\{0\}
  \eqfinp
\end{equation}
Indeed, when $p\in ]1,\infty[$, 
the $\ell_p$-norm $\TripleNorm{\cdot}=\Norm{\cdot}_{p}$
is orthant-strictly monotonic, and so is its dual norm
$\TripleNormDual{\cdot}=\Norm{\cdot}_q$ where \( 1/p + 1/q = 1 \)
as easily seen. 
      %
When $p=\infty$, the $\ell_{\infty}$-norm $\TripleNorm{\cdot}=\Norm{\cdot}_{\infty}$
is not orthant-strictly monotonic.
When $p=1$, the $\ell_1$-norm $\TripleNorm{\cdot}=\Norm{\cdot}_{1}$
is orthant-strictly monotonic, 
but the dual norm $\TripleNorm{\cdot}=\Norm{\cdot}_{\infty}$ is not.


\subsubsubsection{Applications to sparse optimization}

Supposing that the assumptions of Theorem~\ref{th:Variational_formulation_for_the_pseudo_norm}
are satisfied, we get the two following 
reformulations for exact sparse optimization problems. 
  Let \( \FunctionBigF : 2^\IndexSet \to \RR_+ \) be a nondecreasing finite-valued set function
  such that \( \FunctionBigF\np{\emptyset}=0 \). 





  Let $\Convex \subset \RR^d$ be such that \( 0 \not\in \Convex \)
(to avoid a division by zero, and also because the minimum would be achieved at zero).
  Then, we have that 
    \begin{align}
      \min_{ \primal \in \Convex } \FunctionBigF\bp{ \SupportMapping\np{\primal} }
      &= 
        \min_{ \primal \in \Convex}
        \frac{ 1 }{ \TripleNorm{\primal} } 
        \underbrace{%
        \min_{ \substack{%
        z \in \RR_{\IndexSet}^{\backslash\emptyset}
      \\
      \sum_{\emptyset \subsetneq\IndexSubset\subset \IndexSet} 
      \SupportDualNorm{\TripleNorm{z_{(\IndexSubset)}}}{\IndexSubset} \leq \TripleNorm{\primal} 
      \\
      \sum_{\emptyset \subsetneq\IndexSubset\subset \IndexSet} z_{(\IndexSubset)} = \primal } }
      \sum_{\emptyset \subsetneq\IndexSubset\subset \IndexSet}
      \SupportDualNorm{\TripleNorm{z_{(\IndexSubset)}}}{\IndexSubset}
      \FunctionBigF\np{\IndexSubset}
      }_{\textrm{convex optimization problem}}
      \eqfinp
    \end{align}




For any $\alpha \in \RR$, we have that 
  \begin{subequations}
    \begin{align}
      \min_{ \np{ \FunctionBigF \circ \SupportMapping }\np{\primal} \leq \alpha } \fonctionprimal\np{\primal} 
      &= 
        \min_{ \substack{%
        \primal \in \RR^d, 
        z \in  \RR_{\IndexSet}^{\backslash\emptyset}
      \\
      \sum_{\emptyset \subsetneq\IndexSubset\subset \IndexSet}
      \SupportDualNorm{\TripleNorm{z_{(\IndexSubset)}}}{\IndexSubset} \leq \TripleNorm{\primal}
      \\
      \sum_{\emptyset \subsetneq\IndexSubset\subset \IndexSet} z_{(\IndexSubset)} = \primal  
      \\
      \sum_{\emptyset \subsetneq\IndexSubset\subset \IndexSet} \FunctionBigF\np{\IndexSubset}
      \SupportDualNorm{\TripleNorm{z_{(\IndexSubset)}}}{\IndexSubset} \leq 
      \alpha \TripleNorm{\primal}
      } }
      \fonctionprimal\np{\primal} 
      \eqfinv
      \\
      &= 
        \min_{
        \substack{
        z \in  \RR_{\IndexSet}^{\backslash\emptyset} 
      \\
      \sum_{\emptyset \subsetneq\IndexSubset\subset \IndexSet}
      \SupportDualNorm{\TripleNorm{z_{(\IndexSubset)}}}{\IndexSubset} \leq
      \TripleNorm{\sum_{\emptyset \subsetneq\IndexSubset\subset \IndexSet} z_{(\IndexSubset)}}
      \\
      \sum_{\emptyset \subsetneq\IndexSubset\subset \IndexSet} \FunctionBigF\np{\IndexSubset}
      \SupportDualNorm{\TripleNorm{z_{(\IndexSubset)}}}{\IndexSubset} \leq 
      \alpha
      \TripleNorm{\sum_{\emptyset \subsetneq\IndexSubset\subset \IndexSet} z_{(\IndexSubset)}}
      }
      }
      \fonctionprimal\Bp{\sum_{\emptyset \subsetneq\IndexSubset\subset \IndexSet} z_{(\IndexSubset)} }
      \eqfinp
    \end{align}
  \end{subequations}

\subsection{Upper and lower bounds for nondecreasing \FSM\ as norm ratios}
\label{Upper_and_lower_bounds_for_the_support_mapping_as_norm_ratios}

We now show that the variational formulation obtained
in~\S\ref{Variational_formulation_for_the_pseudo_norm}
yields a family of lower and upper bounds for suitable nondecreasing
\FSM\, as a ratio between two norms --- the denominator norm being any
orthant-strictly monotonic norm with orthant-strictly monotonic dual norm.

\begin{proposition}
  Let $\TripleNorm{\cdot}$ be the source norm with associated families 
 \( \sequence{\TopDualNorm{\TripleNorm{\cdot}}{\IndexSubset}}{\IndexSubset \subset \IndexSet} \)
  of generalized local-top-$\IndexSubset$ dual~norms
and
 \( \sequence{\SupportDualNorm{\TripleNorm{\cdot}}{\IndexSubset}}{\IndexSubset \subset \IndexSet} \) 
of generalized local-$\IndexSubset$-support dual~norms,
  as in Definition~\ref{de:top_dual_norm}. 

Suppose that both the norm $\TripleNorm{\cdot}$ and the dual norm $\TripleNormDual{\cdot}$
  are orthant-strictly monotonic.
  Let \( \FunctionBigF : 2^\IndexSet \to \RR_+ \) be a nondecreasing nonnegative
  set function,
  such that \( \FunctionBigF\np{\IndexSubset} > \FunctionBigF\np{\emptyset}=0 \) for all
  $\emptyset \subsetneq \IndexSubset\subset \IndexSet$.
  Then, we have the inequalities
  \begin{equation}
    \frac{ \SupportDualNormF{\TripleNorm{\primal}}{\FunctionBigF} }{ \TripleNorm{\primal} } 
    \leq 
    \FunctionBigF\bp{ \SupportMapping \np{\primal} }
    \leq 
    \min_{\IndexSubset \subset \IndexSet}
    \frac{ \FunctionBigF\np{\IndexSubset} \SupportDualNorm{\TripleNorm{\primal}}{\IndexSubset} }
    { \TripleNorm{\primal} } 
    \eqsepv \forall \primal \in \RR^d\backslash\{0\}
    \eqfinv
    \label{eq:Variational_formulation_for_the_pseudo_norm_inequality_varphi}  
  \end{equation}
  where the norm \( \SupportDualNormF{\TripleNorm{\cdot}}{\FunctionBigF} \)
  is characterized
  \begin{subequations}
    \label{four-eq:TopDual_norm_varphi}
    \begin{itemize}
    \item 
      either by its dual norm \(
      \TopDualNormF{\TripleNorm{\cdot}}{\FunctionBigF}
      = \bp{ \SupportDualNormF{\TripleNorm{\cdot}}{\FunctionBigF} }^{\star} \) 
      which has unit ball
      \( \bigcap_{ \IndexSubset \subset \IndexSet } \FunctionBigF\np{\IndexSubset} \TopDualNorm{\TripleNormBall}{\IndexSubset} \),
      that is,
      \begin{align}
        \SupportDualNormF{\TripleNorm{\cdot}}{\FunctionBigF}
        &= 
          \bp{ \TopDualNormF{\TripleNorm{\cdot}}{\FunctionBigF} }^{\star} =
          \sigma_{ \TopDualNormF{\TripleNormBall}{\FunctionBigF} }
          \mtext{ with } \, 
          \TopDualNormF{\TripleNormBall}{\FunctionBigF}=
          \bigcap_{ \IndexSubset \subset \IndexSet } \FunctionBigF\np{\IndexSubset}
          \TopDualNorm{\TripleNormBall}{\IndexSubset}
          \eqfinv
          \label{eq:TopDual_norm_varphi}
          \intertext{or, equivalently, by the formula} 
        \bp{ \SupportDualNormF{\TripleNorm{\cdot}}{\FunctionBigF} }^{\star}\np{\dual} 
        &=
          \TopDualNormF{\TripleNorm{\dual}}{\FunctionBigF} 
          =
          \sup_{ \IndexSubset \subset \IndexSet }
          \frac{\TopDualNorm{\TripleNorm{\dual}}{\IndexSubset}}%
          { \FunctionBigF\np{\IndexSubset} }
          \eqsepv \forall \dual \in \RR^d 
          \eqfinv
          \label{eq:TopDual_norm_varphi_as_a_supremum}
      \end{align}
    \item 
      or by the inf-convolution
      \begin{align}
        \SupportDualNormF{\TripleNorm{\cdot}}{\FunctionBigF}
        &=
          \bigbox_{ \IndexSubset \subset \IndexSet }
          \Bp{ \FunctionBigF\np{\IndexSubset} 
          \SupportDualNorm{\TripleNorm{\cdot}}{\IndexSubset} }
          \eqfinv
          \\
        \text{that is,}\quad
        \SupportDualNormF{\TripleNorm{\primal}}{\FunctionBigF} 
        &=
          \inf_{ \substack{%
          z \in \RR_{\IndexSet}^{\backslash\emptyset}
        \\
        \sum_{ \IndexSubset=1 }^{ d } z^{(\IndexSubset)} = \primal  } }
        \sum_{\emptyset \subsetneq \IndexSubset\subset \IndexSet } \FunctionBigF\np{\IndexSubset}
        \SupportDualNorm{\TripleNorm{z^{(\IndexSubset)}}}{\IndexSubset} 
        \eqsepv \forall \primal \in \RR^d 
        \eqfinp
        \label{eq:TopDual_norm_flat_inf-convolution}
      \end{align}
    \end{itemize}
  \end{subequations}
\end{proposition}

\begin{proof}
  The equivalence between the four equivalent formulations in Equation~\eqref{four-eq:TopDual_norm_varphi}
  is straightforward
  (See~\cite[Proposition~\ref{CAPRA-pr:coordinate_norm_varphi}]{Chancelier-DeLara:2020_CAPRA_OPTIMIZATION}
  for a similar proof) as it is immediate to
  verify that $\SupportDualNormF{\TripleNorm{\cdot}}{\FunctionBigF}$ is a norm since
  \( \FunctionBigF\np{\IndexSubset} > \FunctionBigF\np{\emptyset}=0 \) for all
  $\emptyset \subsetneq \IndexSubset\subset \IndexSet$.

  Now, the two inequalities in Equation~\eqref{eq:Variational_formulation_for_the_pseudo_norm_inequality_varphi}
  follow from Equation~\eqref{eq:Variational_formulation_for_the_pseudo_norm},
as the assumptions of
Theorem~\ref{th:Variational_formulation_for_the_pseudo_norm} are satisfied.

To get the right hand side (upper bound) inequality
in Equation~\eqref{eq:Variational_formulation_for_the_pseudo_norm_inequality_varphi},
it suffices to consider the
Equality~\eqref{eq:Variational_formulation_for_the_pseudo_norm}
and to deduce an upper bound by reducing the set over which the minimization is
done to all the vectors
 \( z \in \RR_{\IndexSet}^{\backslash\emptyset} \) such that 
  \( z_{(\IndexSubset)}=0 \) for \( \IndexSubset \neq L \) and \(
  z_{(L)}=\primal \), for all \( L \subset \IndexSet \).

To get the left hand side (lower bound) inequality in
  Equation~\eqref{eq:Variational_formulation_for_the_pseudo_norm_inequality_varphi},
it suffices to consider the
Equality~\eqref{eq:Variational_formulation_for_the_pseudo_norm}
and to deduce a lower bound by extending the set over which the minimization is
done to all the vectors
\( z \in \RR_{\IndexSet}^{\backslash\emptyset} \) such that 
\( \sum_{\emptyset \subsetneq\IndexSubset\subset \IndexSet} z_{(\IndexSubset)} =
\primal \). Then, we recognize the inf-convolution formulation in Equation~\eqref{eq:TopDual_norm_flat_inf-convolution}.
\smallskip

  This ends the proof.  
\end{proof}

Equation~\eqref{eq:Variational_formulation_for_the_pseudo_norm_inequality_varphi}
can also we written as 
\begin{equation}
  \SupportDualNormF{\TripleNorm{\primal}}{\FunctionBigF}
  \leq 
  \FunctionBigF\bp{ \SupportMapping \np{\primal} } \TripleNorm{\primal}
  \leq 
  \min_{\IndexSubset \subset \IndexSet}
  \Bp{ \FunctionBigF\np{\IndexSubset}
    \SupportDualNorm{\TripleNorm{\primal}}{\IndexSubset} }
  \eqsepv \forall \primal \in \RR^d 
  \eqfinp
  \label{eq:Variational_formulation_for_the_pseudo_norm_inequality_varphi_bis}  
\end{equation}
When the norm~$\TripleNorm{\cdot}$ is the $\ell_p$-norm ($1<p<+\infty$),
the left hand side inequality 
in~\eqref{eq:Variational_formulation_for_the_pseudo_norm_inequality_varphi_bis}  
coincides with the inequality obtained
in \cite[Proposition~2]{Obozinski-Bach:hal-01412385};
moreover, Equation~\eqref{eq:TopDual_norm_varphi_as_a_supremum} corresponds to 
\cite[Equation~(2)]{Obozinski-Bach:hal-01412385}, and
Equation~\eqref{eq:TopDual_norm_flat_inf-convolution} to 
\cite[Equation~(3)]{Obozinski-Bach:hal-01412385}.
So, we extend the results in \cite[Proposition~2]{Obozinski-Bach:hal-01412385}
beyond the $\ell_p$-norms, and we also provide an upper bound
(right hand side inequality 
in~\eqref{eq:Variational_formulation_for_the_pseudo_norm_inequality_varphi_bis}).

\section{Conclusion}
\label{Conclusion}

The combinatorial expression of the support mapping
makes it difficult to handle it as such in continuous optimization problems,
and we have seen that the Fenchel conjugacy is not adapted.
  In this paper, we have introduced a class of conjugacies that 
  are suitable for functions of the support (\FSM). Each conjugacy is induced by a
  \Capra-coupling that depends on a given source norm on~$\mathbb{R}^d$.
Our main result is that any nondecreasing finite-valued \FSM\
is a \Capra-convex function (that is, is equal to its \Capra-biconjugate)
when both the source norm and its dual norm are orthant-strictly monotonic.
We have also shown the surprising consequence that any nondecreasing finite-valued \FSM\
coincides, on the unit sphere of the source norm, with a proper convex lsc function.
From there, we have obtained exact variational formulations for 
normalized nondecreasing finite-valued \FSM.
For this purpose, we have introduced
sequences of generalized local-top-$\IndexSubset$ and local-$\IndexSubset$-support dual~norms,
generated from the source norm on~$\RR^d$.

The reformulations that we propose for exact sparse optimization problems
make use of $2^d$ new (latent) vectors, making a direct numerical implementation
out of reach. However, the variational formulation may suggest approximations of the \FSM\,
or algorithms making use of the partial convexity that 
our analysis has put to light.
Moreover, we have provided expressions for the \Capra-subdifferential of 
nondecreasing finite-valued \FSM, which can inspire ``gradient-like'' algorithms.
In all cases, the variational formulation obtained yields
a new family of lower and upper bounds for 
null at zero nondecreasing finite-valued \FSM\, as a ratio between two norms;
this may lead to new smooth sparsity inducing terms, proxies for the \FSM.
\bigskip

\textbf{Acknowledgements.}
We want to thank 
Guillaume Obozinski for several discussions on the topic.

\appendix

\section{Properties of the constant along primal rays coupling (\Capra)}
\label{Properties_Constant_along_primal_rays_coupling}

We recall properties of the \Capra-conjugacy, 
induced by the coupling~\Capra\ in Definition~\ref{de:coupling_CAPRA}
(see~\S\ref{Background_on_Fenchel-Moreau_conjugacies}
and \cite{Chancelier-DeLara:2020_CAPRA_OPTIMIZATION}).

Here are expressions for \Capra-conjugates and biconjugates.
We recall that, in convex analysis, 
for any subset $S \subset \RR^d$, \( \sigma_{S} : \RR^d \to \barRR\) denotes the 
\emph{support function\footnote{%
    The support \emph{function} (of a subset) in~\eqref{eq:support_function}
    should be distinguished from the support \emph{mapping}
    in~\eqref{eq:support_mapping}.
  }
  of the subset~$S$}:
\begin{equation}
  \sigma_{S}\np{\dual} = 
  \sup_{\primal\in S} \proscal{\primal}{\dual}
  \eqsepv \forall \dual \in \RR^d
  \eqfinp
  \label{eq:support_function}
\end{equation}

\begin{proposition}(\cite[Proposition~\ref{CAPRA-pr:CAPRA_Fenchel-Moreau_conjugate}]{Chancelier-DeLara:2020_CAPRA_OPTIMIZATION})
  For any function \( \fonctiondual : \RR^d \to \barRR \), 
  the $\CouplingCapra'$-Fenchel-Moreau conjugate is 
  the function \( \SFMr{\fonctiondual}{\CouplingCapra}: \RR^d \to \barRR \)
  given by 
  \begin{equation}
    \SFMr{\fonctiondual}{\CouplingCapra}=
    \LFM{ \fonctiondual }~\circ~\normalized
    \eqfinp 
    \label{eq:CAPRA'_Fenchel-Moreau_conjugate}
  \end{equation}
  For any function \( \fonctionprimal : \RR^d \to \barRR \), 
  the $\CouplingCapra$-Fenchel-Moreau conjugate is 
  the function \( \SFM{\fonctionprimal}{\CouplingCapra}: \RR^d \to \barRR \)
  given by 
  \begin{equation}
    \SFM{\fonctionprimal}{\CouplingCapra}=
    \LFM{ \bp{\InfimalPostComposition{\normalized}{\fonctionprimal}} }
    \eqfinv 
    \label{eq:CAPRA_Fenchel-Moreau_conjugate}
  \end{equation}
  where the \infimalpostcomposition\
  \( \InfimalPostComposition{\normalized}{\fonctionprimal} \) 
  has the expression\footnote{%
    The name ``conditional infimum'' comes from \cite{Witsenhausen:1975b}.
    We chose the notation to stress the analogy with a conditional expectation.
    We adopt the convention that \( \inf \emptyset = +\infty \).}
  \begin{equation}
    \InfimalPostComposition{\normalized}{\fonctionprimal}\np{\primal}
    =
    \inf\defset{\fonctionprimal\np{\primal'}}{
      \normalized\np{\primal'}=\primal}
    =
    \begin{cases}
      \inf_{\lambda > 0} \fonctionprimal\np{\lambda\primal}
      & \text{if } \primal \in \TripleNormSphere  \cup \{0\} 
      \eqfinv 
      \\
      +\infty  
      & \text{if } \primal \not\in \TripleNormSphere  \cup \{0\} 
      \eqfinv 
    \end{cases}
    \label{eq:CAPRA_InfimalPostComposition}
  \end{equation}
  and the $\CouplingCapra$-Fenchel-Moreau biconjugate is 
  the function \( \SFMbi{\fonctionprimal}{\CouplingCapra}: \RR^d \to \barRR \)
  given by
  \begin{equation}
    \SFMbi{\fonctionprimal}{\CouplingCapra}
    = 
    \LFMr{ \bp{ \SFM{\fonctionprimal}{\CouplingCapra} } }
    \circ \normalized
    =
    \LFMbi{ \bp{\InfimalPostComposition{\normalized}{\fonctionprimal}} }
    \circ \normalized
    \eqfinp
    \label{eq:CAPRA_Fenchel-Moreau_biconjugate}
  \end{equation}
  For any subset \( \Uncertain \subset \UNCERTAIN \), 
  the $\CouplingCapra$-Fenchel-Moreau conjugate of the characteristic
  function~\eqref{eq:characteristic_function} of the set~$\Uncertain$
  is given by
  \begin{equation}
    \SFM{ \delta_{\Uncertain} }{\CouplingCapra}
    = \sigma_{ \normalized\np{\Uncertain} } 
    \eqfinp
    \label{eq:one-sided_linear_Fenchel-Moreau_characteristic}
  \end{equation}
  \label{pr:CAPRA_Fenchel-Moreau_conjugate}
\end{proposition}


Here are characterizations of the \Capra-convex functions (see Definition~\ref{de:coupling-convex_function}).

\begin{proposition}(\cite[Proposition~\ref{CAPRA-pr:CAPRA_convex_functions}]{Chancelier-DeLara:2020_CAPRA_OPTIMIZATION})
  \label{pr:CAPRA_convex_functions}  
  A function is $\CouplingCapra$-convex 
  if and only if it is the composition of 
  a closed convex function on~$\RR^d$
  with the normalization mapping~\eqref{eq:normalization_mapping}.
  More precisely, for any function \( \fonctionuncertain : \RR^d \to \barRR \),
  we have the equivalences
  \begin{subequations}
    \begin{align*}
      \fonctionuncertain \textrm{ is }
      \CouplingCapra\textrm{-convex}
       &\Leftrightarrow 
        \fonctionuncertain =
        \SFMbi{\fonctionuncertain}{\CouplingCapra}
      \\
      &\Leftrightarrow
        \fonctionuncertain = \LFMr{ \bp{ \SFM{\fonctionuncertain}{\CouplingCapra}}} 
        \circ \normalized
        \textrm{ (where } \LFMr{ \bp{ \SFM{\fonctionuncertain}{\CouplingCapra} } } 
        \textrm{ is a closed convex function) }
      \\
      &\Leftrightarrow \textrm{ there exists a closed convex function }
        \fonctionprimal: \RR^d  \to \barRR 
        \textrm{ such that }
        \fonctionuncertain = \fonctionprimal \circ \normalized
        \eqfinp
    \end{align*}
  \end{subequations}
\end{proposition}


Following the definition of the subdifferential
of a function with respect to a duality 
in \cite{Akian-Gaubert-Kolokoltsov:2002},
the \emph{\Capra-subdifferential} of 
the function \( \fonctionprimal : \RR^d \to \barRR \) 
at~\( \primal \in  \RR^d \) has the following expressions
\begin{subequations}
  \begin{align}
    \partial_{\CouplingCapra}\fonctionprimal\np{\primal} 
    &=
      \defset{ \dual \in \RR^d }{ %
      \CouplingCapra\np{\primal', \dual} 
      \LowPlus \bp{ -\fonctionprimal\np{\primal'} } 
      \leq 
      \CouplingCapra\np{\primal, \dual} 
      \LowPlus \bp{ -\fonctionprimal\np{\primal} } 
      \eqsepv \forall \primal' \in \RR^d }
      \label{eq:Capra-subdifferential_a}
    \\
    &=
      \defset{ \dual \in \RR^d }{ %
      \SFM{\fonctionprimal}{\CouplingCapra}\np{\dual}=
      \CouplingCapra\np{\primal, \dual} 
      \LowPlus \bp{ -\fonctionprimal\np{\primal} } }
      \label{eq:Capra-subdifferential_b}
    \\
    &=
      \defset{ \dual \in \RR^d }{ %
      \LFM{ \bp{\InfimalPostComposition{\normalized}{\fonctionprimal}} }\np{\dual}=
      \proscal{\normalized\np{\primal}}{\dual} 
      \LowPlus \bp{ -\fonctionprimal\np{\primal} } }
      \eqfinv
      \label{eq:Capra-subdifferential_c}
      \intertext{so that, thanks to the definition~\eqref{eq:normalization_mapping}
      of the normalization mapping~$\normalized$, we deduce that}
      \partial_{\CouplingCapra}\fonctionprimal\np{0} 
    &=
      \defset{ \dual \in \RR^d }{ %
      \LFM{ \bp{\InfimalPostComposition{\normalized}{\fonctionprimal}} }\np{\dual}=
      -\fonctionprimal\np{0} }
      \label{eq:Capra-subdifferential_zero}
    \\
    \partial_{\CouplingCapra}\fonctionprimal\np{\primal} 
    &=
      \defset{ \dual \in \RR^d }{ %
      \LFM{ \bp{\InfimalPostComposition{\normalized}{\fonctionprimal}} }\np{\dual}=
      \frac{ \proscal{\primal}{\dual} }{ \TripleNorm{\primal} } 
      \LowPlus \bp{ -\fonctionprimal\np{\primal} } }
      \eqsepv \forall \primal \in \RR^d\backslash\{0\} 
      \eqfinp 
      \label{eq:Capra-subdifferential_neq_zero}
  \end{align}
  \label{eq:Capra-subdifferential}
\end{subequations}

\section{Material on local-coordinate-$\IndexSubset$ and generalized local-top-$\IndexSubset$ norms}
\label{Coordinate-k_and_dual_coordinate-k_norms}

We start, 
in \S\ref{Orthant-monotonic_and_orthant-strictly_monotonic_norms}, by providing background on 
orthant-monotonic and orthant-strictly monotonic norms.
Then, in~\S\ref{Definition_of_coordinate-k_and_dual_coordinate-k_norms},
we introduce local-coordinate-$\IndexSubset$ norms
and dual local-coordinate-$\IndexSubset$ norms, 
and in~\S\ref{Definitions_of_generalized_top-k_and_k-support_dual_norms},
we introduce generalized local-top-$\IndexSubset$ and
  local-$\IndexSubset$-support dual~norms; they are norms on the
subspaces~$\FlatRR_\IndexSubset$ of $\RR^d$ in~\eqref{eq:FlatRR} constructed from a source norm.

\subsection{Orthant-monotonic and orthant-strictly monotonic norms}
\label{Orthant-monotonic_and_orthant-strictly_monotonic_norms}

\subsubsubsection{Dual norms}

We recall that the expression
\( \TripleNorm{\dual}_\star = 
  \sup_{ \TripleNorm{\primal} \leq 1 } \proscal{\primal}{\dual} \), 
  \( \forall \dual \in \RR^d \),
defines a norm on~$\RR^d$, 
called the \emph{dual norm} \( \TripleNormDual{\cdot} \).
  By definition of the dual norm, we have the inequality
  \begin{equation}
    \proscal{\primal}{\dual} \leq 
    \TripleNorm{\primal} \times \TripleNormDual{\dual} 
    \eqsepv \forall  \np{\primal,\dual} \in \RR^d \times \RR^d 
    \eqfinp 
    \label{eq:norm_dual_norm_inequality}
  \end{equation}
We denote the unit sphere~$\TripleNormDualSphere$ and the unit ball~$\TripleNormDualBall$
of the dual norm~$\TripleNormDual{\cdot}$ by 
\begin{equation}
      \TripleNormDualSphere = 
      \defset{\dual \in \RR^d}{\TripleNormDual{\dual} = 1} 
      \eqsepv
    \TripleNormDualBall = 
      \defset{\dual \in \RR^d}{\TripleNormDual{\dual} \leq 1} 
      \eqfinp
\label{eq:triplenorm_Dual_unit_sphere}
\end{equation}
%
Denoting by $\sigma_S$ the \emph{support function} of the set~$S \subset \RR^d$
  (\( \sigma_S\np{\dual}=\sup_{\primal \in S} \proscal{\primal}{\dual} \)),
  we have
  \begin{equation}
    \TripleNorm{\cdot} = \sigma_{\TripleNormDualBall} = \sigma_{\TripleNormDualSphere} 
    \mtext{ and } 
    \TripleNormDual{\cdot} = \sigma_{\TripleNormBall} = \sigma_{\TripleNormSphere}
    \eqfinv
    \label{eq:norm_dual_norm}
  \end{equation}
  where \( \TripleNormDualBall =\TripleNormBall^{\odot} 
    = \defset{\dual \in \RR^d}{\proscal{\primal}{\dual} \leq 1 
      \eqsepv \forall \primal \in \TripleNormBall } \)
  is the polar set~\( \TripleNormBall^{\odot} \) 
  of the unit ball~\( \TripleNormBall \).

We recall properties of orthant-monotonic 
and orthant-strictly monotonic norms.

\begin{proposition}(\cite[Proposition~\ref{OSM-pr:orthant-monotonic}]{Chancelier-DeLara:2020_OSM},
  \cite[Theorem~ 2.26]{Gries:1967},
  \cite[Theorem 3.2]{Marques_de_Sa-Sodupe:1993})    
  Let $\TripleNorm{\cdot}$ be a norm on~$\RR^d$.
  The following assertions are equivalent.
  \begin{enumerate}
  \item
    \label{it:OM}
    The norm $\TripleNorm{\cdot}$ is orthant-monotonic.
  \item
    \label{it:OM_dual}
    The dual norm $\TripleNormDual{\cdot}$ is orthant-monotonic.
  \item 
    \label{it:ICS}
    The norm $\TripleNorm{\cdot}$ is \emph{increasing with the coordinate subspaces},
    in the sense that, for any \( \primal \in \RR^d \)
    and any \( J \subset \IndexSubset \subset\IndexSet \),
    we have $ \TripleNorm{\primal_{J}} \leq \TripleNorm{\primal_\IndexSubset}$.
  \end{enumerate}
  \label{pr:orthant-monotonic}
\end{proposition}


\begin{proposition}(\cite[Proposition~\ref{OSM-pr:orthant-strictly_monotonic}]{Chancelier-DeLara:2020_OSM})
  \label{pr:orthant-strictly_monotonic}
  Let $\TripleNorm{\cdot}$ be a norm on~$\RR^d$.
  The following assertions are equivalent.
  \begin{enumerate}
  \item
    \label{it:OSM}
    The norm $\TripleNorm{\cdot}$ is orthant-strictly monotonic.
  \item
    \label{it:SICS}
    The norm $\TripleNorm{\cdot}$ is \emph{strictly increasing
      with the coordinate subspaces} 
    in the sense that,
    for any \( \primal \in \RR^d \)
    and any \( J \subsetneq \IndexSubset \subset\IndexSet \),
    we have $ \primal_J \neq \primal_\IndexSubset \implies
    \TripleNorm{\primal_{J}} < \TripleNorm{\primal_\IndexSubset}$.
  \item
    \label{it:SDC}
    For any vector \( u \in \RR^d\backslash\{0\} \), 
    there exists a vector \( v \in \RR^d\backslash\{0\} \)
    such that \( \Support{v} = \Support{u} \),
    that \( u~\circ~v \geq 0 \), and that 
    $v$ is \( \TripleNorm{\cdot} \)-dual to~$u$, that is,
    \( \proscal{u}{v} = \TripleNorm{u} \times \TripleNormDual{v} \). 
  \end{enumerate}
\end{proposition}

\subsection{Local-coordinate-$\IndexSubset$ and dual local-coordinate-$\IndexSubset$ norms}
\label{Definition_of_coordinate-k_and_dual_coordinate-k_norms}

After their definitions
in~\S\ref{Definition_of_local-coordinate-IndexSubset_and_dual_local-coordinate-IndexSubset_norms},
we study properties of dual local-coordinate-$\IndexSubset$ norms
and of local-coordinate-$\IndexSubset$ norms
in~\S\ref{Properties_of_dual_coordinate-k_norms}.

\subsubsection{Definition of local-coordinate-$\IndexSubset$ and dual
  local-coordinate-$\IndexSubset$ norms}
\label{Definition_of_local-coordinate-IndexSubset_and_dual_local-coordinate-IndexSubset_norms}

For the sake of completeness, we duplicate the
Definition~\ref{de:coordinate_norm} here.
      %
For any subset \( \IndexSubset \subset \IndexSet \), 
we call \emph{local-coordinate-$\IndexSubset$ norm} 
(associated with the source norm~$\TripleNorm{\cdot}$)
the norm \( \CoordinateNorm{\TripleNorm{\cdot}}{\IndexSubset} \)
(on the subspace $\FlatRR_\IndexSubset$ in~\eqref{eq:FlatRR}) 
given by 
\begin{equation*}
  \CoordinateNorm{\TripleNorm{\cdot}}{\IndexSubset} 
= \bp{ \CoordinateNormDual{\TripleNorm{\cdot}}{\IndexSubset} }_\star 
\eqfinv
\end{equation*}
that is, whose dual norm (on the subspace $\FlatRR_\IndexSubset$) is the 
\emph{dual local-coordinate-$\IndexSubset$ norm}, denoted by
\( \CoordinateNormDual{\TripleNorm{\cdot}}{\IndexSubset} \), 
with expression
\begin{equation*}
  \CoordinateNormDual{\TripleNorm{\dual}}{\IndexSubset}
  =
  \sup_{\IndexSubsetLess \subset \IndexSubset} \TripleNorm{\dual_{\IndexSubsetLess}}_{\IndexSubsetLess,\star} 
  \eqsepv \forall \dual \in \FlatRR_\IndexSubset
  \eqfinv
\end{equation*}
where the $\SetStar{\IndexSubset}$-norm \( \TripleNorm{\cdot}_{\IndexSubset,\star} \) is given in
Definition~\ref{de:K_norm}.
      %
We will also use the following extension (defined for all \( \dual \in \RR^d \)
and not only for \( \dual \in \FlatRR_\IndexSubset \))
\begin{subequations}
\begin{equation}
  \CoordinateNormDual{\TripleNorm{\dual}}{\IndexSubset}=
\CoordinateNormDual{\TripleNorm{\dual_\IndexSubset}}{\IndexSubset}
\eqsepv \forall \dual \in \RR^d
\eqfinv
\label{eq:dual_coordinate_seminorm_definition}
\end{equation}
which defines a \emph{seminorm}\footnote{%
  A seminorm satisfies all the axioms of a norm, except that other vectors than
  the null vector give zero.}
on~$\RR^d$, and not a norm.
The local-coordinate-$\IndexSubset$ norm 
 \( \CoordinateNorm{\TripleNorm{\cdot}}{\IndexSubset} \) 
on~$\FlatRR_\IndexSubset$ can also be extended 
to a seminorm on~$\RR^d$ by 
\begin{equation}
  \CoordinateNorm{\TripleNorm{\primal}}{\IndexSubset} =
\CoordinateNorm{\TripleNorm{\primal_\IndexSubset}}{\IndexSubset} 
\eqsepv \forall \primal\in  \RR^d
  \eqfinp
\label{eq:coordinate_seminorm_definition}
\end{equation}
\end{subequations}
Now, to establish results in Sect.~\ref{The_Capra_conjugacy_and_the_support_mapping},
we provide properties of local-coordinate-$\IndexSubset$ and dual local-coordinate-$\IndexSubset$ norms.

\subsubsection{Properties of dual local-coordinate-$\IndexSubset$ norms
  and of local-coordinate-$\IndexSubset$ norms}
\label{Properties_of_dual_coordinate-k_norms}

\subsubsubsection{Properties of dual local-coordinate-$\IndexSubset$ norms}

\begin{subequations}
  \begin{proposition}\quad
    \label{pr:dual-coordinate-IndexSubset-properties}
    \begin{itemize}
    \item For any \( \IndexSubset \subset \IndexSet \), 
      the dual local-coordinate-$\IndexSubset$ norm satisfies 
      \begin{align}
        \CoordinateNormDual{\TripleNorm{\dual}}{\IndexSubset}
        &=
          \sup_{\IndexSubsetLess \subset \IndexSubset} 
          \sigma_{ \np{ \FlatRR_{\IndexSubsetLess} \cap \TripleNormSphere } }\np{\dual}
          = 
          \sigma_{ \np{\SuppLevelSet{\SupportMapping}{\IndexSubset} \cap \TripleNormSphere} }\np{\dual} 
          =
          \sigma_{ \np{\SuppLevelCurve{\SupportMapping}{\IndexSubset} \cap \TripleNormSphere} }\np{\dual}
          \eqsepv \forall \dual \in \FlatRR_\IndexSubset 
          \eqfinp
          \label{eq:dual_coordinate_norm}
      \end{align}
    \item We have the equality
      \begin{equation}
        \CoordinateNormDual{\TripleNorm{\dual}}{\IndexSet}
        = \TripleNormDual{\dual}
        \eqsepv \forall \dual \in \RR^d
        \eqfinp
        \label{eq:dual_coordinate_norm-d_equality}
      \end{equation}
    \item 
      The family
      \( \sequence{\CoordinateNormDual{\TripleNorm{\cdot}}{\IndexSubset}}{\IndexSubset \subset \IndexSet} \)
      of dual local-coordinate-$\IndexSubset$ norms in Definition~\ref{de:coordinate_norm}
      is nondecreasing, that is, 
      \begin{equation}
        \IndexSubsetLess \subset \IndexSubset \subset \IndexSet \implies
        \CoordinateNormDual{\TripleNorm{\dual}}{\IndexSubsetLess}
        \leq \CoordinateNormDual{\TripleNorm{\dual}}{\IndexSubset}
        \leq \CoordinateNormDual{\TripleNorm{\dual}}{\IndexSet}
        = \TripleNormDual{\dual}
            \eqsepv \forall \dual \in \FlatRR_{\IndexSubsetLess} 
        \eqfinp
        \label{eq:dual_coordinate_norm_inequalities}
      \end{equation}
    \item 
      The family
      \( \sequence{ \CoordinateNormDual{\TripleNormBall}{\IndexSubset} }{\IndexSubset \subset \IndexSet} \)
      of units balls of the dual local-coordinate-$\IndexSubset$ norms 
      in Definition~\ref{de:coordinate_norm} 
      is nonincreasing, that is, 
      \begin{equation}
      \IndexSubsetLess \subset \IndexSubset \subset \IndexSet \implies
          \CoordinateNormDual{\TripleNormBall}{\IndexSubsetLess}
\supset 
\CoordinateNormDual{\TripleNormBall}{\IndexSubset}
 \supset 
 \CoordinateNormDual{\TripleNormBall}{\IndexSet}
 =\TripleNormDualBall
          \eqfinp 
        \label{eq:dual_coordinate_norm_unit-balls_inclusions}
      \end{equation}
%
      %
    \end{itemize}
  \end{proposition}
\end{subequations}

\begin{proof} \quad
  
  \noindent $\bullet$ For any \( \dual \in \RR^d \), we have
  \begin{align*}
    \CoordinateNormDual{\TripleNorm{\dual}}{\IndexSubset}
    &=
      \sup_{\IndexSubsetLess \subset \IndexSubset} \TripleNorm{\dual_{\IndexSubsetLess}}_{\IndexSubsetLess,\star} 
      \tag{by definition~\eqref{eq:dual_coordinate_norm_definition}}
    \\
    &=
      \sup_{\IndexSubsetLess \subset \IndexSubset}
      \sigma_{ \np{ \FlatRR_{\IndexSubsetLess} \cap \TripleNormSphere } }\np{\dual}
      \tag{by~\eqref{eq:K_star=sigma}}
    \\
    &=
      \sigma_{ \bigcup_{\IndexSubsetLess \subset \IndexSubset}
      \np{ \FlatRR_{\IndexSubsetLess} \cap \TripleNormSphere } }\np{\dual}
    \\
    &=
      \sigma_{ \np{\SuppLevelSet{\SupportMapping}{\IndexSubset} \cap \TripleNormSphere} } \np{\dual}\eqfinp
      \tag{as we have \( \SuppLevelSet{\SupportMapping}{\IndexSubset} \cap \TripleNormSphere 
      = \bigcup_{ \IndexSubsetLess \subset \IndexSubset} \np{ \FlatRR_{\IndexSubsetLess} \cap \TripleNormSphere } \)
      by~\eqref{eq:support_mapping_level_set_FlatRR}}
  \end{align*}
  
  To finish, we will now prove that 
  \( \sigma_{ \SuppLevelSet{\SupportMapping}{\IndexSubset} \cap \TripleNormSphere } =
  \sigma_{ \LevelCurve{\SupportMapping}{\IndexSubset} \cap \TripleNormSphere } \).
  For this purpose, we show in two steps that 
  \(      \SuppLevelSet{\SupportMapping}{\IndexSubset} \cap \TripleNormSphere 
  =
  \overline{ \LevelCurve{\SupportMapping}{\IndexSubset} \cap \TripleNormSphere } \).
  
  First, we establish that 
  \( \overline{ \LevelCurve{\SupportMapping}{\IndexSubset} } 
= \SuppLevelSet{\SupportMapping}{\IndexSubset} \). 
  The inclusion \( \overline{ \LevelCurve{\SupportMapping}{\IndexSubset} } 
  \subset \SuppLevelSet{\SupportMapping}{\IndexSubset} \) is easy.
  Indeed, that the level set \( \SuppLevelSet{\SupportMapping}{\IndexSubset} \) is closed is straightforward:
  if a sequence $\np{\primal_n}_{n \in \NN} \in
  \SuppLevelSet{\SupportMapping}{\IndexSubset}$ is converging towards $\primal$,
  we get that $(\primal)_{V\backslash \IndexSubset}=0$ 
  since $(\primal_n)_{V\backslash \IndexSubset}=0$ for all $n\in \NN$;
  thus, we obtain that $\primal\in \SuppLevelSet{\SupportMapping}{\IndexSubset}$.
  There remains to prove the reverse inclusion
  \( \SuppLevelSet{\SupportMapping}{\IndexSubset} \subset \overline{ \LevelCurve{\SupportMapping}{\IndexSubset} } \).
  For this purpose, we consider 
  \( \primal \in \SuppLevelSet{\SupportMapping}{\IndexSubset} \). 
  If \( \primal \in \LevelCurve{\SupportMapping}{\IndexSubset} \), obviously 
  \( \primal \in \overline{ \LevelCurve{\SupportMapping}{\IndexSubset} } \).
  Therefore, we suppose that \( \Support{\primal}=L \subsetneq \IndexSubset\) and
  we consider the sequence $\np{\primal_n}_{n \in \NN^*}$ defined by
  $\primal_n = \primal + (1/n)\1_{\IndexSubset\backslash L}$, where $\1$ is the
  vector of~$\RR^{d}$ made of ones. We have that
  $\Support{\primal_n}=\IndexSubset$ for all $n\in \NN$ and $\primal_n \to \primal$ when $n$ goes to infinity, thus
  $\primal\in \overline{ \LevelCurve{\SupportMapping}{\IndexSubset} }$.
  This proves that 
  \( \SuppLevelSet{\SupportMapping}{\IndexSubset} \subset \overline{ \LevelCurve{\SupportMapping}{\IndexSubset} } \).

  Second, we prove that \( \SuppLevelSet{\SupportMapping}{\IndexSubset} \cap \TripleNormSphere 
  = \overline{ \LevelCurve{\SupportMapping}{\IndexSubset} \cap \TripleNormSphere } \).
  The inclusion 
  \( \overline{ \LevelCurve{\SupportMapping}{\IndexSubset} \cap \TripleNormSphere } 
  \subset \SuppLevelSet{\SupportMapping}{\IndexSubset} \cap \TripleNormSphere \) is easy.
  Indeed, we have 
  \( \overline{ \LevelCurve{\SupportMapping}{\IndexSubset} \cap \TripleNormSphere } \subset 
  \overline{ \LevelCurve{\SupportMapping}{\IndexSubset} } \cap \overline{ \TripleNormSphere } = 
  \SuppLevelSet{\SupportMapping}{\IndexSubset} \cap \TripleNormSphere \)
since we have just proved that \( \overline{ \LevelCurve{\SupportMapping}{\IndexSubset} } 
= \SuppLevelSet{\SupportMapping}{\IndexSubset} \).
  To prove the reverse inclusion
  \( \SuppLevelSet{\SupportMapping}{\IndexSubset} \cap \TripleNormSphere 
  \subset \overline{ \LevelCurve{\SupportMapping}{\IndexSubset} \cap \TripleNormSphere } \),
  we consider \( \primal \in \SuppLevelSet{\SupportMapping}{\IndexSubset} \cap \TripleNormSphere \).
  As we have just seen that \( \SuppLevelSet{\SupportMapping}{\IndexSubset} = 
  \overline{ \LevelCurve{\SupportMapping}{\IndexSubset} }\), 
  we deduce that \( \primal \in \overline{ \LevelCurve{\SupportMapping}{\IndexSubset} }\).
  Therefore, there exists a sequence
  \( \sequence{z_n}{n\in\NN} \) in \( \LevelCurve{\SupportMapping}{\IndexSubset} \)
  such that \( z_n \to \primal \) when \( n \to +\infty \).
  Since \( \primal \in \TripleNormSphere \), we can always suppose that 
  \( z_n \neq 0 \), for all $n\in\NN$. Therefore \( z_n/\TripleNorm{z_n} \) is well
  defined and, when \( n \to +\infty \), 
  we have \( z_n/\TripleNorm{z_n} \to \primal/\TripleNorm{\primal}=\primal \)
  since \( \primal \in \TripleNormSphere = \defset{\primal \in \PRIMAL}{\TripleNorm{\primal} = 1} \).
  Now, on the one hand, 
  \( z_n/\TripleNorm{z_n} \in \LevelCurve{\SupportMapping}{\IndexSubset} \), for all $n\in\NN$,
  and, on the other hand, \( z_n/\TripleNorm{z_n} \in \TripleNormSphere \).
  As a consequence \( z_n/\TripleNorm{z_n} \in \LevelCurve{\SupportMapping}{\IndexSubset} \cap \TripleNormSphere \),
  and we conclude that \( \primal \in 
  \overline{ \LevelCurve{\SupportMapping}{\IndexSubset} \cap \TripleNormSphere } \). 
  Thus, we have proved that 
  \( \SuppLevelSet{\SupportMapping}{\IndexSubset} \cap \TripleNormSphere 
  \subset \overline{ \LevelCurve{\SupportMapping}{\IndexSubset} \cap \TripleNormSphere } \).
  \medskip

  From \( \SuppLevelSet{\SupportMapping}{\IndexSubset} \cap \TripleNormSphere 
  = \overline{ \LevelCurve{\SupportMapping}{\IndexSubset} \cap \TripleNormSphere } \), we get that 
  \( \sigma_{ \SuppLevelSet{\SupportMapping}{\IndexSubset} \cap \TripleNormSphere } =
  \sigma_{ \overline{ \LevelCurve{\SupportMapping}{\IndexSubset} \cap \TripleNormSphere } } = 
  \sigma_{ \LevelCurve{\SupportMapping}{\IndexSubset} \cap \TripleNormSphere } \).
  \smallskip

  Thus, we have proved all equalities in~\eqref{eq:dual_coordinate_norm}.
  \medskip 

  \noindent $\bullet$ 
  By the equality $\CoordinateNormDual{\TripleNorm{\dual}}{\IndexSubset} =
  \sigma_{ \np{\SuppLevelSet{\SupportMapping}{\IndexSubset} \cap \TripleNormSphere}} \np{\dual}$
  in~\eqref{eq:dual_coordinate_norm}, we get that, for all $\dual \in \RR^d$,
  $\CoordinateNormDual{\TripleNorm{\dual}}{\IndexSet} = 
  \sigma_{ \np{\SuppLevelSet{\SupportMapping}{\IndexSet} \cap \TripleNormSphere}} \np{\dual}
  =\sigma_{ \TripleNormSphere} \np{\dual} =\TripleNormDual{\dual}$, 
  since $\SuppLevelSet{\SupportMapping}{\IndexSet} = \RR^d$ and 
  by~\eqref{eq:norm_dual_norm}.
  Thus, we have proved the equality~\eqref{eq:dual_coordinate_norm-d_equality}.
  \medskip 

  \noindent $\bullet$ 
  The inequalities in~\eqref{eq:dual_coordinate_norm_inequalities}
  easily derive from the very definition~\eqref{eq:dual_coordinate_norm_definition}
  of the dual local-coordinate-$\IndexSubset$ norms~\( \CoordinateNormDual{\TripleNorm{\cdot}}{\IndexSubset}
  \).
  \medskip 

  \noindent $\bullet$
  The inclusions and equality 
  in~\eqref{eq:dual_coordinate_norm_unit-balls_inclusions}
  directly follow from the equality and the inequalities between norms
  in~\eqref{eq:dual_coordinate_norm-d_equality} and~\eqref{eq:dual_coordinate_norm_inequalities}.
  \medskip 

  This ends the proof. 
\end{proof}

\subsubsubsection{Properties of local-coordinate-$\IndexSubset$ norms}

\begin{subequations}
  \begin{proposition}\quad
    \begin{itemize}
    \item 
      For any subset \( \IndexSubset \subset \IndexSet \), the local-coordinate-$\IndexSubset$ norm
      \( \CoordinateNorm{\TripleNorm{\cdot}}{\IndexSubset} \) 
      has unit ball (in the subspace $\FlatRR_\IndexSubset$ in~\eqref{eq:FlatRR}) 
      \begin{equation}
        \CoordinateNorm{\TripleNormBall}{\IndexSubset}
        {=
          \closedconvexhull \np{ \FlatRR_{\IndexSubset} \cap \TripleNormSphere } 
        }
        \subset \FlatRR_\IndexSubset
        \eqfinv
        \label{eq:coordinate_norm_unit_ball_property}
      \end{equation}
      where \( \closedconvexhull\np{S} \) denotes the closed convex hull
      of a subset \( S \subset \RR^d \).
    \item 
      We have the equality
      \begin{equation}
        \CoordinateNorm{\TripleNorm{\primal}}{\IndexSet}
        =
        \TripleNorm{\primal}
        \eqsepv \forall \primal \in \RR^d 
        \eqfinp  
        \label{eq:coordinate_norm-d_equality}
      \end{equation}
    \item 
      The family 
      \( \sequence{\CoordinateNorm{\TripleNorm{\cdot}}{\IndexSubset}}{\IndexSubset \subset \IndexSet} \)
      of local-coordinate-$\IndexSubset$ norms in Definition~\ref{de:coordinate_norm} is nonincreasing, that is,
      \begin{equation}
        \IndexSubsetLess \subset \IndexSubset \subset \IndexSet \implies
        \CoordinateNorm{\TripleNorm{\primal}}{\IndexSubsetLess}
        \geq
        \CoordinateNorm{\TripleNorm{\primal}}{\IndexSubset}
      \geq
      \CoordinateNorm{\TripleNorm{\primal}}{\IndexSet}
      = \TripleNorm{\primal} 
        \eqsepv \forall \primal \in \FlatRR_\IndexSubsetLess
        \eqfinp
        \label{eq:coordinate_norm_inequalities}
      \end{equation}
    \item 
      The family 
      \( \sequence{ \CoordinateNorm{\TripleNormBall}{\IndexSubset} }{\IndexSubset \subset  \IndexSet} \)
      of unit balls of the local-coordinate-$\IndexSubset$ norms 
      in~\eqref{eq:coordinate_norm_unit_ball} is nondecreasing, that is,
      \begin{equation}
       \IndexSubsetLess \subset \IndexSubset \subset \IndexSet \implies
        \CoordinateNorm{\TripleNormBall}{\IndexSubsetLess} 
        \subset 
        \CoordinateNorm{\TripleNormBall}{\IndexSubset} 
        \subset \CoordinateNorm{\TripleNormBall}{\IndexSet} 
        = \TripleNormBall 
       \eqfinp 
        \label{eq:coordinate_norm_unit-balls_inclusions}
      \end{equation}
    \item 
      We have the implication
      \begin{equation}
        \primal \in \FlatRR_{\IndexSubset}
        \implies 
        \TripleNorm{\primal}=
        \CoordinateNorm{\TripleNorm{\primal}}{\IndexSubset} 
        \eqfinp
        \label{eq:coordinate_norm_graded_b}
      \end{equation}
    \end{itemize}
  \end{proposition}
\end{subequations}

\begin{proof} \quad

  \noindent $\bullet$ 
  For any $\dual \in \FlatRR_\IndexSubset$, we have
  \begin{align*}
    \CoordinateNormDual{\TripleNorm{\dual}}{\IndexSubset}
    &=
      \sup_{\IndexSubsetLess \subset \IndexSubset}
      \sigma_{ \np{ \FlatRR_{\IndexSubsetLess} \cap \TripleNormSphere } }\np{\dual}
      \tag{by~\eqref{eq:K_star=sigma}}
    \\
    &=
      \sigma_{ \np{ \FlatRR_{\IndexSubset} \cap \TripleNormSphere } }\np{\dual}
      \intertext{since $\FlatRR_{\IndexSubsetLess} \subset \FlatRR_{\IndexSubset}$ 
for $\IndexSubsetLess\subset \IndexSubset$ and by
      definition~\eqref{eq:support_function}
of the support functions \( \sigma_{ \np{ \FlatRR_{\IndexSubsetLess} \cap \TripleNormSphere } }\)}
    &=
      \sigma_{ \closedconvexhull{ 
      \np{ \FlatRR_{\IndexSubset} \cap \TripleNormSphere } } }\np{\dual}
      \tag{by \cite[Prop.~7.13]{Bauschke-Combettes:2017} }
  \end{align*}
  As  \( \closedconvexhull\bp{ \bigcup_{\IndexSubsetLess \subset \IndexSubset} 
    \np{ \FlatRR_{\IndexSubsetLess} \cap \TripleNormSphere } } \) is a
  closed convex subset of~$\FlatRR_\IndexSubset$, we conclude that 
  \( \CoordinateNorm{\TripleNormBall}{\IndexSubset} =
  \closedconvexhull\bp{ \bigcup_{\IndexSubsetLess \subset \IndexSubset} \np{ \FlatRR_{\IndexSubsetLess} \cap \TripleNormSphere } } \) 
  by~\eqref{eq:norm_dual_norm}.
  Thus, we have proven~\eqref{eq:coordinate_norm_unit_ball_property}.
  \medskip

  \noindent $\bullet$ 
  From the equality~\eqref{eq:dual_coordinate_norm-d_equality},
  we deduce the equality~\eqref{eq:coordinate_norm-d_equality} between the dual norms
  by definition of the dual norm. 
  \medskip
  
  \noindent $\bullet$ 
  The equality and inequalities between norms in~\eqref{eq:coordinate_norm_inequalities}
  easily derive from the inclusions and equality between unit balls 
  in~\eqref{eq:coordinate_norm_unit-balls_inclusions}.
  \medskip

  \noindent $\bullet$ 
  The inclusions and equality between unit balls 
  in~\eqref{eq:coordinate_norm_unit-balls_inclusions}
  directly follow from the inclusions and equality between unit balls 
  in~\eqref{eq:dual_coordinate_norm_unit-balls_inclusions}
  and as the unit ball of the dual norm is 
  \( \CoordinateNorm{\TripleNormBall}{\IndexSubset} =
  \bp{  \CoordinateNormDual{\TripleNormBall}{\IndexSubset} }^{\odot} \),
  the polar set of~\(  \CoordinateNormDual{\TripleNormBall}{\IndexSubset} \).
  \medskip
  
  \noindent $\bullet$ 
  We prove~\eqref{eq:coordinate_norm_graded_b} using the fact that
  $\primal \in \FlatRR_{\IndexSubset}\Leftrightarrow \primal \in \SuppLevelSet{\SupportMapping}{\IndexSubset}$.
  For any \( \primal \in \RR^d \) and
  for any \( \IndexSubset \subset \IndexSet \), we have\footnote{%
    In what follows, by ``or'', we mean the so-called \emph{exclusive or}
    (exclusive disjunction). Thus, every ``or'' should be understood as ``or
    $\dual\not=0$ and''.
    \label{ft:exclusive_or}} 
  \begin{align*}
    \primal \in \SuppLevelSet{\SupportMapping}{\IndexSubset} 
    &\Leftrightarrow 
      \primal=0 \text{ or }
      \frac{\primal}{\TripleNorm{\primal}} \in\SuppLevelSet{\SupportMapping}{\IndexSubset}
      \intertext{by 0-homogeneity~\eqref{eq:support_mapping_is_0-homogeneous}
      of the $\SupportMapping$ mapping, and
      by definition~\eqref{eq:support_mapping_level_set}
      of $\SuppLevelSet{\SupportMapping}{\IndexSubset}$}
    &\Leftrightarrow 
      \primal=0 \text{ or }
      \frac{\primal}{\TripleNorm{\primal}} \in 
      \SuppLevelSet{\SupportMapping}{\IndexSubset} \cap \TripleNormSphere 
      \tag{as \( \frac{\primal}{\TripleNorm{\primal}} \in \TripleNormSphere \)}
    \\
    &\Leftrightarrow 
      \primal=0 \text{ or }
      \frac{\primal}{\TripleNorm{\primal}} \in 
      \bigcup_{ {\IndexSubsetLess \subset \IndexSubset}} \np{ \FlatRR_{\IndexSubsetLess} \cap \TripleNormSphere }
      \tag{as \( \SuppLevelSet{\SupportMapping}{\IndexSubset}
      = \bigcup_{ \IndexSubsetLess \subset \IndexSubset } \FlatRR_{\IndexSubsetLess}
      \) by~\eqref{eq:support_mapping_level_set_FlatRR}}
    \\
    &\implies
      \primal=0 \text{ or }
      \frac{\primal}{\TripleNorm{\primal}} \in 
      \CoordinateNorm{\TripleNormBall}{\IndexSubset} 
      \tag{as \( \CoordinateNorm{\TripleNormBall}{\IndexSubset}=
      \closedconvexhull\bp{ \bigcup_{ \IndexSubsetLess \subset \IndexSubset} 
      \np{ \FlatRR_{\IndexSubsetLess} \cap \TripleNormSphere } } \) by~\eqref{eq:coordinate_norm_unit_ball_property}}
    \\
    &\implies
      \primal=0 \text{ or }
      \CoordinateNorm{\TripleNorm{\frac{\primal}{\TripleNorm{\primal}}}}{\IndexSubset}
      \leq 1 
      \intertext{since \( \CoordinateNorm{\TripleNormBall}{\IndexSubset} \) is the unit ball
      of the norm \( \CoordinateNorm{\TripleNorm{\cdot}}{\IndexSubset} \) 
      by~\eqref{eq:coordinate_norm_unit_ball} } 
    &\implies
      \CoordinateNorm{\TripleNorm{\primal}}{\IndexSubset}
      \leq \TripleNorm{\primal}
    \\
    &\implies
      \CoordinateNorm{\TripleNorm{\primal}}{\IndexSubset}
      \leq \TripleNorm{\primal}
      =\CoordinateNorm{\TripleNorm{\primal}}{\IndexSet}
      \tag{where the last equality comes
      from~\eqref{eq:coordinate_norm_inequalities} } 
    \\
    &\implies
      \CoordinateNorm{\TripleNorm{\primal}}{\IndexSubset}
      = \CoordinateNorm{\TripleNorm{\primal}}{\IndexSet}
      \tag{as \( \CoordinateNorm{\TripleNorm{\primal}}{\IndexSet}
\leq \CoordinateNorm{\TripleNorm{\primal}}{\IndexSubset}  \) by~\eqref{eq:coordinate_norm_inequalities}} 
      \eqfinp
  \end{align*}
  Therefore, we have obtained~\eqref{eq:coordinate_norm_graded_b}.
\smallskip

  This ends the proof.
\end{proof}

\subsection{Generalized local-top-$\IndexSubset$ and local-$\IndexSubset$-support dual~norms}
\label{Definitions_of_generalized_top-k_and_k-support_dual_norms}

After the definitions of 
generalized local-top-$\IndexSubset$ and local-$\IndexSubset$-support dual~norms
in~\S\ref{Definition_of_generalized_local-top-IndexSubset_and_local-IndexSubset-support_dual_norms_bis},
we study their properties in~\S\ref{se:properties-of-local-coordinate-K}.
The key Proposition~\ref{pr:properties-of-dual-ccord-IndexSubset-under-orth-mon}
is used in the proof of Proposition~\ref{pr:nonempty_subdifferential}

\subsubsection{Definition of generalized local-top-$\IndexSubset$ and
  local-$\IndexSubset$-support dual~norms}
\label{Definition_of_generalized_local-top-IndexSubset_and_local-IndexSubset-support_dual_norms_bis}

\begin{definition}
  \label{de:top_norm}
  For \( \IndexSubset \subset \IndexSet \), 
we call \emph{generalized local-top-$\IndexSubset$ norm}
(associated with the source norm~$\TripleNorm{\cdot}$)
  the norm (on the subspace $\FlatRR_\IndexSubset$ in~\eqref{eq:FlatRR}) defined by 
  \begin{equation}
    \TopNorm{\TripleNorm{\primal}}{\IndexSubset}
    =
    \sup_{\IndexSubsetLess \subset \IndexSubset} \TripleNorm{\primal_\IndexSubsetLess} 
    =
    \sup_{\IndexSubsetLess \subset \IndexSubset} \TripleNorm{\primal_\IndexSubsetLess}_\IndexSubsetLess 
    \eqsepv \forall \primal \in \FlatRR_\IndexSubset 
    \eqfinp
    \label{eq:top_norm}
  \end{equation}
  We call \emph{generalized local-$\IndexSubset$-support norm}
(associated with the source norm~$\TripleNorm{\cdot}$)
  the dual norm (on the subspace $\FlatRR_\IndexSubset$ in~\eqref{eq:FlatRR})
  of the generalized local-top-$\IndexSubset$ norm, denoted by\footnote{%
    We use the symbol~$\star$ in the superscript to indicate that the generalized
    local-$\IndexSubset$-support norm  \( \SupportNorm{\TripleNorm{\cdot}}{K} \) is a dual norm.
  }
  \( \SupportNorm{\TripleNorm{\cdot}}{\IndexSubset} \):
  \begin{equation}
    \SupportNorm{\TripleNorm{\cdot}}{\IndexSubset} 
    = \bp{ \TopNorm{\TripleNorm{\cdot}}{\IndexSubset} }_{\star}
    \eqfinp
    \label{eq:support_norm}
  \end{equation}
\end{definition}

Now, we do the same but with the dual norm~$\TripleNormDual{\cdot}$ in lieu of the 
source norm~$\TripleNorm{\cdot}$. For the sake of completeness, we duplicate the
Definition~\ref{de:top_dual_norm} here.

For any subset \( \IndexSubset \subset \IndexSet \), 
we call \emph{generalized local-top-$\IndexSubset$ dual~norm} 
(on the subspace $\FlatRR_\IndexSubset$ in~\eqref{eq:FlatRR}) 
the local norm defined by 
\begin{equation}
  \TopDualNorm{\TripleNorm{\dual}}{\IndexSubset}
  =
  \sup_{\IndexSubsetLess \subset \IndexSubset} \TripleNormDual{\dual_{\IndexSubsetLess}} 
  =
\sup_{\IndexSubsetLess \subset \IndexSubset} 
  \TripleNorm{\dual_{\IndexSubsetLess}}_{\star,\IndexSubsetLess}
  \eqsepv \forall \dual \in \FlatRR_\IndexSubset 
  \eqfinp
\end{equation}
We call \emph{generalized local-$\IndexSubset$-support dual norm}
the dual norm (on the subspace $\FlatRR_\IndexSubset$)
of the generalized local-top-$\IndexSubset$ dual~norm, denoted by\footnote{%
  We use the symbol~$\star$ in the superscript to indicate that the generalized
  local-$\IndexSubset$-support dual~norm \( \SupportDualNorm{\TripleNorm{\cdot}}{\IndexSubset} \)
  is a dual norm.}
\( \SupportDualNorm{\TripleNorm{\cdot}}{\IndexSubset} \): 
\begin{equation}
  \SupportDualNorm{\TripleNorm{\cdot}}{\IndexSubset} 
  = \bp{ \TopDualNorm{\TripleNorm{\cdot}}{\IndexSubset} }_{\star}
  \eqfinp
\end{equation}

We adopt the convention \( \TopDualNorm{\TripleNorm{\cdot}}{\emptyset} = 0 \) 
(although this is not a norm, but a seminorm). 
We denote the unit sphere and the unit ball 
of the generalized local-$\IndexSubset$-support dual~norm~\(
\SupportDualNorm{\TripleNorm{\cdot}}{k} \)
by
\begin{subequations}
  \begin{align}
    \SupportDualNorm{\TripleNormSphere}{\IndexSubset} 
    &  = 
      \defset{\primal \in \FlatRR_\IndexSubset}{\SupportDualNorm{\TripleNorm{\primal}}{\IndexSubset} = 1} 
      \eqsepv \forall \IndexSubset \subset  \IndexSet
      \eqfinv
      \label{eq:generalized_k-support_norm_unit_sphere}
    \\
    \SupportDualNorm{\TripleNormBall}{\IndexSubset} 
    &  = 
      \defset{\primal \in \FlatRR_\IndexSubset}{\SupportDualNorm{\TripleNorm{\primal}}{\IndexSubset} \leq 1} 
      \eqsepv \forall \IndexSubset \subset  \IndexSet
      \eqfinp
      \label{eq:generalized_k-support_norm_unit_ball}
  \end{align}
  \label{eq:generalized_k-support_norm_unit}
\end{subequations}

We will also use the following extension of the 
generalized local-top-$\IndexSubset$ dual~norm
\( \TopDualNorm{\TripleNorm{\cdot}}{\IndexSubset} \) on~$\FlatRR_\IndexSubset$
into a seminorm on~$\RR^d$ by 
\begin{subequations}
\begin{equation}
  \TopDualNorm{\TripleNorm{\dual}}{\IndexSubset}
  =
  \TopDualNorm{\TripleNorm{\dual_\IndexSubset}}{\IndexSubset}
\eqsepv \forall \dual \in \RR^d
  \eqfinv
\label{eq:top_dual_seminorm}
\end{equation}
and the extension of the 
generalized local-$\IndexSubset$-support dual norm
\( \SupportDualNorm{\TripleNorm{\cdot}}{\IndexSubset} \)
 on~$\FlatRR_\IndexSubset$ 
into a seminorm on~$\RR^d$ by 
\begin{equation}
  \SupportDualNorm{\TripleNorm{\primal}}{\IndexSubset} 
  = 
  \SupportDualNorm{\TripleNorm{\primal_\IndexSubset}}{\IndexSubset} 
\eqsepv \forall \primal \in \RR^d
  \eqfinp
      \label{eq:support_dual_seminorm}
\end{equation}  
\end{subequations}

\subsubsection{Properties of local-coordinate-$\IndexSubset$ and dual local-coordinate-$\IndexSubset$ norms}
\label{se:properties-of-local-coordinate-K}

\begin{proposition}
  Let $\TripleNorm{\cdot}$ be the source norm with associated families
  \( \sequence{\CoordinateNorm{\TripleNorm{\cdot}}{\IndexSubset}}{\IndexSubset\subset \IndexSet} \)
  of local-coordinate-$\IndexSubset$ norms and 
  \( \sequence{\CoordinateNormDual{\TripleNorm{\cdot}}{\IndexSubset}}{\IndexSubset\subset \IndexSet} \)
  of dual local-coordinate-$\IndexSubset$ norms, as in
  Definition~\ref{de:coordinate_norm},
and with associated families
 \( \sequence{\TopDualNorm{\TripleNorm{\cdot}}{\IndexSubset}}{\IndexSubset \subset \IndexSet} \)
  of generalized local-top-$\IndexSubset$ dual~norms and
 \( \sequence{\SupportDualNorm{\TripleNorm{\cdot}}{\IndexSubset}}{\IndexSubset \subset \IndexSet} \) 
of generalized local-$\IndexSubset$-support dual~norms,
  as in Definition~\ref{de:top_dual_norm}. 

We have that 
local-coordinate-$\IndexSubset$ norms are always lower than local-$\IndexSubset$-support dual~norms,
  that is, 
  \begin{subequations}
    \begin{align}
      \CoordinateNorm{\TripleNorm{\primal}}{\IndexSubset} 
      & \leq 
        \SupportDualNorm{\TripleNorm{\primal}}{\IndexSubset}
        \eqsepv \forall \primal \in \FlatRR_\IndexSubset 
        \eqsepv \forall \IndexSubset \subset \IndexSet
        \eqfinv
        \label{eq:coordinate-k_norm_leq_generalized_k-support_norm}
        \intertext{whereas dual local-coordinate-$\IndexSubset$ norms
        are always greater than generalized local-top-$\IndexSubset$ dual~norms, that is,}
        \CoordinateNormDual{\TripleNorm{\dual}}{\IndexSubset}
      & \geq 
        \TopDualNorm{\TripleNorm{\dual}}{\IndexSubset}
        \eqsepv \forall \dual \in \FlatRR_\IndexSubset
        \eqsepv \forall \IndexSubset \subset \IndexSet
        \eqfinp
        \label{eq:dual_coordinate-k_norm_geq_generalized_top-k_norm}
    \end{align}
         \label{eq:coordinate-k_norm_leq_geq_generalized_k-support_norm}
  \end{subequations}
  If the source norm norm $\TripleNorm{\cdot}$ is orthant-monotonic,
  then equalities in~\eqref{eq:coordinate-k_norm_leq_geq_generalized_k-support_norm}
hold true, that is, 
  \begin{equation}
    \TripleNorm{\cdot} \textrm{is orthant-monotonic}
    \implies
    \forall \IndexSubset \subset \IndexSet \quad
    \begin{cases}
      \CoordinateNorm{\TripleNorm{\cdot}}{\IndexSubset} 
      & =
      \SupportDualNorm{\TripleNorm{\cdot}}{\IndexSubset}
      \eqfinv 
      \\[4mm]
      \CoordinateNormDual{\TripleNorm{\cdot}}{\IndexSubset}
      & =
      \TopDualNorm{\TripleNorm{\cdot}}{\IndexSubset}
      \eqfinp 
    \end{cases}
    \label{eq:dual_coordinate-k_norm_=_generalized_top-k_norm}
  \end{equation}
  The two equalities in~\eqref{eq:dual_coordinate-k_norm_=_generalized_top-k_norm}
  are equalities between local norms
  (on the subspace $\FlatRR_\IndexSubset$ in~\eqref{eq:FlatRR}),
  but they remain valid as equalities between seminorms (on~$\RR^d$).
  \label{pr:dual_coordinate-k_norm_=_generalized_top-k_norm}
\end{proposition}

\begin{proof}
  It is easily established that, for any subset \( \IndexSubset \subset \IndexSet \), 
  we have the inequality \( \TripleNorm{\cdot}_{\IndexSubset,\star} 
  \leq \TripleNorm{\cdot}_{\star,\IndexSubset} \)
  \cite[Lemma~\ref{OSM-lem:star_\IndexSubset=sigma_AND_\IndexSubset_star=sigma}]{Chancelier-DeLara:2020_OSM}. 
  From the definition~\eqref{eq:top_dual_norm} of the generalized local-top-$\IndexSubset$
  dual~norm,
  and the definition~\eqref{eq:dual_coordinate_norm_definition} 
  of the dual local-coordinate-$\IndexSubset$ norm,
  we obtain~\eqref{eq:dual_coordinate-k_norm_geq_generalized_top-k_norm}.
  By taking the dual norms, we
  get~\eqref{eq:coordinate-k_norm_leq_generalized_k-support_norm}.

  The norms for which the equality 
  \( \TripleNorm{\cdot}_{\IndexSubset,\star} =
  \TripleNorm{\cdot}_{\star,\IndexSubset} \)
  holds true for all subsets 
  \( \IndexSubset \subset\IndexSet \)
  are the orthant-monotonic norms (\cite[Theorem~ 2.26]{Gries:1967},%
  \cite[Theorem 3.2]{Marques_de_Sa-Sodupe:1993})
  Therefore, if the norm~$\TripleNorm{\cdot}$ is orthant-monotonic,
we get that, for any \( \dual \in \FlatRR_\IndexSubset \), 
\begin{align*}
 \CoordinateNormDual{\TripleNorm{\dual}}{\IndexSubset}
    &=
    \sup_{\IndexSubsetLess \subset \IndexSubset} 
\TripleNorm{\dual_{\IndexSubsetLess}}_{\IndexSubsetLess,\star} 
\tag{by definition~\eqref{eq:dual_coordinate_norm_definition} 
of~$\CoordinateNormDual{\TripleNorm{\dual}}{\IndexSubset}$}
\\
   &=
    \sup_{\IndexSubsetLess \subset \IndexSubset} 
\TripleNorm{\dual_{\IndexSubsetLess}}_{\star,\IndexSubsetLess} 
\tag{because the norm~$\TripleNorm{\cdot}$ is orthant-monotonic}
\\
   &=
\TopDualNorm{\TripleNorm{\dual}}{\IndexSubset}
\tag{by definition~\eqref{eq:top_dual_norm} 
of~$\TopDualNorm{\TripleNorm{\dual}}{\IndexSubset}$}
\end{align*}
Thus, we have obtained the lower equality in the right hand side
of~\eqref{eq:dual_coordinate-k_norm_=_generalized_top-k_norm}. 

The upper equality in the right hand side
of~\eqref{eq:dual_coordinate-k_norm_=_generalized_top-k_norm}
follows by taking the dual norms. 
\smallskip

The two equalities in~\eqref{eq:dual_coordinate-k_norm_=_generalized_top-k_norm}
  remain valid as equalities between seminorms (on~$\RR^d$) 
because of the definitions~\eqref{eq:dual_coordinate_seminorm_definition},
\eqref{eq:coordinate_seminorm_definition},
\eqref{eq:top_dual_seminorm}
and \eqref{eq:support_dual_seminorm}
of the four corresponding seminorms.
\smallskip

  This ends the proof.
\end{proof}


We end this part with
Proposition~\ref{pr:properties-of-dual-ccord-IndexSubset-under-orth-mon},
which is key to prove Proposition~\ref{pr:nonempty_subdifferential}.


\begin{proposition}
  \label{pr:properties-of-dual-ccord-IndexSubset-under-orth-mon}
  We consider a norm $\TripleNorm{\cdot}$ whose dual norm~$\TripleNormDual{\cdot}$ is strictly
  orthant-monotonic. 
  Then, we have the implications (that involve seminorms)
  \begin{subequations}
    \begin{align}
      \bp{\forall \dual \in \RR^d } \quad
      \Support{\dual}=L\subset \IndexSubset \subset \IndexSet
      &\implies 
\TopDualNorm{\TripleNorm{\dual}}{\IndexSubset}
=\CoordinateNormDual{\TripleNorm{\dual}}{\IndexSubset} 
=\CoordinateNormDual{\TripleNorm{\dual}}{L}
=\TopDualNorm{\TripleNorm{\dual}}{L}
     \eqfinv
        \label{eq:dual_tn-L-included-in-IndexSubset}
      \\ 
      \bp{\forall \dual \in \RR^d } \quad
      \Support{\dual}=L\not\subset \IndexSubset \subset \IndexSet 
      &\implies 
\TopDualNorm{\TripleNorm{\dual}}{\IndexSubset}
=\CoordinateNormDual{\TripleNorm{\dual}}{\IndexSubset} 
< \CoordinateNormDual{\TripleNorm{\dual}}{L}
=\TopDualNorm{\TripleNorm{\dual}}{L}
        \eqfinp
        \label{eq:dual_tn-L-not-included-in-IndexSubset}
    \end{align}
  \end{subequations}
\end{proposition}

\begin{proof}
  The proof is in four points.
  \medskip

  \noindent $\bullet$ 
  Since we have supposed that the dual norm $\TripleNormDual{\cdot}$ is strictly
  orthant-monotonic, it is orthant-monotonic and,
  by Proposition~\ref{pr:orthant-monotonic}, we get that the source norm $\TripleNorm{\cdot}$ is
  also orthant-monotonic. 
  Thus, using Equation~\eqref{eq:dual_coordinate-k_norm_=_generalized_top-k_norm}
  in Proposition~\ref{pr:dual_coordinate-k_norm_=_generalized_top-k_norm}, 
  we obtain the equality 
  $\CoordinateNormDual{\TripleNorm{\cdot}}{\IndexSubset} =
  \TopDualNorm{\TripleNorm{\cdot}}{\IndexSubset}$ 
  between the two local norms, for any \( \IndexSubset \subset \IndexSet \).
As this equality between local norms is valid as an equality between seminorms, 
we get that 
\begin{equation}
\CoordinateNormDual{\TripleNorm{\dual}}{\IndexSubset} =
  \TopDualNorm{\TripleNorm{\dual}}{\IndexSubset}
\eqsepv \forall \dual \in \RR^d
  \eqfinp 
    \label{eq:dual_coordinate-k_norm_=_generalized_top-k_norm_seminorm}
  \end{equation}
  \medskip

  \noindent $\bullet$ 
  We prove Implication~\eqref{eq:dual_tn-L-included-in-IndexSubset}. 
  Let $\IndexSubset \subset \IndexSet$ and $\dual \in \RR^d$ be such that
  \( \Support{\dual}=L\subset \IndexSubset \).
  Since $L=\Support{\dual}$, we have that $\dual\in \FlatRR_{L}$. 
  Now, by  Equation~\eqref{eq:dual_coordinate_norm_inequalities} of
  Proposition~\ref{pr:dual-coordinate-IndexSubset-properties}, we have that
  \begin{equation}
    \CoordinateNormDual{\TripleNorm{\dual}}{L}
    \le \CoordinateNormDual{\TripleNorm{\dual}}{\IndexSubset}
    \le \CoordinateNormDual{\TripleNorm{\dual}}{\IndexSet}
    \eqfinp
    \label{eq:yineq}
  \end{equation}
  Thus, to prove~\eqref{eq:dual_tn-L-included-in-IndexSubset} it is enough to prove that
  $\CoordinateNormDual{\TripleNorm{\dual}}{L} = \CoordinateNormDual{\TripleNorm{\dual}}{\IndexSet}$
  which, by
  Equation~\eqref{eq:dual_coordinate-k_norm_=_generalized_top-k_norm_seminorm}, 
is equivalent to prove
  that $\TopDualNorm{\TripleNorm{\dual}}{L} = \TopDualNorm{\TripleNorm{\dual}}{\IndexSet}$.
  On the one hand, we show that 
  \( \TripleNormDual{\dual} \leq \TopDualNorm{\TripleNorm{\dual}}{L} \).
  Indeed, since \( \dual=\dual_L \), we have 
  \( \TripleNormDual{\dual} = \TripleNormDual{\dual_L} = \TripleNorm{\dual_L}_{\star,L}
  \le \TopDualNorm{\TripleNorm{\dual}}{L} \),
  by the very definition~\eqref{eq:top_dual_norm}
  of the generalized local-top-$L$ dual norm \( \TopDualNorm{\TripleNorm{\cdot}}{L} \).
  On the other hand, we show that 
  \( \TopDualNorm{\TripleNorm{\dual}}{L} \leq \TripleNormDual{\dual} \).
  Indeed, replacing $\CoordinateNormDual{\TripleNorm{\dual}}{\cdot}$ by
  $\TopDualNorm{\TripleNorm{\dual}}{\cdot}$ in Equation~\eqref{eq:yineq} we have that 
  \(       \TopDualNorm{\TripleNorm{\dual}}{L} \leq 
  \TopDualNorm{\TripleNorm{\dual}}{\IndexSet} =  \TripleNormDual{\dual} \),
  where the last equality comes from replacing $\CoordinateNormDual{\TripleNorm{\dual}}{\IndexSet}$ by
  $\TopDualNorm{\TripleNorm{\dual}}{\IndexSet}$ in Equation~\eqref{eq:dual_coordinate_norm-d_equality}.
  Hence, we deduce that \( \TopDualNorm{\TripleNorm{\dual}}{L} = \TopNorm{\TripleNorm{\dual}}{\IndexSet}\).
  \medskip
  
  \noindent $\bullet$ 
  We prove Implication~\eqref{eq:dual_tn-L-not-included-in-IndexSubset}.
  Let $\IndexSubset \subset \IndexSet$ and $\dual \in \RR^d$ be such that
  \( \Support{\dual}=L \not\subset \IndexSubset \).

  \noindent -- 
  First, we prove an intermediate result: for any $\IndexSubset'$ such that
  $\IndexSubset'\subsetneq L=\Support{\dual}$, we have that \( \TopDualNorm{\TripleNorm{\dual}}{\IndexSubset'} <
  \TopDualNorm{\TripleNorm{\dual}}{L} \). Indeed, we have successively
  \begin{align*}
    \TopDualNorm{\TripleNorm{\dual}}{\IndexSubset'}  
    &=
      \sup_{\IndexSubsetLess \subset \IndexSubset'} 
      \TripleNormDual{\dual_{\IndexSubsetLess}}
      \tag{by definition~\eqref{eq:top_dual_norm} of the generalized local-top-$\IndexSubset'$ dual norm } 
    \\
    &< \TripleNormDual{\dual_{L}}
        \intertext{since $\IndexSubset' \subsetneq L$, and $L=\Support{\dual}$ for any $\IndexSubsetLess\subset \IndexSubset'$
      we have $\dual_{J}\not=\dual_L$, and thus
      \(\TripleNormDual{\dual_{\IndexSubsetLess}} < \TripleNormDual{\dual_{L}}\)
      by strict orthant-monotonicity of the dual norm and
      Item~\ref{it:SICS} in Proposition~\ref{pr:orthant-strictly_monotonic},
and thus \( \sup_{\IndexSubsetLess \subset \IndexSubset'} 
      \TripleNormDual{\dual_{\IndexSubsetLess}} <  \TripleNormDual{\dual_{L}}\),}
    &\le
      \sup_{ \IndexSubsetLess\subset L} \TripleNormDual{\dual_{\IndexSubsetLess}}
      \tag{obviously since $L \in \nset{\IndexSubsetLess}{\IndexSubsetLess\subset L}$}
    \\
    &=
      \TopDualNorm{\TripleNorm{\dual}}{L}
      \eqfinp 
      \tag{by definition~\eqref{eq:top_dual_norm} of the generalized local-top-$L$ dual norm } 
  \end{align*}
  \noindent -- 
Second, we apply the previous inequality with the subset
$\IndexSubset'=\IndexSubset\cap L$ and we obtain that
  \begin{equation}
    \TopDualNorm{\TripleNorm{\dual}}{\IndexSubset\cap L} 
< \TopDualNorm{\TripleNorm{\dual}}{L}
\eqfinv
    \label{ineq:IndexSubset_capL}
  \end{equation}
 because $\IndexSubset'=\IndexSubset\cap L\subsetneq L$
since \( L \not\subset \IndexSubset \) by assumption.

  \noindent -- 
Third, we prove that $ \TopDualNorm{\TripleNorm{\dual}}{\IndexSubset}=
  \TopDualNorm{\TripleNorm{\dual}}{\IndexSubset\cap L}$.
We have 
  \begin{align*}
    \TopDualNorm{\TripleNorm{\dual}}{\IndexSubset}
    &= 
     \sup_{ \IndexSubsetLess\subset \IndexSubset}
      \TripleNormDual{\dual_{\IndexSubsetLess}} 
\tag{by definition~\eqref{eq:top_dual_norm} of the generalized local-top-$\IndexSubset$
          dual norm}
 \\
    &= 
      \TripleNormDual{\dual_{\IndexSubsetLess^\sharp}}
      \intertext{where $\IndexSubsetLess^\sharp\subset \IndexSubset$ exists
          by definition~\eqref{eq:top_dual_norm} of the generalized local-top-$\IndexSubset$
          dual norm}
%
    &= 
      \TripleNormDual{\dual_{\IndexSubsetLess^\sharp\cap L}} 
      \tag{we have that $\dual = \dual_L$ since $L=\Support{\dual}$}
    \\
    &\le 
      \sup_{ \IndexSubsetLess\subset \IndexSubset\cap L}
      \TripleNormDual{\dual_{\IndexSubsetLess}} 
      \tag{$\IndexSubsetLess'=\IndexSubsetLess^\sharp\cap L$ is such that
      $\IndexSubsetLess' \subset \IndexSubset\cap L$
      since $\IndexSubsetLess^\sharp\subset \IndexSubset$}
    \\
    &= 
      \TopDualNorm{\TripleNorm{\dual}}{\IndexSubset\cap L}
      \tag{by definition~\eqref{eq:top_dual_norm} of the generalized local-top-$(\IndexSubset\cap L)$ dual norm }
    \\
    &\le \TopDualNorm{\TripleNorm{\dual}}{\IndexSubset}
\tag{by~\eqref{eq:the_generalized_local-top-_dual_seminorms_nondecreasing},
as proved below}
  \end{align*}
so that we obtain the equality
  $\TopDualNorm{\TripleNorm{\dual}}{\IndexSubset} 
  =\TopDualNorm{\TripleNorm{\dual}}{\IndexSubset\cap L}$.

  \noindent -- 
Fourth, 
  combining the above equality
  $\TopDualNorm{\TripleNorm{\dual}}{\IndexSubset} 
  =\TopDualNorm{\TripleNorm{\dual}}{\IndexSubset\cap L}$
with Equation~\eqref{ineq:IndexSubset_capL}, we obtain the inequality
  $\TopDualNorm{\TripleNorm{\dual}}{\IndexSubset}
  < \TopDualNorm{\TripleNorm{\dual}}{L}$.
Then, using again
  Equation~\eqref{eq:dual_coordinate-k_norm_=_generalized_top-k_norm} --- that is,
  $\TopDualNorm{\TripleNorm{\cdot}}{\IndexSubset} = \CoordinateNormDual{\TripleNorm{\cdot}}{\IndexSubset}$ for
  all $\IndexSubset\subset \IndexSet$ --- 
we finally obtain the inequality
\( \CoordinateNormDual{\TripleNorm{\dual}}{\IndexSubset}
< \CoordinateNormDual{\TripleNorm{\dual}}{L} \).
Thus, we have proved Implication~\eqref{eq:dual_tn-L-not-included-in-IndexSubset}. 
  \medskip

  \noindent $\bullet$
We are going to show the implication (that involves seminorms)
\begin{equation}
  \IndexSubsetLess \subset \IndexSubset \subset \IndexSet
\implies 
\TopDualNorm{\TripleNorm{\dual}}{\IndexSubsetLess}
\leq 
\TopDualNorm{\TripleNorm{\dual}}{\IndexSubset}
\eqsepv \forall \dual \in \RR^d
\eqfinp 
\label{eq:the_generalized_local-top-_dual_seminorms_nondecreasing}
\end{equation}
  Let $\IndexSubsetLess$ and $\IndexSubset$ be two subsets such that
  $\IndexSubsetLess \subset \IndexSubset \subset \IndexSet$.
  For any $\dual \in \RR^d$, we have
  \begin{align*}
    \TopDualNorm{\TripleNorm{\dual}}{\IndexSubsetLess}
    &=     \sup_{\IndexSubsetLess' \subset \IndexSubsetLess} \TripleNorm{\dual_{\IndexSubsetLess'}}_{\star,\IndexSubsetLess'}
      \tag{by definition~\eqref{eq:top_dual_norm} of 
\( \TopDualNorm{\TripleNorm{\dual}}{\IndexSubsetLess} \)}
    \\
    &= \sup_{\IndexSubsetLess' \subset \IndexSubsetLess}
      \TripleNorm{\np{\dual_{\IndexSubsetLess}}_{\IndexSubsetLess'}}_{\star,\IndexSubsetLess'}
      \tag{since $\IndexSubsetLess' \subset \IndexSubsetLess$}
    \\
    &= \TopDualNorm{\TripleNorm{\dual_{\IndexSubsetLess}}}{\IndexSubsetLess}
      \tag{again by definition~\eqref{eq:top_dual_norm} of 
\( \TopDualNorm{\TripleNorm{\dual_{\IndexSubsetLess}}}{\IndexSubsetLess} \)}
    \\
    &\le \TopDualNorm{\TripleNorm{\dual_{\IndexSubsetLess}}}{\IndexSubset}
      \tag{since $\dual_{\IndexSubsetLess} \in \FlatRR_{\IndexSubsetLess}$ 
      and using the implication~\eqref{eq:dual_coordinate_norm_inequalities} }
    \\
    &= \sup_{\IndexSubset' \subset \IndexSubset} 
\TripleNorm{\np{\dual_{\IndexSubsetLess}}_{\IndexSubset'}}_{\star,\IndexSubset}
        \tag{again by definition~\eqref{eq:top_dual_norm} of 
\( \TopDualNorm{\TripleNorm{\dual_{\IndexSubsetLess}}}{\IndexSubset} \)}
    \\
    &\le \TripleNorm{\dual_{\IndexSubset}}_{\star,\IndexSubset}
      \intertext{since $\TripleNorm{\np{\dual_{\IndexSubsetLess}}_{\IndexSubset'}}_{\star,\IndexSubset}=
      \TripleNorm{{\dual_{\IndexSubsetLess \cap \IndexSubset'}}}_{\star,\IndexSubset}
      \le \TripleNorm{{\dual_\IndexSubset}}_{\star,\IndexSubset}
      $
      by orthant-monotonicity of $\TripleNormDual{\cdot}$ 
      (Item~\ref{it:ICS} in Proposition~\ref{pr:orthant-monotonic}),
      using the fact that
      ${\IndexSubsetLess \cap \IndexSubset'}\subset \IndexSubset$ for all the
      considered $\IndexSubset'$, and the expression of 
      the restriction norm~\( \TripleNorm{{\cdot}}_{\star,\IndexSubset}\) 
      in Definition~\ref{de:K_norm} }
    &
      \le  \sup_{\IndexSubset' \subset \IndexSubset} \TripleNorm{{\dual}_{\IndexSubset'}}_{\star,\IndexSubset'}
      \tag{since $\IndexSubset\in \nset{\IndexSubset'}{\IndexSubset'\subset \IndexSubset}$}
    \\
    &= \TopDualNorm{\TripleNorm{\dual}}{\IndexSubset}
         \tag{again by definition~\eqref{eq:top_dual_norm} of 
\( \TopDualNorm{\TripleNorm{\dual}}{\IndexSubset} \)}
      \eqfinp
  \end{align*}
Thus, we have proven the Implication~\eqref{eq:the_generalized_local-top-_dual_seminorms_nondecreasing}.
  \medskip

  This ends the proof.
\end{proof}

\newcommand{\noopsort}[1]{} \ifx\undefined\allcaps\def\allcaps#1{#1}\fi

\end{document}